\documentclass[11pt]{amsart}
\usepackage{verbatim, latexsym, amssymb, amsmath,color,enumitem}
\usepackage{hyperref} 
\usepackage{epsfig}
\usepackage{bm}
\usepackage{mathrsfs}
\usepackage[margin=1.25in]{geometry}
\usepackage{subcaption}

\newtheorem{theorem}{Theorem}[section]
\newtheorem{corollary}[theorem]{Corollary}
\newtheorem{claim}[]{Claim}
\newtheorem{lemma}[theorem]{Lemma}

\newtheorem{proposition}[theorem]{Proposition}

\theoremstyle{definition}
\newtheorem{definition}[theorem]{Definition}

\newtheorem{remark}[theorem]{Remark}

\newtheorem*{acknowledgements}{Acknowledgements}

\numberwithin{equation}{section}

\newcommand{\V}{\mathcal{V}}
\newcommand{\RV}{\mathcal{RV}}
\newcommand{\IV}{\mathcal{IV}}
\newcommand{\R}{\mathbb{R}}
\newcommand{\N}{\mathbb{N}}
\newcommand{\mH}{\mathcal{H}}
\newcommand{\F}{\mathcal{F}}

\newcommand{\C}{\mathcal{C}}
\newcommand{\Pcal}{\mathcal{P}}

\newcommand{\Om}{\Omega}

\newcommand{\mB}{\mathbb{B}}
\newcommand{\mAn}{\mathbb{A}}
\newcommand{\mZ}{\mathbb{Z}}
\newcommand{\Z}{\mathcal{Z}}

\newcommand{\f}{\mathbf{f}}

\newcommand{\M}{\mathbf{M}}

\newcommand{\mF}{\mathbf{F}}
\newcommand{\mI}{\mathbf{I}}
\newcommand{\mR}{\mathcal{R}}
\newcommand{\bmu}{\boldsymbol \mu}
\newcommand{\btau}{\boldsymbol \tau}
\newcommand{\bleta}{\boldsymbol \eta}

\newcommand{\md}{\mathbf{d}}

\newcommand{\tM}{\widetilde{M}}

\newcommand{\Area}{\textrm{Area}}

\newcommand{\VarTan}{\operatorname{VarTan}}

\newcommand{\spt}{\operatorname{spt}}
\newcommand{\dist}{\operatorname{dist}}
\newcommand{\Div}{\operatorname{div}}

\newcommand{\dom}{\operatorname{dom}}

\newcommand{\interior}{\operatorname{int}}
\newcommand{\Ric}{\operatorname{Ric}}
\newcommand{\Clos}{\operatorname{Clos}}

\newcommand{\rom}[1]{\expandafter\romannumeral #1}

\newcommand{\mfX}{\mathfrak{X}}
\newcommand{\wD}{d^{\sphericalangle}}
\newcommand{\wB}{B^{\sphericalangle}}
\newcommand{\wS}{S^{\sphericalangle}}
\newcommand{\wA}{A^{\sphericalangle}}

\newcommand{\bnu}{\boldsymbol \nu}
\newcommand{\mfa}{\mathfrak{a}}
\newcommand{\tx}{\tilde{x}}

\newcommand{\bd}{\partial}

\newcommand{\clos}{\operatorname{Clos}}
\newcommand{\laa}{\langle}
\newcommand{\raa}{\rangle}

\newcommand{\grad}{\nabla}

\newcommand{\eps}{\varepsilon}
\newcommand{\area}{\operatorname{area}}

\title[Min-max theory for FBMHs in locally wedge-shaped manifolds]{Min-max theory for free boundary minimal hypersurfaces in locally wedge-shaped manifolds}

\author[Liam Mazurowski]{Liam Mazurowski}
\address{Cornell University, Department of Mathematics, Ithaca, New York 14850}
\email{lmm334@cornell.edu}

\author{Tongrui Wang}
\address{Institute for Theoretical Sciences, Westlake Institute for Advanced Study, Westlake University, Hangzhou, Zhejiang, 310024, China}
\email{wangtongrui@westlake.edu.cn}

\begin{document}
\maketitle

\begin{abstract}
	We develop a min-max theory for the area functional in the class of locally wedge-shaped manifolds. Roughly speaking, a locally wedge-shaped manifold is a Riemannian manifold that is allowed to have both boundary and certain types of edges. Fix a dimension $3 \le n+1 \le 6$. As our main theorem, we prove that every compact locally wedge-shaped manifold $M^{n+1}$ with acute wedge angles contains a locally wedge-shaped free boundary minimal hypersurface $\Sigma^n$ which is smooth in its interior and on its faces and is $C^{2,\alpha}$ up to and including its edge. We can also handle the case of 90 degree wedge angles under an additional assumption.
\end{abstract}

\section{Introduction}

Minimal hypersurfaces are critical points of the area functional.  It is a classical problem in differential geometry to understand the space of minimal hypersurfaces in a given Riemannian manifold (possibly satisfying certain boundary conditions). At the most basic level, one can ask whether this space is non-empty, i.e. do minimal hypersurfaces always exist? This problem can be posed in a variety different contexts. 

\subsection{Minimal surfaces spanning a prescribed curve} The Plateau problem asks whether every Jordan curve in $\R^3$ can be spanned by a minimal disk. Around 1930, Douglas \cite{douglas1931solution} and Rado \cite{rado1930plateau} independently gave an affirmative answer to the Plateau problem: every Jordan curve in $\R^3$ can be spanned by an energy minimizing conformal map from a disk. This was perhaps the first general existence theorem for minimal surfaces. 

\subsection{Closed minimal hypersurfaces} Given a closed manifold $(M^{n+1},g)$, one can inquire about the existence of closed, embedded minimal hypersurfaces in $M$. When $n+1=2$, this reduces to the problem of finding closed geodesics. It is well-known that there is a length-minimizing closed geodesic in each non-trivial free homotopy class.  When $n+1 > 2$ and $H_n(M)$ is non-trivial, the methods of geometric measure theory give the existence of an area minimizer in each homology class. These area minimizers are smooth and embedded away from a singular set of codimension 7. In particular, they are completely smooth when $3 \le n+1 \le 7$. 

Returning again to the case of surfaces, length minimizing geodesics may not exist when $\pi_1(M)$ is trivial.  However, in this case, Birkhoff \cite{birkhoff1927dynamical} used a min-max argument and a curve shortening procedure to show that there exist closed geodesics of saddle type.  Thus every closed surface contains a closed geodesic. Likewise, in higher dimensions, area minimizing minimal hypersurfaces may not exist when $H_n(M)$ is trivial. However, one may hope to produce saddle-type minimal hypersurfaces via a min-max argument. 

Almgren \cite{almgren1965theory} used the techniques of geometric measure theory to develop a variational calculus in the large for the area functional, which is very general and applies in any dimension and codimension.  Almgren's work implies in particular that every closed manifold $(M^{n+1},g)$ contains a stationary codimension 1 varifold. This is a weak notion of minimal hypersurfaces. In the early 1980s, Pitts \cite{pitts2014existence} proved the regularity of Almgren's stationary varifold in ambient dimension $3\le n+1 \le 6$. Schoen and Simon \cite{schoen1981regularity} then showed that regularity in fact holds in dimensions $3 \le n+1 \le 7$, and more generally, that the varifold is smooth away from a set of codimension 7 in all dimensions. Thus the combined work of Almgren, Pitts, and Schoen-Simon implies that every closed manifold $(M^{n+1},g)$ with $3\le n+1\le 7$ contains a smooth, embedded minimal hypersurface. 

In the closed case, refinements of the Almgren-Pitts min-max theory have led to a very detailed understanding of the Morse theory of the area functional.  In the early 1980s, Yau \cite{yau1982seminar} conjectured based on Morse theoretic considerations that  every closed 3-manifold should in fact contain infinitely many distinct minimal surfaces. To each closed manifold, Gromov \cite{gromov2006dimension} associated a sequence of numbers $\{\omega_p\}_{p\in \N}$ called the {\em volume spectrum}. The numbers $\omega_p$ are defined via a min-max formula and act as a sort of non-linear analog to the spectrum of the Laplacian. Marques and Neves used the Almgren-Pitts min-max theory to show that, for each $p$, there is a collection of smooth, disjoint, embedded minimal hypersurfaces $\Sigma_1,\hdots,\Sigma_k \subset M$ and a collection of positive integer multiplicities $m_1,\hdots,m_k\in \N$ such that 
\[
\omega_p = \sum_{i=1}^k m_i\area(\Sigma_i). 
\]
In \cite{marques2016morse}, they moreover proved that $\sum_{i=1}^k \text{index}(\Sigma_i) \le p$. Marques and Neves conjectured that, for generic metrics, the multiplicities $m_i$ can all be taken equal to 1. This Multiplicity One Conjecture was later confirmed by Zhou \cite{zhou2020multiplicity}. The work of Marques-Neves \cite{marques2021morse} combined with Zhou's resolution of the multiplicity one conjecture yields the following: given a generic metric $g$ on a closed manifold $M^{n+1}$ with $3\le n+1 \le 7$, there exists a sequence $\{\Sigma_p\}_{p\in \N}$ of smooth, closed, embedded (possibly disconnected) minimal hypersurfaces such that $\area(\Sigma_p) = \omega_p$ and $\text{index}(\Sigma_p) = p$. This resolves Yau's conjecture in the generic case with very strong information about the sequence of infinitely many minimal hypersurfaces. Additionally, Yau's conjecture had earlier been proven in the generic case by Irie-Marques-Neves \cite{irie2018density}. We also mention that Marques-Neves-Song proved \cite{marques2019equidistribution} that generically a (possibly different) sequence of minimal hypersurfaces becomes equidistributed in $M$. Moreover, Song \cite{song2023existence} proved that Yau's conjecture holds for arbitrary metrics. 

\subsection{Free boundary minimal hypersurfaces} 

Given a set $A$, a free boundary minimal hypersurface is a minimal hypersurface whose boundary is constrained to lie on $A$ but is otherwise free to vary.  Courant \cite{courant2005dirichlet} was among the first to investigate the existence of free boundary minimal hypersurfaces. The typical problem he considered is the following: given a torus $T$ embedded in $\R^3$, find a free boundary minimal disk with boundary on $T$ which ``spans one of the holes in the torus.'' Like Douglas and Rado, Courant proved the existence of a solution in the form of a conformal energy minimizing map. Later Str\"uwe \cite{struwe1984free} proved the existence of saddle type solutions to the same problem. Gr\"uter and Jost \cite{gruter1986allard} used methods of geometric measure theory to show that any convex domain in $\R^3$ contains an unstable, embedded free boundary minimal disk. 

Given a smooth Riemannian manifold with boundary $(M,\bd M)$, one can look for free boundary minimal hypersurfaces with boundary constrained to lie on $\bd M$. 
In this setting, Almgren's variational calculus still produces stationary varifolds with free boundary. In \cite{li2021min}, Li and Zhou  developed a refined min-max procedure and proved the regularity of the resulting min-max free boundary minimal surfaces. In particular, their work shows that every $(M^{n+1},\bd M^{n})$ with $3 \le n+1\le 7$ contains a smooth, free boundary minimal hypersurface $(\Sigma^n, \bd \Sigma^{n-1})$.  The solution $\Sigma$ is embedded, although interior points of $\Sigma$ may touch the boundary of $M$. The result of Li and Zhou is significant in that no curvature assumption is imposed on $\bd M$. 

\subsection{Minimal surfaces in more singular ambient spaces}

One can ask whether minimal surfaces exist in spaces with edges, corners, and other singular features.  In a series of papers in the 1990s, Hildebrandt and Sauvigny \cite{hildebrandt1997minimal1}\cite{hildebrandt1997minimal2}\cite{hildebrandt1999minimal3}\cite{hildebrandt1999minimal4} studied minimal surfaces inside a wedge in $\R^3$. Among other things, they found that the properties of minimal surfaces in a wedge depend crucially on the wedge angle $\theta$. 

Recently there has been renewed interest in the existence of minimal surfaces and other related classes of surfaces in polyhedral-type domains. For example, capillary surfaces in polyhedra were a key ingredient in Li's proof of Gromov's dihedral rigidity conjecture \cite{li2020polyhedron}.  Also Edelen and Li \cite{edelen2022regularity} have proven an Allard-type regularity theorem for minimal hypersurfaces in locally polyhedral spaces. Their work in particular implies partial regularity of area minimizers in locally polyhedral spaces.  The second author's work on equivariant min-max theory can also be regarded as a min-max theory inside the possibly singular orbit space \cite{wang2022min}\cite{wang2023min}. 

\subsection{Locally wedge-shaped manifolds} The goal of this paper is to establish a min-max theory for the area functional in a class of spaces we call {\it locally wedge-shaped manifolds}. Loosely speaking, a locally wedge-shaped manifold $M^{n+1}$ is built up of three strata: an $n+1$ dimensional set of {\em interior points} near which $M$ looks locally like $\R^{n+1}$, an $n$ dimensional set of {\em face points} near which $M$ looks locally like a half-space in $\R^{n+1}$, and an $n-1$ dimensional set of {\em edge points} near which $M$ looks locally like a wedge in $\R^{n+1}$. Denote the interior points, face points, and edge points of $M$ by $\text{int}(M)$, $\bd^F M$, and $\bd^E M$, respectively. Near each edge point $p\in \bd^E M$, let $\theta(p)$ denote the wedge angle at $p$. We allow $\theta(p)$ to vary with $p$. Later, for the purpose of the min-max theory, we will insist that $\theta(p)\le \frac{\pi}{2}$ for all points $p\in \bd^E M$. This condition is important later for establishing regularity.  
\begin{figure}[h]
\centering
\includegraphics[width=2in]{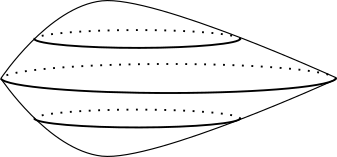}
\caption{A Locally Wedge-Shaped Manifold}
\label{figure:lwsm}
\end{figure}
Figure \ref{figure:lwsm} illustrates a simple example of a locally wedge-shaped manifold which is shaped like a mussel.   There is a single edge running along the center, an upper face, a lower face, and a 3-dimensional interior enclosed by the faces. Note that the wedge angle varies along the edge.  Precise definitions are given in Section \ref{Sec: preliminary}. 

Let $M$ be a locally wedge-shaped manifold and let $\Sigma^n \subset M^{n+1}$ be a locally wedge-shaped hypersurface. We say that $\Sigma$ is {\it properly embedded} in $M$ if $\Sigma$ is embedded in $M$ and $\text{int}(\Sigma)\subset \text{int}(M)$, $\bd^F\Sigma \subset \bd^F M$, and $\bd^E \Sigma \subset \bd^E M$.  Since this is not a closed condition, we will need to work more generally with {\it almost properly embedded} locally wedge-shaped hypersurfaces.  These are embedded locally wedge-shaped hypersurfaces $\Sigma^n \subset M^{n+1}$ for which $\bd^F\Sigma \subset \bd^F M\cup \bd^E M$ and $\bd^E\Sigma\subset \bd^EM$. However, interior points of $\Sigma$ are allowed to touch the faces of $M$, and face points of $\Sigma$ are allowed to touch the edge of $M$. 

\subsection{Main results} In Section \ref{Sec: preliminary} we define a notion of stationary locally wedge-shaped hypersurfaces with free boundary in $M$. An almost properly embedded locally wedge-shaped hypersurface $\Sigma^n \subset M^{n+1}$ is called {\em stationary with free boundary} if the first variation of area vanishes for all vector fields in a certain class $\mathfrak X(M,\Sigma)$ which depends on both $M$ and $\Sigma$.  Loosely speaking, the flow of a vector field in $\mathfrak X(M,\Sigma)$ must map $M$ into $M$, but is allowed to push interior points of $\Sigma$ lying on a face of $M$ into the interior of $M$, and is allowed to push face points of $\Sigma$ lying on the edge of $M$ into a face of $M$. Then the stationarity turns out to be equivalent to the following two conditions:  
\begin{itemize}
\item the mean curvature of $\Sigma$ vanishes at each point $p\in \text{int}(\Sigma)$, and
\item at each point $p\in \bd^F \Sigma$, the outward unit co-normal $\eta(p)$ meets at least one face of $M$ orthogonally. 
\end{itemize}
Note that the second condition ensures that classical free boundary minimal hypersurfaces (without edges) are stationary with free boundary and that limits of such hypersurfaces are also stationary with free boundary. 
\begin{figure}[h]
\centering
\begin{subfigure}{0.4\linewidth}
\centering
\includegraphics[width=2in]{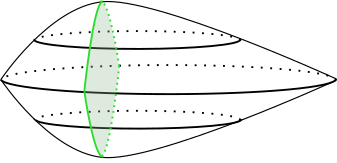} 
\caption{A free boundary minimal surface}
\end{subfigure}
\hspace{1cm}
\begin{subfigure}{0.3\linewidth}
\centering
\includegraphics[width=2in]{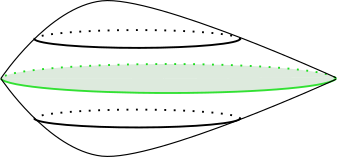}
\caption{A non-example}
\end{subfigure}
\caption{}
\label{figure:examples}
\end{figure}
Figure \ref{figure:examples}(A) shows a free boundary minimal surface, which is properly embedded with two edge points. The surface illustrated in Figure \ref{figure:examples}(B) is almost properly embedded but not properly embedded since face points of $\Sigma$ touch the edge of $M$. This surface is not stationary with free boundary according to our definition since the outward unit co-normal to $\Sigma$ is not perpendicular to either face of $M$. 

We note that the regularity of minimal hypersurfaces in general locally polyhedral spaces seems to be a delicate issue.  Indeed, by analogy with the Neumann problem in a Euclidean wedge, we expect that the optimal regularity of a minimal hypersurface at an edge point is only $C^{1,\alpha}$ when the wedge angle $\theta$ exceeds $\pi/2$. Our assumption that the wedge angle satisfies $\theta(p) \le \pi/2$ is imposed to guarantee a priori $C^{2,\alpha}$ regularity up to and including the edge. Curiously, although one might expect the best regularity when $\theta = \pi/2$, we find that this case requires an additional technical assumption in order for our arguments to carry through. 

Our main result is a min-max theorem which proves the existence of almost properly embedded free boundary minimal hypersurfaces in locally wedge-shaped manifolds. 

\begin{theorem}
\label{theorem:main} 
Let $M^{n+1}$ be a compact locally wedge-shaped manifold of dimension $3\le n+1\le 6$, and assume that $M$ is isometrically embedded in some Euclidean space $\R^L$.  Assume that for each connected component $\bd^E_i M$ of $\bd^E M$ either 
\begin{itemize}
\item $\theta(p) < \frac{\pi}{2}$ for all $p\in \bd^E_i M$, or
\item $\theta(p) = \frac{\pi}{2}$ for all $p\in \bd^E_i M$ and moreover 
\[
A_{\bd^+ M}(\eta_+, v) = A_{\bd^- M}(\eta_-, v) \text{ for all } p\in \bd^E_i M,\ v\in T_p(\bd^E M),
\]
where $\bd^{\pm} M$ denote the two faces of $M$ meeting along $\bd^E_i M$, and $A_{\bd^{\pm}M}$ is the second fundamental form of $\bd^{\pm}M$, and $\eta_{\pm}$ is the outward unit co-normal of $\bd^{\pm}M$ along $\bd^E M$.
\end{itemize}
Then there exists an almost properly embedded free boundary minimal hypersurface $\Sigma^n \subset M^{n+1}$. The hypersurface $\Sigma$ is smooth in its interior and along its faces, and it is $C^{2,\alpha}$ up to and including its edge. 
\end{theorem}

\begin{remark}
We cannot handle the case of ambient dimension $n+1=7$ since we do not know whether an analog of the Schoen-Simon curvature estimates \cite{schoen1981regularity} holds in our setting. Our arguments rely on the curvature estimates in locally wedge-shaped manifolds from \cite{mazurowski2023curvature}, which are based on the techniques of Schoen-Simon-Yau \cite{schoen1975curvature}. 
\end{remark}

\subsection{Further Directions}

The main result of this paper shows that there always exists at least one free boundary minimal hypersurface in a locally wedge-shaped manifold with suitable wedge angles. What more can be said about the structure of this space? As mentioned above, in the closed case, it is known that there always exist infinitely many minimal surfaces. Moreover, generically the union of all minimal surfaces is dense and there even exist sequences of minimal surfaces that become equidistributed in a suitable sense. 

There has also been recent progress toward the analogous results in manifolds with (smooth) boundary. For example, Guang-Li-Wang-Zhou \cite{guang2021min} proved that for a generic metric on $(M,\bd M)$ the union of all free boundary minimal hypersurfaces in $M$ is dense in $M$. In particular, this implies that generically there exist infinitely many free boundary minimal hypersurfaces. 
Wang \cite{wang2022existence} then showed that for a generic metric on $(M,\bd M)$ there exists a free boundary minimal hypersurface with non-empty boundary. In fact, he proved the much stronger result that generically there are infinitely many such surfaces in $(M,\bd M)$ and the union of their boundaries is dense in $\bd M$.   Finally, we note that Wang \cite{wang2020existence} has proved the existence of infinitely many free boundary minimal hypersurfaces in $(M,\bd M)$ equipped with an arbitrary (not necessarily generic) metric. 

It is natural to ask if analogous theorems hold in the setting of locally wedge-shaped manifolds. In particular, we record the following questions:
\begin{itemize}
\item Does a generic locally wedge-shaped manifold contain infinitely many free boundary minimal hypersurfaces? 
\item Is the union of all free boundary minimal hypersurfaces dense in a generic locally wedge-shaped manifold? 
\item Does a generic locally wedge-shaped manifold contain a free boundary minimal hypersurface with non-empty edge? 
\end{itemize}
A non-trivial challenge in addressing the above questions is the fact that White's bumpy metric theorem is not known for locally wedge-shaped manifolds. Answers to these problems would help shed light on whether smoothness plays some essential role in the above theorems, or if more course Lipschitz geometry suffices. 

The edges in a locally wedge-shaped manifold are a relatively mild type of singularity. 
Of course, it is also interesting to determine if min-max minimal surfaces exist in spaces with other kinds of singular features. Is it possible to perform min-max theory in manifolds with corners or locally polyhedral manifolds? What about closed manifolds with conical singularities? It would also be very interesting to know if the condition on the wedge angle in Theorem \ref{theorem:main} can be relaxed by assuming the metric is generic in a suitable sense. 

\subsection{Outline} In the remainder of the introduction, we outline the structure of the paper and sketch the proof of Theorem \ref{theorem:main}. In Section \ref{Sec: preliminary}, we define locally wedge-shaped manifolds. We also define locally wedge-shaped hypersurfaces and introduce the stationary with free boundary condition. Finally, we recall some further preliminaries regarding the second variation and a key compactness theorem for stable free boundary minimal hypersurfaces in locally wedge-shaped manifolds \cite{mazurowski2023curvature}. 

In Section \ref{Sec: variations in wedge manifold}, we begin by recalling some preliminary results from geometric measure theory. As in \cite{li2021min}, we work with the space $\Z_n(M,\bd M)$ of equivalence classes of relative cycles, where in our case $\bd M = \bd^F M \cup \bd^E M$. Each element $\tau \in \Z_n(M,\bd M)$ possesses a unique canonical representative $T\in Z_n(M,\bd M)$ which ``forgets'' the part lying on  $\bd M$. The key results in \cite{li2021min} about the space $\Z_n(M,\bd M)$ continue to hold in the setting of locally wedge-shaped manifolds. We conclude Section \ref{Sec: variations in wedge manifold} by introducing the almost minimizing property. Loosely speaking, an element $\tau \in \Z_n(M,\bd M)$ is $(\eps,\delta)$-almost minimizing with free boundary if any deformation that decreases the mass of $\tau$ by at least $\eps$ must first increase the mass by at least $\delta$. This type of almost minimizing property is crucial to obtaining regularity in the Almgren-Pitts min-max framework. 

In Section \ref{sec:min-max}, we perform the min-max procedure. We work with discrete mappings and discrete homotopies as in \cite{pitts2014existence}. We then follow the Almgren-Pitts min-max procedure with the modifications of Li-Zhou \cite{li2021min}. The overall argument in this section remains relatively unchanged from \cite{li2021min}. In the end, we obtain the existence of a varifold $V\in \mathcal V_n(M)$ which is stationary with free boundary and which is almost minimizing with free boundary in small annuli. 

Finally in Section \ref{Sec: min-max regualrity}, we prove the regularity of the min-max varifold $V$. To begin, we show that the almost minimizing property can be used to construct {\em replacements} for $V$ in annuli. 
To construct replacements, we solve a sequence of constrained minimization problems. The solution $T_i^*$ to each constrained minimization problem is stable and locally area minimizing.
Crucially, our assumption of an acute wedge angle guarantees that such a hypersurface is smooth in its interior and along its faces and is $C^{2,\alpha}$ up to and including the edge.  This is proven in Appendix \ref{Sec: elliptic regularity}.  We can therefore appeal to the curvature and compactness theorems in \cite{mazurowski2023curvature} to show that a subsequence of the $T_i^*$'s converges to a replacement. 

{Next, we use the good replacement property to classify the tangent cones of $V$ and show $V$ is rectifiable. 
As in \cite{li2021min}, we cannot apply the rectifiability theorem \cite[42.4]{simon1983lectures} in wedge domains since we do not know whether the stationarity with free boundary implies locally bounded first variation in $\R^L$. 
Nevertheless, noting that $\mH^n(\bd^EM)=0$ and that $V$ is regular away from the edge, it is sufficient to show $\|C\|(\bd^E T_p M)=\|V\|(\bd^E M) = 0$ to obtain the rectifiability of $C\in\VarTan(V,p)$ and $V$.} 

Now fix a point $p\in \spt \|V\|$. We can assume $p$ is an edge point, as otherwise regularity is already known. To prove regularity at $p$, we first show that successive replacements on overlapping annuli centered at $p$ can be  nicely glued together. Here it is important to verify that the support of $\|V\|$ intersects the inner annulus. This issue is handled in \cite{li2021min} by using a maximum principle for stationary varifolds with free boundary \cite{li2021maximum}. This maximum principle is proven using a foliation constructed from the Fermi distance function to a point on the boundary.  Unfortunately, in a locally wedge-shaped manifold $M$ there is no natural Fermi distance function near a point $p\in \bd^E M$. Thus in Appendix \ref{Sec: Fermi distance} we construct by hand a Fermi-type distance function near a point $p\in \bd^E M$. This Fermi-type distance function has convex level sets that meet both faces of $M$ orthogonally. In Appendix \ref{Sec: maximum principle for varifolds}, we use the Fermi-type distance function to prove a maximum principle for stationary varifolds with free boundary in a locally wedge-shaped manifold. 

After showing successive replacements glue nicely, we then let the inner radius of the inner annulus tend to zero to obtain a replacement in a punctured neighborhood of $p$. We use the Allard-type theorem of Edelen and Li \cite{edelen2022regularity} to show that the singularity of the replacement at $p$ can be removed. Finally, we use a unique continuation argument to show that $V$ coincides with its replacement in a neighborhood of $p$, proving the regularity. 
    
\begin{acknowledgements}
	The authors would like to thank Professor Xin Zhou for suggesting the problem and for many helpful discussions. The second author also would like to thank Professor Gang Tian for his constant encouragement, thank Gaoming Wang for helpful discussions, and thank Professor Xin Zhou and Cornell University for their hospitality. 
	The second author is partially supported by China Postdoctoral Science Foundation 2022M722844.
\end{acknowledgements}

\section{Preliminaries}\label{Sec: preliminary}

In this section, we introduce some notation and definitions that will be used throughout the remainder of this paper. 
Firstly, we collect some notation in Euclidean space together with some terminology for Euclidean wedge domains. 
Then, after introducing the related concept of locally wedge-shaped manifolds, we present the definition and some useful theorems for stationary varifolds with free boundary.
Finally, we define almost properly embedded locally wedge-shaped submanifolds and discuss some associated notions for variation of area. 

\subsection{Notation for Euclidean wedge domains}\label{Subsec: Euclidean wedge}

To begin with, we collect the following notation in Euclidean spaces ($m=1,2,\dots$): 
\begin{itemize}
	\item $0^m$: the origin in $\R^m$;
    \item $\mB^m_r(p)$: the open ball in $\R^m$ of radius $r$ centered at $p$;
	\item $\mAn^m_{s,t}(p)$: the open annulus $\mB^m_t \setminus \Clos(\mB^m_s)$ in $\R^m$;
	\item $\bmu_r$: the homothety map $x\mapsto r\cdot x$;
	\item $\btau_p$: the translation map $x\mapsto x - p$;
    \item $\bleta_{p,r}$: the composition $\bmu_{r^{-1}}\circ\btau_p $;
	\item $\Clos(A)$: the closure of a subset $A$;
	\item $\mH^k$: the $k$-dimensional Hausdorff measure;
	\item $\mfX(\R^m)$: the space of smooth vector fields on $\R^m$.
\end{itemize}

Next, we introduce some notations for Euclidean wedge domains. 
\begin{definition}\label{Def: Euclidean wedge domain}
    Given $m\geq 2$ and two closed half-spaces $H^m_{\pm}$ in $\R^{m}$, the intersection 
    \[\Omega^m := H^m_{+}\cap H^m_-\]
    is said to be an $m$-dimensional {\em wedge domain} if it has non-empty interior. 
    In particular, we say $\Omega^m$ is a {\em trivial} wedge domain if $\Omega^m=H^m_+=H^m_-$. 
\end{definition}

Given a wedge domain $\Omega^m$, it is easy to see that there exists $\theta\in (0,\pi]$, called the {\em wedge angle} of $\Omega^m$, so that after a rotation, 
\begin{align}\label{Eq: standard wedge}
	 \Omega^m &= \Omega_\theta^2 \times \R^{m-2} 
	\\  &= \Clos\big(\big\{(x_1,\dots,x_m)\in \R^m: x_1> 0, x_2\in \left(\tan\left(-\theta/2\right) x_1, ~\tan\left(\theta/2\right) x_1\right) \big\} \big) . \nonumber
\end{align}
Unless otherwise specified, we always assume $\Omega^m$ is in this standard form, and 
\[H_{\pm}^m = \big\{(x_1,x_2,\dots,x_m)\in \R^m: \pm x_2\leq \pm \tan(\pm \theta/2 )x_1 \big\}.\]
Note also that $\Omega^m_\pi = \{(x_1,\dots,x_m)\in \R^m: x_1\geq 0\}$ represents the trivial wedge domain.

\begin{definition}[Stratification]\label{Def: stratification of wedge}
	For an $m$-dimensional wedge domain $\Omega = \Omega^m_\theta$ with wedge angle $\theta\in (0,\pi]$ in the form of (\ref{Eq: standard wedge}), define
	$$ \partial_m\Omega := \Omega, 
	\quad \partial_{m-1}\Omega := \partial \Omega, 
	\quad \partial_{m-2}\Omega:= \left\{ \begin{array}{ll}
		 0^2\times \R^{m-2}, ~&\theta \in (0,\pi),
		 \\ \emptyset ,~ &\theta=\pi.
	\end{array} \right.$$
    as the stratification $\partial_{m-2}\Omega \subset \partial_{m-1}\Omega \subset \partial_m\Omega$ of $\Omega$. 
	Additionally, define
	\begin{itemize}
		\item $\interior(\Omega) := \partial_m\Omega\setminus\partial_{m-1}\Omega$ to be the {\em interior} of $\Omega$;
		\item $\partial^F\Omega := \partial_{m-1}\Omega\setminus\partial_{m-2}\Omega$ to be the {\em face} of $\Omega$;
		\item $\partial^E\Omega := \partial_{m-2}\Omega$ to be the {\em edge} of $\Omega$. 
	\end{itemize}
\end{definition}

It should be noted that the edge $\partial^E\Omega^m_\theta$ is empty if and only if $\Omega^m_\theta$ is a trivial wedge domain, i.e. $\theta=\pi$. 
Additionally, we have 
\[ T_x\Omega^m = \left\{ \begin{array}{ll}
		 	\R^m ~& x \in \interior(\Omega^m) 
		 \\ H^m_{\pm} ~ & x \in \partial^F \Omega^m \cap H^m_{\pm}
		 \\ \Omega^m ~& x\in \partial^E\Omega^m
	\end{array} \right. , \]
where $T_x\Omega$ is the tangent space of $\Omega$ at $x$.

Now, we introduce the horizontal and vertical hyperplanes in wedge domains.

\begin{definition}\label{Def: vertical/horizontal hyperplane in wedge}
    Let $\Omega^m_\theta$ be a non-trivial wedge domain ($\theta<\pi$) in the standard form (\ref{Eq: standard wedge}). 
    Then, a hyperplane $P\subset \R^m$ is said to be a {\em horizontal hyperplane} of $\Omega$, if there exists an $(m-3)$-subspace $W^{m-3}\subset\R^{m-2}$ so that 
    \[P=\R^2\times W^{m-3}.\]
    Denote by $\mathcal{P}_{\Omega}$ the set of horizontal hyperplanes of $\Omega$. 
    Additionally, given $ \alpha\in [-\theta/2, \theta/2] $,  
    \[P_\alpha^{+} := \left\{(x_1,x_2,\dots,x_m)\in \R^m: x_1\geq 0, x_2=\tan\left(\alpha\right)x_1 \right\}\]
    is called the {\em vertical half hyperplane} of angle $\alpha$ in $\Om^m_\theta$. 
\end{definition}

For simplicity, we also denote and assume 
\[\partial^{\pm} \Omega := \bd\Omega\cap H_{\pm} = P^+_{\pm\frac{\theta}{2}} \]
to be the two smooth pieces of $\bd\Omega$, for a non-trivial wedge domain $\Omega=\Omega^m_\theta$.

\subsection{Notations for locally wedge-shaped manifolds}\label{Subsec: local wedge manifold}

Now, let us consider the manifolds that are locally modeled by wedge domains. 

\begin{definition}[Locally wedge-shaped manifolds]\label{Def: wedge manifold}
	Let $M^m\subset \R^L$ be a (compact) $m$-dimensional manifold with (possibly empty) boundary $\partial M$, and $m,L\in \{2,3,\dots\}$. 
	Then $M$ is said to be a {\em locally wedge-shaped $m$-manifold} if for any $p\in M$, there exist $R=R(p)>0$ and a diffeomorphism $\phi = \phi_p: \mB^L_{R}(0)\to \mB^L_{R}(p)$ so that 
	\begin{itemize}
		\item[(i)] $\phi(0) = p$, and the tangent map $(D\phi)_0\in O(L)$ is an orthogonal transformation;
		\item[(ii)] $\phi\big((\Omega \times 0^{L-m} ) \cap \mB^L_R(0)\big) = M\cap \mB^L_R(p)$, where
			\[ \Omega = \Omega(p) = \left\{ \begin{array}{ll}
		 		\R^m, ~&p\in \interior(M) ,
			 	\\ \mbox{an $m$-dimensional wedge domain } \Omega^m_\theta(p) ,~ &p\in \partial M,
				\end{array} \right. \]
			for some $\theta = \theta(p)\in (0, \pi]$. 
	\end{itemize}
	We call $(\phi, \mB_R^L(p), \Omega)$ a {\em local model} of $M$ around $p$, and call $\theta$ the {\em wedge angle} of $M$ at $p\in \partial M$. 
    Additionally, given $l\in\mZ_+$ and $\alpha\in (0,1]$, if for any $p\in\bd M$ with $\theta(p)<\pi$, $\phi_p$ is globally a $C^{l,\alpha}$-diffeomorphism and is a $C^\infty$-diffeomorphism in $ (\R^{m}\times 0^{L-m})\setminus \bd^E\Omega(p)$, then we say $M$ is a {\em $C^{l,\alpha}$-to-edge} locally wedge-shaped $m$-manifold.
\end{definition}

Let $M^m\subset \R^L$ be a locally wedge-shaped manifold. 
After equipping $M$ with a Riemannian metric $g_{_M}$ induced from the Euclidean metric $g_0$ in $\R^L$,  it follows from Definition \ref{Def: wedge manifold}(i) that $\theta(p)$ is the {\em intrinsic} wedge angle of $M$ at $p\in\partial M$, and this is independent of the choice of local model. 
Hence, unless otherwise specified, we always assume $g_{_M} = g_0\llcorner M$.

Additionally, $M$ can also be locally extended by the local model $(\phi, \mB_R^L(p), \Omega)$ around any $p\in\partial M$. 
Specifically, let $\R^m\supsetneq \Omega$, and $\tM_{p,R} \subset\mB^L_R(p)$ be given by
\begin{equation}\label{Eq: local extension}
	\tM_{p,R} := \phi\big((\R^m \times 0^{L-m} ) \cap \mB^L_R(0)\big), 
\end{equation}
which is an embedded $m$-manifold in $\mB^L_R(p)$. 
Then, $M_{p,R} := M \cap \mB^L_R(p)$ is a wedge-shaped domain in $\tM_{p,R}$. 
Furthermore, denote by
\begin{equation}\label{Eq: two pieces of face}
	\partial^{\pm} M_{p,R} := \phi\big((\partial^{\pm}\Omega \times 0^{L-m} ) \cap \mB^L_R(0)\big)
\end{equation}
the two smooth pieces of $\bd M \cap \mB^L_R(p)$. 
Then
\begin{equation}\label{Eq: local extension across face}
	\tM_{p,R}^\pm := \phi\big((H^m_{\pm} \times 0^{L-m} ) \cap \mB^L_R(0)\big) \quad{\rm and}\quad  \bd \tM_{p,R}^\pm := \phi\big((\partial H_{\pm} \times 0^{L-m} ) \cap \mB^L_R(0)\big)
\end{equation}
are the extensions of $M_{p,R}$ across $\partial^{\mp} M_{p,R}$ and the extensions of $\partial^{\pm} M_{p,R}$ respectively (see Figure \ref{fig:local-extensions}). 

\begin{figure}[h]
\centering
\begin{subfigure}{0.4\linewidth}
\centering
\includegraphics[height=1.75in]{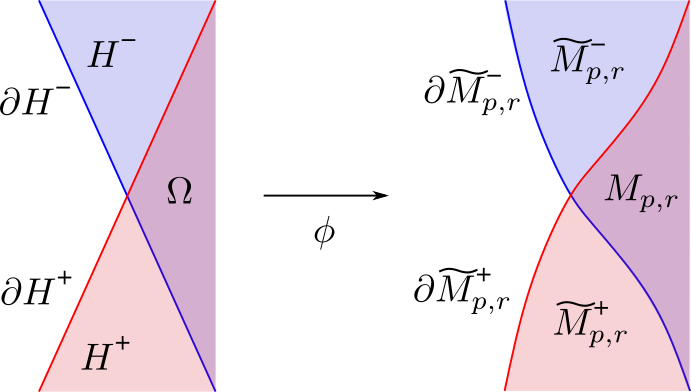} 
\end{subfigure}
\hspace{1cm}
\begin{subfigure}{0.5\linewidth}
\raggedleft
\includegraphics[height=1.75in]{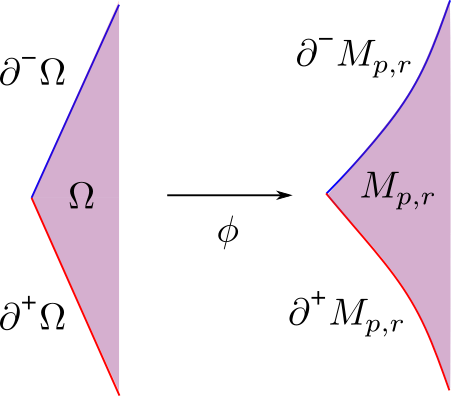}
\end{subfigure}
\caption{Local Extensions of $M$}
\label{fig:local-extensions}
\end{figure}

Next, we stratify a locally wedge-shaped manifold using its local models similar to Definition \ref{Def: stratification of wedge}. 
\begin{definition}[Stratification of manifolds]\label{Def: stratification of wedge manifold}
	Let $M^m\subset \R^L$ be a compact locally wedge-shaped $m$-manifold. 
	Define the stratification $\partial_{m-2}M\subset \partial_{m-1}M\subset \partial_mM$ of $M$ by
	$$ \partial_m M := M, \quad \partial_{m-1}M := \partial M, \quad \partial_{m-2}M := \cup_{p\in\partial M} \phi_p\big((\partial_{m-2}\Omega^m_{\theta}(p)\times 0^{L-m}) \cap \mB^L_R(0)\big),$$
	where $(\phi_p, \mB^L_R(p), \Omega^m_\theta(p))$ is a local model of $M$ around $p$. 
	Moreover, define
	\begin{itemize}
		\item $\interior(M) := \partial_m M\setminus\partial_{m-1}M$ to be the {\em interior} of $M$;
		\item $\partial^FM := \partial_{m-1}M \setminus\partial_{m-2}M$ to be the {\em face} of $M$;
		\item $\partial^EM := \partial_{m-2}M$ to be the {\em edge} of $M$. 
	\end{itemize}
\end{definition}

One should notice that $\bd^F M$ and $\bd^E M$ are the smooth and non-smooth parts of $\bd M$ respectively. 
Specifically, we have the following characterizations, which also show the above definitions are well-defined.   
\begin{lemma}\label{Lem: classify face and edge}
    Let $M^m\subset \R^L$ be a compact ($C^{l,\alpha}$-to-edge, $l\geq 1, \alpha\in (0,1]$) locally wedge-shaped manifold. 
    Then 
    \begin{itemize}
        \item[(i)] $\partial M$ is smooth on $\bd^F M = \{p\in \bd M: \theta(p)=\pi \}$;
        \item[(ii)] $\bd^E M = \{p\in\partial M: \theta(p)< \pi\}$ is a closed ($C^{l,\alpha}$) embedded $(m-2)$-submanifold. 
    \end{itemize}
\end{lemma}
\begin{proof}
    Firstly, note a non-trivial wedge domain can not be $C^1$-diffeomorphic to a half-space. 
    Hence, for any $p\in\bd M$ with local model $(\phi, \mB_R^L(p), \Omega)$, we have 
    \begin{itemize}
        \item if $\theta(p)=\pi$, then $\theta(q)=\pi$ for all $q\in \bd M\cap \mB_R^L(p) $;
        \item if $\theta(p)<\pi$, then $\theta(q)<\pi$ for all $q\in \phi((\bd^E\Omega\times 0^{L-m})\cap \mB_R^L(0)) $.
    \end{itemize}
    Together with the definitions, we see $\bd^E M =\bd_{m-2}M = \{p\in\bd M: \theta(p)<\pi\}$, $\bd^FM = \{p\in\bd M: \theta(p)=\pi\}$, and $\bd^F M$ is locally diffeomorphic to an $(m-1)$-plane.  
    Additionally, the first bullet indicates $\bd^F M$ is open in $\partial M$, and thus $\bd^E M$ is compact and locally ($C^{l,\alpha}$) diffeomorphic to an open set in $\partial_{m-2}\Omega\cong \R^{m-2}$, which gives the lemma.
\end{proof}

Now, by the above lemma, the term `$C^{l,\alpha}$-to-edge' in Definition \ref{Def: wedge manifold} indicates that $M$ is $C^{l,\alpha}$ near the edge $\bd^E M$ and $C^\infty$ away from the edge. 
Moreover, given $p\in\partial^E M$ with local model $(\phi, \mB^L_R(p), \Omega)$, we have $T_pM = (D\phi)_0(T_0\Omega\times 0^{L-m})$. 
For simplicity, we define the following concepts parallel to Definition \ref{Def: vertical/horizontal hyperplane in wedge}. 

\begin{definition}\label{Def: vertical/horizontal hyperplane in manifold}
    Let $M^m\subset \R^L$ be a compact locally wedge-shaped manifold. 
    Then for any $p\in \bd^EM$ with local model $(\phi, \mB^L_R(p), \Omega)$ and local extension $\tM_{p,R}$, we say $P\subset T_p\tM_{p,R}$ is a {\em horizontal hyperplane} of $T_pM$ if 
    \[ P = (D\phi)_0(P'\times 0^{L-m}) \quad\mbox{for some $P'\in \Pcal_{T_0\Omega}$}.\]
    Denote by $\Pcal_{T_pM}$ the set of horizontal hyperplanes of $T_pM$. 
    Additionally, given $\alpha\in [-\theta(p)/2,\theta(p)/2]$ and the vertical half hyperplane $P_\alpha^+$ in $\Omega$ of angle $\alpha$, we call
    \[P_\alpha(p) := (D\phi)_0(P_\alpha^+\times 0^{L-m}) \]
    the {\em vertical half hyperplane} of angle $\alpha$ in $T_pM$.  
\end{definition}


\subsection{Stationary varifolds with free boundary}
In the rest of this paper, we always assume $M^{n+1}\subset \R^L$ is a compact connected $(n+1)$-dimensional locally wedge-shaped Riemannian manifold with induced metric so that $\bd^EM\neq\emptyset$, and on each connected component $\bd^E_i M$ of $\bd^E M$, either
\begin{align}
    &\theta(p)<\frac{\pi}{2} \mbox{ for all } p\in \bd^E_i M ; \mbox{ or } \tag{$\dagger$}\label{dag}
    \\ &\theta(p)\equiv\frac{\pi}{2} \mbox{ and } A_{\bd^+M_{p,r}}(\eta_+, v) = A_{\bd^-M_{p,r}}(\eta_-, v) \mbox{ for all } p\in \bd^E_i M, v\in T_p(\bd^EM), \tag{$\ddagger$}\label{ddag}
\end{align}
where $A_{\bd^{\pm}M_{p,r}}$ is the second fundamental form of $\bd^{\pm}M_{p,r}$ (see (\ref{Eq: two pieces of face})), and $\eta_{\pm}$ is the outward unit co-normal of $\bd^{\pm}M_{p,r}$ along $\bd^E M_{p,r}$. 
We also use the following notations throughout the paper (see Appendix \ref{Sec: Fermi distance} for the construction of the Fermi-type distance function):
\begin{itemize}
    \item $\wD_p$: the Fermi-type distance function at $p\in\bd^EM$;
    \item $\wS_r(p)$: the Fermi-type wedge-cut sphere at $p\in\bd^E M$ with radius $r>0$; 
    \item $\wB_r(p)$: the (relatively) open Fermi-type spherical wedge at $p\in\bd^E M$ with radius $r$;
    \item $\wA_{s,t}(p)$: the open Fermi-type wedge-cut annulus $\wB_t(p)\setminus \clos(\wB_s(p))$;
    \item $\interior_M(A)$: the interior of of $A\subset M$ with respect to the subspace topology of $M$;
    \item $\bd_{rel}A$: the relative boundary of $A\subset M$ given by $\clos(A)\setminus\interior_M(A)$.
\end{itemize}
Additionally, we define the following spaces of vector fields parallel to \cite[(2.2)]{li2021min}:
\begin{eqnarray}
    \mfX(M)&:=& \{X\in\mfX(\R^L): X(p)\in T_pM \mbox{ for all $p\in M$}\},\nonumber
    \\
    \mfX_{tan}(M)&:=& \{ X\in \mathfrak{X}(M) : X(p) \in T_p\partial_i M, ~\forall p\in \partial_iM, ~i\in\{n-1,n,n+1\} \}.
\end{eqnarray}
Note that $T_pM$ is an $(n+1)$-dimensional half-space or an $(n+1)$-dimensional wedge domain for $p\in\bd^FM$ or $p\in\bd^E M$ respectively. 
Also, $\mfX_{tan}(M)$ is formed by the fields $X\in \mathfrak X(M)$ whose flow preserves the stratification of $M$. 

Next, we consider the varifolds (see \cite[2.1(18)]{pitts2014existence}\cite{simon1983lectures}) and their variations in the locally wedge-shaped manifold $M^{n+1}\subset \R^L$. 
Firstly, let $\RV_k(M)$ (resp. $\IV_k(M)$) be the space of (integer) rectifiable $k$-varifolds in $\R^L$ with support contained in $M$. 
Then we define $\V_k(M)$ to be the closure of $\RV_k(M)$ in the weak topology. 
For a $k$-varifold $V\in \V_k(M)$, denote by $\|V\|$ the weight measure of $V$, by $\spt(\|V\|)$ the support of $V$, and by $\VarTan(V,p)$ the varifold tangents of $V$ at $p\in\spt(\|V\|)$. 
Additionally, let the $\mF$-metric on $\V_k(M)$ be given as in \cite[2.1(19)]{pitts2014existence}, and recall this induces the weak topology on any mass bounded subset of $\V_k(M)$.

Given a $k$-varifold $V\in \V_k(M)$ and a vector field $X\in\mfX_{tan}(M)$, let $\{f_t\}$ be the diffeomorphisms generated by $X$, and note these satisfy $f_t(M)=M$. 
Then the first variation of $V$ along the vector field $X$ is given by 
\begin{equation}
    \delta V(X) = \frac{d}{dt}\Big|_{t=0} \|(f_t)_\# V\|(M)  = \int \Div_{S} X(p) \, dV(p,S). 
\end{equation}

\begin{definition}\label{Def: stationary varifold}
    Let $U$ be a relatively open subset of $M$. 
    Then we say a varifold $V\in\V_k(M)$ is {\it stationary in $U$ with free boundary} if $\delta V(X) = 0$ for all $X\in \mfX_{tan}(M)$ compactly supported in $U$. 
\end{definition}

Next, we show the monotonicity formula for stationary varifolds with free boundary near the edge of a locally wedge-shaped manifold (see \cite[Appendix A]{mazurowski2023curvature} for the proof). 

\begin{theorem}[Monotonicity Formula]\label{Thm: monotonicity fomula} 
    Let $M^{n+1}\subset\R^L$ be a locally wedge-shaped Riemannian manifold.
    Suppose $V\in \mathcal{V}_n(M)$ is a stationary varifold in $M$ with free boundary. 
    Then there are constants $C_{mono} > 0$ and $r_{mono} > 0$, depending only on $M\subset \R^L$, such that for all 
    $p\in \bd^E M$ and $0<s<t<r_{mono}$ we have 
    \[ \frac{\|V\|(\mB^L_s(p))}{s^n} \le C_{mono}\frac{\|V\|(\mB^L_t(p))}{t^n}.  \]
\end{theorem}

\begin{remark}\label{Rem: monotonicity radius}
    It also follows from the proof in \cite[Appendix A]{mazurowski2023curvature} (see also \cite[17.5]{simon1983lectures}) that $r_{mono}$ can be taken as $\infty$ provided $M$ is a Euclidean wedge domain. 
\end{remark}

Moreover, we also have the following maximum principle for $1$-codimensional stationary varifolds with free boundary in a relatively convex domain (\cite[Definition 2.4]{li2021min}) of a locally wedge-shaped manifold. 

\begin{theorem}[Maximum principle]\label{Thm: maximum principle for varifolds}
    Let $M^{n+1}\subset\R^L$ be a locally wedge-shaped Riemannian manifold satisfying (\ref{dag}) or (\ref{ddag}), and let $U\subset M$ be a relative open subset. 
    Suppose $V\in\V_n(M)$ is stationary in $U$ with free boundary, and $K\subset\subset U$ is a smooth relatively open connected subset in $M$ so that 
    \begin{itemize}
        \item[(1)] $\bd_{rel}K$ is a smooth strongly mean-convex hypersurface meeting $\bd M$ orthogonally (even at the edge $\bd^E M$);
        \item[(2)] $\spt(\|V\|)\subset \Clos(K)$.
    \end{itemize}
    Then $\spt(\|V\|)\cap \bd_{rel}K=\emptyset$. 
\end{theorem}
\begin{proof}
    See Appendix \ref{Sec: maximum principle for varifolds}. 
\end{proof}

\subsection{Almost properly embedded FBMHs}\label{subsec: almost properly embedded FBMHs}
As in the above sub-section, we consider a compact connected $(n+1)$-dimensional locally wedge-shaped Riemannian manifold $M^{n+1}\subset \R^L$ satisfying (\ref{dag}) or (\ref{ddag}). 
We first define the locally wedge-shaped hypersurfaces in $M$. 

\begin{definition}[Locally wedge-shaped hypersurfaces]\label{Def: locally wedge-shaped hypersurfaces}
    Let $\Sigma\subset M^{n+1}$ be a subset of $M$ that admits a $C^{l,\alpha}$-to-edge locally wedge-shaped $n$-manifold structure for some $l\in\mZ_+$, $\alpha\in(0,1]$. 
	Then $\Sigma$ is said to be {\em $C^{l,\alpha}$-to-edge embedded} in $M^{n+1}$, if $\Sigma\setminus\bd^E\Sigma$ is $C^\infty$-embedded in $M$ and $\Sigma$ is $C^{l,\alpha}$-embedded in $M$ near $\bd^E\Sigma$. 
\end{definition}

Equivalently speaking, the term `$C^{l,\alpha}$-to-edge embedded' in the above definition means that for any $p\in \Sigma$, there are local models $(\psi,\mB^L_r(p), \Omega_\Sigma)$ and $(\phi, \mB^L_r(p), \Omega_M)$ of $\Sigma$ and $M$ around $p$ respectively, so that the corresponding local extension (see (\ref{Eq: local extension})) $\widetilde{\Sigma}_{p,r}$ is $C^{l,\alpha}$-embedded in $\tM_{p,r}$, and the embedding is $C^\infty$ away from $\bd^E \Sigma$. 
Additionally, we also generalize the definition of almost proper embeddings from \cite[Definition 2.6]{li2021min} using the stratification of locally wedge-shaped manifolds. 

\begin{definition}[Almost proper embeddings]\label{Def: almost properly embedded}
	Let $\Sigma^n\subset M^{n+1}$ be a $C^{l,\alpha}$-to-edge embedded locally wedge-shaped hypersurface in $M$ for some $l\in\mZ_+$, $\alpha\in(0,1]$. 
	We say $\Sigma$ is {\em almost properly embedded} in $M$, denoted by $(\Sigma,\{\partial_m\Sigma\})\subset (M, \{\partial_{m+1} M\})$, if 
	\[ \partial_m\Sigma \subset \partial_{m+1}M,\qquad \forall m\in\{n-2, n-1, n\}. \]
	In particular, if $\partial^F\Sigma = \partial^F M \cap \Sigma$ and $\partial^E\Sigma = \partial^E M\cap\Sigma$, then we say $(\Sigma,\{\partial_m\Sigma\})\subset (M, \{\partial_{m+1} M\})$ is {\em properly embedded}. 
\end{definition}

Let $(\Sigma,\{\partial_m\Sigma\})\subset (M, \{\partial_{m+1} M\})$ be a $C^{l,\alpha}$-to-edge almost properly embedded locally wedge-shaped hypersurface. 
Then $\interior(\Sigma)$ must touch $\bd M$ tangentially at $\interior(\Sigma)\cap \bd^F M$, and $\bd^F\Sigma$ must touch $\bd^E M$ tangentially at $\bd^F\Sigma\cap \bd^E M$.
We can introduce the notations
\begin{equation}\label{Eq: boundary points on F&E}
	\partial^{FF}\Sigma := \partial^F\Sigma \cap \partial^FM \qquad{\rm and}\qquad \partial^{FE}\Sigma := \partial^F\Sigma\cap\partial^EM
\end{equation}
to represent the face of $\Sigma$ on $\partial^FM$ and $\partial^EM$ respectively. 
Note $\Sigma\cap \bd^EM= \bd^E\Sigma\cup\bd^{FE}\Sigma$. 
Thus, we call $p\in\Sigma\cap\bd^EM$ a 
\begin{itemize}
    \item {\em true edge point} if $p\in \bd^E\Sigma$; or a 
    \item {\em false edge point} if $p\in\bd^{FE}\Sigma$. 
\end{itemize}
Moreover, one can also classify the true/false edge points of $\Sigma$ by their tangent spaces. 
Specifically, for any $p\in\Sigma\cap\partial^EM$, $T_p\Sigma$ is an $n$-dimensional wedge domain almost properly embedded in $T_pM$. 
Therefore, by Lemma \ref{Lem: classify face and edge}, if $p\in\bd^E\Sigma$ is a true edge point, then $(T_p\Sigma,\{\partial_iT_p\Sigma\})\subset (T_pM,\{\partial_{i+1}T_pM\})$ is a non-trivial wedge domain. 
On the other hand, if $p\in\bd^{FE}\Sigma$ is a false edge point, then $T_p\Sigma$ is a half-hyperplane in an $(n+1)$-dimensional wedge domain $T_pM$.
Hence, 
\begin{equation}\label{Eq: classify flase edge point}
	p\in\bd^{FE}\Sigma \quad \Leftrightarrow \quad  T_p\Sigma = P_\alpha(p) \mbox{ for some $\alpha\in [-\theta/2, \theta/2]$},
\end{equation}
where $\theta=\theta(p)$ is the wedge angle of $M$ at $p$. 

Next, we consider the variation vector fields associated with an almost properly embedded locally wedge-shaped hypersurface. 

\begin{definition}\label{Def: variation vector fields}
    Let $(\Sigma,\{\bd_m \Sigma\})\subset  (M,\{\bd_{m+1}M\})$ be a $C^{l,\alpha}$-to-edge almost properly embedded locally wedge-shaped hypersurface with $l\in\mZ_+, \alpha\in(0,1]$. 
    Then we define \[\mfX(M,\Sigma)\] to be the space of vector fields $X\in \mathfrak X(M)$ such that for any $p\in\bd\Sigma\cap\bd M$ there exists $r=r(p,\Sigma,M)>0$ satisfying
    \begin{itemize}
        \item[(i)] if $p\in \bd^E \Sigma$, then $X(q)\in T_q \bd^E M$ for all $q\in \mB^L_r(p) \cap \bd^E M$;
        \item[(ii)] if $p\in \clos(\bd^{FF} \Sigma)$, then $X(q)\in T_q\bd M$ for all $q\in \mB^L_r(p)\cap \bd M$;
        \item[(iii)] if $p\in \bd^F\Sigma \setminus \clos(\bd^{FF}\Sigma)$, then $\bd^F \Sigma \cap \mB^L_r(p) \subset \bd^E M$ and $X\llcorner (\mB^L_r(p) \cap M) = \widetilde{X}\llcorner M$ for some
        $$
        \widetilde{X} \in
        \begin{cases}
            \mfX_{tan}(\tM_{p,r}^+) & \text{if } \alpha < 0,\\
            \mfX_{tan}(\tM_{p,r}^-) & \text{if } \alpha > 0,\\
            \mfX_{tan}(\tM_{p,r}^+) \cup \mfX_{tan}(\tM_{p,r}^-) &\text{if } \alpha = 0,
        \end{cases}
        $$
        where $\alpha\in [-\theta(p)/2,\theta(p)/2]$ is given by (\ref{Eq: classify flase edge point}) with $T_p\Sigma = P_\alpha(p)$.
    \end{itemize}
\end{definition}

Note $\{X\llcorner\Sigma: X\in \mfX_{tan}(M)\}= \{X\llcorner \Sigma: X\in \mfX(M,\Sigma)\}$ provided $\Sigma$ is properly embedded in $M$. 
In general, $\mfX_{tan}(M)\subset \mfX(M,\Sigma)$, which indicates $\mfX(M,\Sigma)$ is non-empty. 
Indeed, similar to \cite[Section 2.3]{li2021min} where the variations can push $\interior(\Sigma)\cap \bd M$ into the interior of $M$, we here also allow the flow of $X\in \mfX(M,\Sigma)$ to push $\bd^{FE}\Sigma$ into certain pieces of $\bd^FM$. 
See Figure \ref{fig:admissible-variations} for an illustration of the variations allowed by (iii). 

\begin{figure}[h]
    \centering
    \includegraphics[width=3in]{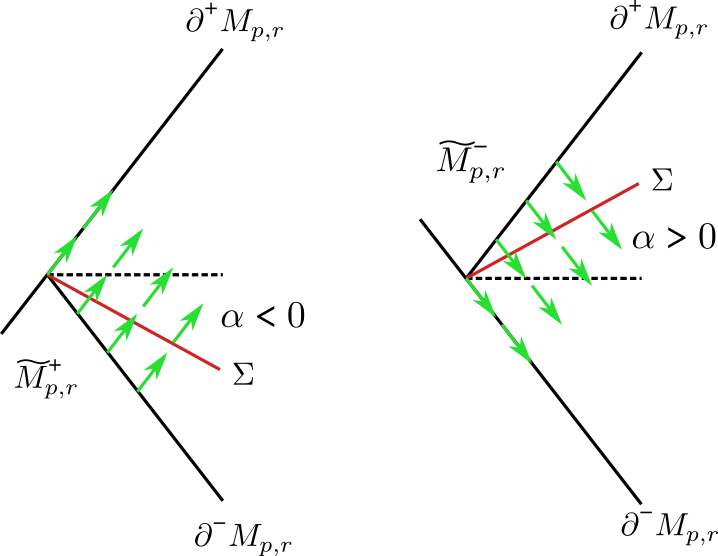}
    \caption{Admissible Variations}
    \label{fig:admissible-variations}
\end{figure}

Next, given $l\in\mZ_+, \alpha\in(0,1]$, and a $C^{l,\alpha}$-to-edge almost properly embedded locally wedge-shaped hypersurface $(\Sigma,\{\bd_m \Sigma\})\subset  (M,\{\bd_{m+1}M\})$, let $X\in\mfX(M,\Sigma)$ and let $\{f_t\}$ be the diffeomorphisms generated by $X$. 
Then $f_t(\Sigma)$ is also almost properly embedded in $M$ for $t\geq 0$ small enough (see \cite[Proposition 3.5]{mazurowski2023curvature}). 
Noting $\mH^{n-1}(\bd^E\Sigma)=0$, the first variation formula gives
\begin{equation}\label{Eq: 1st variation formula}
    \delta \Sigma(X) = \frac{d}{dt}\Big|_{t=0} {\rm Area}(f_t(\Sigma)) = -\int_{\Sigma} \laa H_{\Sigma}, X\raa + \int_{\bd^F \Sigma} \laa \eta,X\raa,
\end{equation}
where $\eta$ denotes the unit outward co-normal along $\bd^F \Sigma$. 

\begin{definition}\label{Def: stationary FBMH}
    Given a relative open subset $U\subset M$ and $l\in\mZ_+,\alpha\in(0,1]$, we say a $C^{l,\alpha}$-to-edge almost properly embedded locally wedge-shaped hypersurface $(\Sigma,\{\bd_m\Sigma\})\subset (U,\{U\cap\bd_{m+1}M\})$ is {\em stationary in $U$} if $\delta \Sigma(X) = 0$ for all $X\in \mathfrak X(M,\Sigma)$ that is compactly supported in $U$. 
\end{definition}

\begin{remark}
    As we mentioned after Definition \ref{Def: variation vector fields}, $\mfX_{tan}(M)\subsetneq \mfX(M,\Sigma)$ provided the touching set $(\interior(\Sigma)\cap\bd^FM)\cup \bd^{FE}\Sigma$ is non-empty. 
    Hence, a stationary varifold with free boundary and regular support (i.e. induced by some $(\Sigma,\{\bd_m \Sigma\})\subset  (M,\{\bd_{m+1}M\})$, may not be stationary in the sense of Definition \ref{Def: stationary FBMH}.
\end{remark}

Combining the first variation formula with Definition \ref{Def: variation vector fields}, one verifies that a $C^{l,\alpha}$-to-edge almost properly embedded locally wedge-shaped hypersurface $(\Sigma,\{\bd_m \Sigma\})\subset  (M,\{\bd_{m+1}M\})$ is stationary in $M$ if and only if the mean curvature of $\Sigma$ vanishes identically in $\interior(\Sigma)$ and $\Sigma$ is orthogonal to at least one smooth piece of the faces of $\bd M$ along $\bd^F\Sigma$, i.e. 
\begin{equation}\label{Eq: classify FBMH}
    H_\Sigma \equiv 0 ~\mbox{in $\interior(\Sigma)$}, \quad \eta \perp \bd M ~\mbox{on $\bd^F\Sigma$}, \quad \eta\perp \bd^+ M_{p,r}~{\rm or} ~\bd^- M_{p,r} ~\mbox{for $p\in \bd^{FE}\Sigma$},
\end{equation}
where $H_\Sigma$ is the mean curvature of $\Sigma$ and $\eta$ is the unit outward co-normal along $\bd^F \Sigma$. (See \cite[Proposition 3.8]{mazurowski2023curvature}.)
As a direct generalization of the definition in smooth manifolds, we say $(\Sigma,\{\bd_m \Sigma\})\subset  (M,\{\bd_{m+1}M\})$ is a {\em free boundary minimal hypersurface} in $M$ (abbreviated as FBMH) if it satisfies (\ref{Eq: classify FBMH}). 
We also say $\Sigma$ is a FBMH in a relative open subset $U\subset M$ if it satisfies (\ref{Eq: classify FBMH}) in $U$, or equivalently it is stationary in $U$.    

In addition, under the assumption (\ref{dag}) or (\ref{ddag}), the $C^{1,\alpha}$-to-edge regularity of a locally wedge-shaped FBMH can be upgraded to $C^{2,\alpha_0}$-to-edge ($\alpha_0\in (0,1)$) by Appendix \ref{Sec: elliptic regularity}. 
\begin{theorem}\label{Thm: regularity of FBMH}
    Suppose $M^{n+1}\subset\R^L$ is a locally wedge-shaped Riemannian manifold satisfying (\ref{dag}) or (\ref{ddag}). 
    Let $\alpha\in(0,1]$ and $(\Sigma,\{\bd_m \Sigma\})\subset  (M,\{\bd_{m+1}M\})$ be a $C^{1,\alpha}$-to-edge almost properly embedded locally wedge-shaped FBMH. 
    Then $\Sigma$ is $C^{2,\alpha_0}$-to-edge for some $\alpha_0\in (0,1)$. 
\end{theorem}

Let $(\Sigma,\{\bd_m \Sigma\})\subset  (M,\{\bd_{m+1}M\})$ be a $C^{l,\alpha}$-to-edge almost properly embedded locally wedge-shaped FBMH. 
Combining the orthogonally meeting conditions in (\ref{Eq: classify FBMH}) with (\ref{Eq: classify flase edge point}) and conditions (\ref{dag})(\ref{ddag}), we see 
\begin{eqnarray}
    &p\in\bd^{E}\Sigma &\Rightarrow ~~T_p\Sigma=P\cap T_pM \mbox{ for some horizontal hyperplane $P\in \mathcal{P}_{T_pM}$;}
    \\
    &p\in\bd^{FE}\Sigma &\Rightarrow ~~\theta(p)=\pi/2, ~T_p\Sigma = P_{-\pi/4}(p) {\rm ~or~} P_{\pi/4}(p) \mbox{ (see Definition \ref{Def: vertical/horizontal hyperplane in manifold}).}
\end{eqnarray}
Additionally, further assume that $\Sigma$ is {\em two-sided}, i.e. there is a continuous choice of the unit normal vector field $\nu$ on $\Sigma$. 
Then for any $X\in\mfX(M,\Sigma)$ with $X\llcorner\Sigma = f\nu$, the second variation of $\Sigma$ for the area functional along $X$ is given by 
\begin{equation}\label{Eq: 2nd variation formula}
    \delta^2\Sigma(X) = Q_\Sigma(f,f) := \int_{\Sigma} \vert \nabla^\Sigma f\vert^2 - (\vert A_\Sigma\vert^2 - \Ric_M(\nu,\nu))f^2  + \int_{\bd \Sigma} h_{\pm}(\nu,\nu)f^2,
\end{equation}
where $A_\Sigma$ is the second fundamental form of $\Sigma$, and $h_{\pm}$ is the second fundamental form of $\bd^FM$ (for $p\in\bd^F\Sigma$) or $\bd^{\pm}M_{p,r}$ (for $p\in\bd^{FE}\Sigma$ with $\eta\perp \bd^{\pm}M_{p,r}$) with respect to the outward unit normal $\nu_{\bd M}$. 
Note the boundary integral is taken over $\bd^F\Sigma$ by Lemma \ref{Lem: classify face and edge}.

\begin{definition}\label{Def: stabe FBMH}
    Let $U\subset M$ be a relative open subset and $(\Sigma,\{\bd_m\Sigma\}) \subset (U,\{U\cap \bd_{m+1}M\})$ be a $C^{l,\alpha}$-to-edge almost properly embedded locally wedge-shaped hypersurface which is stationary in $U$ in the sense of Definition \ref{Def: stabe FBMH}. 
    Then we say $\Sigma$ is {\it stable in $U$} if $\Sigma$ is two-sided and $\delta^2\Sigma(X) \ge 0$ for all $X = f\nu \in \mathfrak X(M,\Sigma)$ compactly supported in $U$.  
\end{definition}

Finally, we state the following compactness theorem for stable almost properly embedded locally wedge-shaped FBMHs with a uniform area bound. 
\begin{theorem}[Compactness theorem for stable FBMHs]\label{Thm: compactness for FBMHs}
    Let $2\leq n\leq 5$, $M^{n+1}\subset\R^L$ be a locally wedge-shaped Riemannian manifold satisfying (\ref{dag}) or (\ref{ddag}), and $U\subset M$ be a simply connected relative open subset. 
    Suppose $\{(\Sigma_k,\{\bd_m\Sigma_k\}) \subset (U,\{U\cap \bd_{m+1}M\})\}_{k\in\N}$ is a sequence of $C^{2,\alpha}$-to-edge almost properly embedded locally wedge-shaped FBMHs which are stable in $U$ and $\sup_{k\in\N}\Area(\Sigma_k)<\infty$. 
    
    Then, up to a subsequence, $(\Sigma_k,\{\bd_m\Sigma_k\})$ converges (possibly with multiplicities) to an $C^{2,\alpha}$-to-edge almost properly embedded locally wedge-shaped FBMH $(\Sigma_\infty,\{\bd_m\Sigma_\infty\}) \subset (U,\{U\cap \bd_{m+1}M\})$ which is also stable in $U$. 
    Additionally, the convergence is $C^{2,\alpha'}$-to-edge ($\forall \alpha'\in (0,\alpha)$), $C^\infty$ away from the edge $\bd^E\Sigma_\infty$, locally uniform and graphical. 
\end{theorem}
\begin{proof}
    Notice that the simply connectivity assumption on $U$ is used to guarantee that every $\Sigma_k$ is two-sided. 
    Then, the above compactness theorem follows directly from the curvature estimates in \cite[Theorem 1.3]{mazurowski2023curvature}, the graphical convergence theorem \cite[Theorem 1.1]{mazurowski2023curvature} and the elliptic regularity estimates in Appendix \ref{Sec: elliptic regularity}.
\end{proof}

\section{Almost minimizing varifolds with free boundary}\label{Sec: variations in wedge manifold}

In this section, let $M^{n+1}\subset\R^L$ be a locally wedge-shaped Riemannian manifold.
After collecting the definitions and some useful lemmas for the equivalent classes of relative cycles, we will introduce the important concept of {\em almost minimizing varifolds with free boundary}. 
It should also be noted that all of the definitions and conclusions in this section are valid without the angle assumption (\ref{dag})(\ref{ddag}). 

\subsection{Equivalence classes of relative cycles}\label{Subsec: relative cycles}

In this subsection, we always denote by $U\subset M$ a relative open subset of $M$, by $K\subset\subset U$ a compact subset in $U$, and by $(A,B):=(U, U\cap M)$. 

Let $\mR_k(M)$ (resp. $\mI_k(M)$) be the space of integer rectifiable (resp. integral) $k$-currents in $\R^L$ supported in $M$ (see \cite{federer2014geometric}). 
Additionally, we denote by $\M$ and $\F^K$ the mass norm and the flat semi-norm (relative to $K$) in $\mR_k(M)$ respectively. 
For any $T\in\mR_k(M)$, let $|T|$ and $\|T\|$ be the integer rectifiable varifold and the Radon measure induced by $T$ respectively. 
Now, the spaces of {\em relative cycles} are given as in \cite[(3.1)]{li2021min}:
\begin{eqnarray}\label{Eq: relative cycles}
    Z_k(A,B) &:= & \{T\in\mR_k(M): \spt(T)\subset A, ~\spt(\bd T)\subset B\},
    \\
    Z_k(B,B) &:= & \{T\in\mR_k(M): \spt(T)\subset B\}.
\end{eqnarray}
Then define the quotient group
\begin{equation}
    \Z_k(A,B) := Z_k(A,B)/Z_k(B,B)
\end{equation}
to be the space of {\em equivalence classes of relative $k$-cycles} (cf. \cite[Definition 3.1]{li2021min}). 
Additionally, for any $\tau,\{\tau_i\}_{i\in\N}$ in $\Z_k(A,B)$ and $\sigma,\sigma_1,\sigma_2$ in $\Z_k(M,\bd M)$, we define the following notation:
\[ 
\begin{array}{cl}
    \spt(\tau):=\cap_{T\in\tau} \spt(T) :& \mbox{ The support of $\tau$} ;\\
    \M(\tau):=\inf_{T\in\tau}\M(T) :& \mbox{ The mass norm of $\tau$} ;\\
    \F^K(\tau):=\inf_{T\in\tau}\F^K(T) :& \mbox{ The flat (semi)-norm of $\tau$ (with respect to $K\subset A$)} ;\\
    \mbox{$T\in\tau$ with $T\llcorner B = 0$} :& \mbox{ The (unique) {\em canonical representative} of $\tau$} ;\\
    \tau_i\rightharpoonup \tau :& 
        \begin{array}{l} 
        \mbox{The weak convergences of $\tau_i$ to $\tau$ in the sense of} \\  
        \mbox{$\lim_{i\to\infty}\F^K(\tau_i-\tau_\infty) = 0$ for each compact $K\subset A$};
        \end{array} \\
    \mF(\sigma_1,\sigma_2):= \F^M(\sigma_1-\sigma_2) + \mF(|S_1|,|S_2|) :& 
    \begin{array}{l} 
        \mbox{The $\mF$-distance between $\sigma_1,\sigma_2$, where $S_1\in\sigma_1$ } \\  
        \mbox{and $S_2\in\sigma_2$ are the canonical representatives};
        \end{array} \\
    f_\# \sigma := [f_\# S] : & 
    \begin{array}{l} 
        \mbox{The push-forward of $\tau$ under a Lipschitz map} \\  
        \mbox{$f$ with $f(\bd M)\subset \bd f(M)$, and $S\in\tau$}.
    \end{array}
\end{array}
\]
For simplicity, we write $\F := \F^M$ in the rest of this paper.

Note the canonical representative of $\tau\in Z_k(A,B)$ always exists uniquely by taking $T:=T'-T'\llcorner B$ for any $T'\in\tau$. The canonical representative also satisfies 
\begin{equation}\label{Eq: canonical representative property}
    \M(T) = \M(\tau) \quad{\rm and}\quad \spt(T)=\spt(\tau). 
\end{equation}
Moreover, given $\tau_i\rightharpoonup\tau$ in $\Z_k(A,B)$, it follows from the proof of \cite[Lemma 3.7]{li2021min} that the mass norm is lower semi-continuous, i.e. 
\begin{equation}\label{Eq: lower semi-continuous mass}
    \M(\tau)\leq \liminf_{i\to\infty} \M(\tau_i),
\end{equation}
and $\spt(\tau)\subset K$ provided $\spt(\tau_i)\subset K$ for some compact $K\subset A$.  

Since $M$ is compact with piece-wise smooth boundary, we see $M$ is a compact Lipschitz neighborhood retract in the sense of \cite[\S 1]{almgren1962homotopy}, and the intrinsic distance function $\rho(\cdot) := \dist_A(\cdot, B)$ is a Lipschitz function on $A$. 
In addition, there also exists a Lipschitz projection map $\pi: A\to B$ onto $B$. 
Hence, using these observations, the proof of the following lemmas can be taken verbatim from \cite[Section 3.1, 3.2]{li2021min}, and thus we only list the conclusions here for simplicity. 

\begin{lemma}\label{Lem: integral representative}
    {\rm (\cite[Lemma 3.8]{li2021min}).} 
    For any $\tau\in\Z_k(A,B)$, there is a sequence of $\{T_i\}_{i\in\N}\subset \tau\cap \mI_k(A)$ so that $\M(T_i)\to \M(\tau)$ as $i\to\infty$. 
\end{lemma}

\begin{lemma}\label{Lem: compactness of currents}
    {\rm (\cite[Lemma 3.10]{li2021min}).} 
    For any $C>0$ and compact Lipschitz neighborhood retracts $K,K'\subset A$ in the sense of \cite[\S 1]{almgren1962homotopy} so that $K\subset \interior_M(K')$, the set formed by $\tau\in \Z_k(A,B)$ with $\spt(\tau)\subset K$ and $\M(\tau)\leq C$ is (sequentially) compact under the $\F^{K'}$-topology.
\end{lemma}

\begin{lemma}\label{Lem: classify F convergence}
    {\rm (\cite[Lemma 3.13]{li2021min}).} 
    Given $\tau$ and $\{\tau_i\}_{i\in\N}$ in $\Z_k(M,\bd M)$, we have $\mF(\tau_i,\tau)\to 0$ if and only if $\F(\tau_i-\tau)\to 0$ and $\M(\tau_i)\to\M(\tau)$.
\end{lemma}

\begin{lemma}\label{Lem: F-isoperimetric lemma}
    {\rm (\cite[Lemma 3.15]{li2021min}).} 
    There exist $\epsilon_M>0$ and $C_{M}\geq 1$ depending only on $M^{n+1}\subset\R^L$ so that for any $\tau_1,\tau_2\in\Z_n(M,\bd M)$ with $\F(\tau_1-\tau_2)<\epsilon_{M}$, there is $Q\in \mI_{n+1}(M)$, known as the {\em $\F$-isoperimetric choice of $\tau_1,\tau_2$}, satisfying
    \[\spt(T_2-T_1 - \bd Q)\subset\bd M \quad {\rm and}\quad \M(Q)\leq C_{M} \F(\tau_1-\tau_2),\]
    where $T_1,T_2$ are the canonical representatives of $\tau_1,\tau_2$ respectively. 
\end{lemma}

\begin{lemma}\label{Lem: M-isoperimetric lemma}
    {\rm (\cite[Lemma 3.17]{li2021min}).} 
    There exist $\epsilon_M>0$ and $C_M\geq 1$ depending only on $M^{n+1}\subset\R^L$ so that for any $\tau_1,\tau_2\in\Z_n(M,\bd M)$ with $\M(\tau_1-\tau_2)<\epsilon_M$, there are $Q\in \mI_{n+1}(M)$ and $R\in \mR_n(\bd M)$ satisfying
    \[ T_2-T_1 = \bd Q + R \quad {\rm and}\quad \M(Q)\leq C_M \M(\tau_1-\tau_2),\]
    where $T_1,T_2$ are the canonical representatives of $\tau_1,\tau_2$ respectively. 
\end{lemma}

\subsection{Almost minimizing varifolds}\label{Subsec: def amv}

We now introduce the definition of almost minimizing varifolds with free boundary. 
As before, denote by $U\subset M$ a relative open subset of $M$, and by $\F:=\F^M$. 
To begin with, we recall the definition of $(\epsilon,\delta)$-deformations. 

\begin{definition}\label{Def: epsilon delta deformations}
    Let $\bnu$ be $\F$, $\M$, or $\mF$. 
    Then for any $\epsilon,\delta>0$ and $\tau\in \Z_n(M,\bd M)$, a finite sequence $\{\tau_i\}_{i=1}^q$ in $\Z_n(M,\bd M)$ is said to be an {\em $(\epsilon,\delta)$-deformation} of $\tau$ in $U$ under the metric $\bnu$ provided 
    \begin{itemize}
        \item $\tau=\tau_0$ and $\spt(\tau - \tau_i)\subset U$ for $i=1,\dots,q$;
        \item $\bnu(\tau_i - \tau_{i-1})\leq \delta$ for $i=1,\dots,q$;
        \item $\M(\tau_i)\leq \M(\tau) + \delta$ for $i=1,\dots,q$;
        \item $\M(\tau_q)<\M(\tau) - \epsilon$. 
    \end{itemize}
    Additionally, denote by 
    \[ \mfa_n(U;\epsilon,\delta;\bnu) \]
    the set of all $\tau\in\Z_n(M,\bd M)$ which admits no $(\epsilon,\delta)$-deformation in $U$. 
\end{definition}

Next, we can define the almost minimizing varifolds with free boundary. 
\begin{definition}\label{Def: amv}
    We say a varifold $V\in\V_n(M)$ is {\em almost minimizing in $U$ with free boundary} if there are $\{\epsilon_i\}_{i\in\N}, \{\delta_i\}_{i\in\N}\subset \R_+$ and $\{\tau_i\}_{i\in\N}\subset \mfa_n(U;\epsilon_i,\delta_i;\F)$ so that as $i\to\infty$, $\epsilon_i\to 0$, $\delta_i\to 0$, and $\mF(|T_i|,V)\to 0$, where $\{T_i\in\tau_i\}_{i\in\N}$ are the canonical representatives. 
\end{definition}

\begin{lemma}\label{Lem: amv imply stationary}
    Suppose $V\in \V_n(M)$ is almost minimizing in a relative open subset $U\subset M$ with free boundary, then $V$ is stationary in $U$ with free boundary.
\end{lemma}
\begin{proof}
    The proof can be obtained almost verbatim from \cite[Theorem 3.3]{pitts2014existence}. 
\end{proof}

Finally, although the flat norm $\F$ is used in Definition \ref{Def: amv}, the following theorem shows that the $\M$-norm and the $\mF$-distance can also be used to define almost minimizing varifolds with free boundary after shrinking the relative open subset $U\subset M$. 

\begin{theorem}\label{Thm: equivalent amv}
    Given $V\in\V_n(M)$, the following statements satisfies $(a)\Rightarrow(b)\Rightarrow(c)\Rightarrow(d)$:
    \begin{itemize}
        \item[(a)] $V$ is almost minimizing in $U\subset M$ with free boundary;
        \item[(b)] for any $\epsilon$, there exist $\delta$ and $\tau\in \mfa_n(U;\epsilon,\delta;\mF)$ so that $\mF(|T|,V)<\epsilon$, where $T\in\tau$ is the canonical representative;
        \item[(c)] for any $\epsilon$, there exist $\delta$ and $\tau\in \mfa_n(U;\epsilon,\delta;\M)$ so that $\mF(|T|,V)<\epsilon$, where $T\in\tau$ is the canonical representative;
        \item[(d)] $V$ is almost minimizing in $W$ with free boundary for any relative open $W\subset\subset U$.
    \end{itemize}
\end{theorem}

\begin{proof}
It is easy to see that (a)$\implies$(b)$\implies$(c). The proof that (c)$\implies$(d) is identical to the proof given in \cite[Theorem 3.20]{li2021min}, once we show that the analog of \cite[Lemma B.1]{li2021min} holds in our setting. Thus we need to prove the following.  

\begin{claim} Given $L>0$, $\delta > 0$, a relatively open subset $W\subset\subset U$, and $\tau \in \Z_n(M,\bd M)$, there exists 
\[
\eta = \eta(L,\delta,W,\tau) > 0
\]
such that for any $\sigma_1,\sigma_2\in \Z_n(M,\bd M)$ satisfying 
\begin{itemize}
\item $\spt(\sigma_i-\tau)\subset W$, $i=1,2$,
\item $\M(\sigma_i) \le L$, $i=1,2$, 
\item $\mathcal F(\sigma_1-\sigma_2)\le \eta$, 
\end{itemize}
there exists a sequence $\sigma_1 = \tau_0,\tau_1,\hdots,\tau_k = \sigma_2\in \Z_n(M,\bd M)$ such that for $j=0,1,\hdots,k-1$ we have 
\begin{itemize}
\item $\spt(\tau_j-\tau)\subset U$,
\item $\M(\tau_j) \le L+\delta$,
\item $\M(\tau_j-\tau_{j+1})\le \delta$. 
\end{itemize}
\end{claim}

Rather than giving a direct proof of the claim, we will instead use a smoothing procedure and then appeal to Lemma B.1 in \cite{li2021min}. Suppose we are given $L > 0$, $\delta > 0$, a relatively open subset $W\subset\subset U$, and $\tau \in \Z_n(M,\bd M)$. Fix a number $\gamma > 0$ to be specified later. By smoothing out the edge of $M$, we can find a smooth manifold with boundary $(N,\bd N)$ such that there is a bilipschitz homeomorphism $f\colon (M,\bd M)\to (N,\bd N)$ satisfying $\text{Lip}(f),\text{Lip}(f^{-1}) \le 1+\gamma$. We can then apply Lemma B.1 in \cite{li2021min} with $L(1+\gamma)^n$, $\frac{\delta}{2(1+\gamma)^n}$, the relatively open set $f(W)\subset\subset f(U)$, and $f_\sharp \tau \in \Z_n(N,\bd N)$ to get a constant 
\[
\eps = \eps\left(L(1+\gamma)^n,\frac{\delta}{2(1+\gamma)^n},f(W),f_\sharp(\tau)\right)=\eps(L,\delta,W,\tau,\gamma).
\]
Define $\eta :=\frac{\eps}{(1+\gamma)^{n+1}}$. 

Now suppose that $\sigma_1,\sigma_2\in \Z_n(M,\bd M)$ satisfy the assumptions of the claim. Observe that 
\begin{itemize}
\item $\spt(f_\sharp(\sigma_i)-f_\sharp(\tau))\subset f(W)$, $i=1,2$,
\item $\M(f_\sharp(\sigma_i)) \le (1+\gamma)^n \M(\sigma_i) \le L(1+\gamma)^n$, $i=1,2$, 
\item $\mathcal F(f_\sharp(\sigma_1)-f_\sharp(\sigma_2)) \le (1+\gamma)^{n+1} \mathcal F(\sigma_1-\sigma_2) \le \eta (1+\gamma)^{n+1} \le \eps$. 
\end{itemize}
Hence Lemma B.1 in \cite{li2021min} gives the existence of $\rho_0 = f_\sharp(\sigma_1),\rho_1,\hdots,\rho_k = f_\sharp(\sigma_2)\in \Z_n(N,\bd N)$ such that for $j=0,1,\hdots,k-1$ we have 
\begin{itemize}
\item $\spt(\rho_j-f_\sharp(\tau)) \subset f(U)$, $\phantom{\frac{\delta}{\gamma^n}}$
\item $\M(\rho_j) \le L(1+\gamma)^n + \frac{\delta}{2(1+\gamma)^n}$,
\item $\M(\rho_j-\rho_{j+1}) \le \frac{\delta}{2(1+\gamma)^n}$.
\end{itemize} 
Now define $\tau_j = (f^{-1})_\sharp \rho_j$ and observe that $\tau_0 = \sigma_1$ and $\tau_k = \sigma_2$ and for $j=0,1,\hdots,k-1$ we have
\begin{itemize}
\item $\spt(\tau_j-\tau)\subset U$,
\item $\M(\tau_j) \le L(1+\gamma)^{2n} + \frac{\delta}{2}$,
\item $\M(\rho_j-\rho_{j+1}) \le \frac{\delta}{2}$. 
\end{itemize}
Finally, select $\gamma = \gamma(L,\delta)$ small enough that $L(1+\gamma)^{2n} + \frac{\delta}{2} \le L + \delta$. Then $\eta = \eta(L,\delta,W,\tau)$ is as required. This completes the proof of the claim, and hence the proof of the theorem.
\end{proof}

\section{Min-max constructions}
\label{sec:min-max}

This section is devoted to the construction of stationary varifolds with free boundary satisfying a suitable almost-minimizing property. We continue to assume that $M^{n+1}\subset \R^L$ is a locally wedge-shaped manifold. 

\subsection{Homotopy notions}

To begin, we recall the discrete homotopy classes introduced by Almgren and Pitts \cite{pitts2014existence}.  The basic objects in the construction are sequences of maps $\phi_i\colon \dom(\phi_i)\to \Z_{n}(M,\bd M)$ whose domains are finer and finer discretizations of the unit cube.  

\begin{definition}
\label{cube-complex} 
Let $j,k,m,p,q\in \N$. Let $I^m = [0,1]^m$ and let $I^m_0 = \bd I^m$.  
\begin{itemize}
\item[(i)] Let $I(1,j)$ denote the cell structure on the unit interval $I$ whose 0-cells are the points $[\frac{i}{3^k}]$ and whose intervals are $[\frac{i}{3^k},\frac{i+1}{3^k}]$. 
\item[(ii)] Let $I(m,j)$ denote the cubical complex structure on $I^m$ given by 
\[
\underbrace{I(1,j) \otimes I(1,j) \otimes \cdots \otimes I(1,j)}_{m \text{ times}}.
\]
Thus a cell in $I(m,j)$ has the form $\alpha =\alpha_1 \otimes \cdots \otimes \alpha_m$, where each $\alpha_i$ is either a vertex or an interval in $I(1,j)$.  We say $\alpha$ is a $p$-cell when $\sum_{i=1}^m \dim(\alpha_i) = p$. 
\item[(iii)] Let $I_0(m,j)$ denote the cubical complex structure on $I^m_0$ induced by $I(m,j)$. 
\item[(iv)] Let $I(m,j)_p$ denote the set of all $p$-cells in $I(m,j)$. Likewise let $I_0(m,j)_p$ denote the set of all $p$-cells in $I_0(m,j)$.
\item[(v)] Given a cell $\alpha \in I(m,j)$, let $\alpha(p)$ be the cubical subcomplex consisting of all cells in $I(m,j+p)$ which are contained in $\alpha$. Likewise, let $\alpha_0(p)$ be the set of all cells in $I(m,j+p)$ which are contained in the boundary of $\alpha$.  Let $\alpha(p)_q$ denote the set of all $q$-cells in $\alpha(p)$, and let $\alpha_0(p)_q$ denote the set of all $q$-cells in $\alpha_0(p)$. 
\item[(vi)] The boundary map is defined on $I(1,j)$ by $\bd [a,b] = [b]-[a]$ for a 1-cell and $\bd [v] = 0$ for a 0-cell. More generally, the boundary map $\bd \colon I(m,j)\to I(m,j)$ 
is defined by 
\[
\bd(\alpha_1 \otimes \cdots \otimes \alpha_m) =  \sum_{i=1}^m (-1)^{\sigma(i)} \alpha_1 \otimes \cdots \otimes \bd \alpha_i \otimes \cdots \otimes \alpha_m 
\]
where $\sigma(i) = \sum_{\ell < i} \dim(\alpha_\ell)$. 
\item[(vii)] The distance function $\md \colon I(m,j)_0 \times I(m,j)_0 \to \N$ is defined by 
\[
\md(x,y) = 3^j \sum_{i=1}^m \vert x_i - y_i\vert.
\]
\item[(viii)] Finally, the map $n(j,k)\colon I(m,j)_0\to I(m,k)_0$ is defined by letting $n(j,k)(x)$ be the unique $y\in I(m,k)_0$ which is closest to $x$. 
\end{itemize}
\end{definition}

\begin{definition}
Given a map $\phi\colon I(m,j)_0\to\mathcal Z_n(M,\bd M)$, the {\em fineness} of $\phi$ is 
\[
\f(\phi) = \sup \left\{\frac{\M(\phi(x)-\phi(y))}{\md(x,y)}:\, x,y\in I(m,j)_0,\ x\neq y\right\}.
\]
\end{definition}

\begin{definition} We write $\phi\colon I(m,j)_0\to (\mathcal Z_n(M,\bd M),\{0\})$ for a map $\phi \colon I(m,j)_0 \to \mathcal Z_n(M,\bd M)$ such that $\phi(x) = 0$ for all $x\in I_0(m,j)_0$.  
\end{definition} 

Next, we recall the notion of homotopy for discrete mappings. 

\begin{definition}
Two maps 
\[
\phi \colon I(m,j)_0 \to (\Z_n(M,\bd M),\{0\}) \text{ and }  \theta \colon I(m,k)_0 \to (\Z_n(M,\bd M),\{0\})
\]
are called {\em homotopic with fineness $\delta$} if there exists an integer $\ell \ge \max\{j,k\}$ and a map $\psi\colon I(m,\ell)_0\times I(1,\ell)_0\to \mathcal Z_n(M,\bd M)$ such that 
\begin{itemize}
\item[(i)] $\psi(x,[0]) = \phi(n(\ell,j)(x))$ and $\psi(x,[1]) = \theta(n(\ell,k)(x))$ for all $x\in I(m,\ell)_0$,
\item[(ii)] $\psi(x,t) = 0$ for all $x\in I_0(m,\ell)_0$ and all $t\in I(1,\ell)_0$,
\item[(iii)] $\f(\psi) \le \delta$. 
\end{itemize}
\end{definition}

\begin{definition}
Let $\{j_i\}_{i\in \N}$ be an increasing sequence of natural numbers. 
An {\em $(m,\M)$-homotopy sequence into $(\mathcal Z_n(M,\bd M),\{0\})$} is a sequence of maps 
\[
\phi_i\colon I(m,j_i)_0\to (\Z_n(M,\bd M),\{0\})
\]
such that 
\begin{itemize}
\item[(i)] $\phi_i$ is homotopic to $\phi_{i+1}$ with fineness $\delta_i$ and $\delta_i\to 0$ as $i\to \infty$,
\item[(ii)] $\sup\{\M(\phi_i(x)): i\in \N,\ x\in I(m,j_i)_0\} < \infty$. 
\end{itemize}
\end{definition}

\begin{definition}
Let 
\[
\phi_i\colon I(m,j_i)_0\to (\Z_n(M,\bd M),\{0\}) \text{ and }\theta_i\colon I(m,k_i)_0 \to (\Z_n(M,\bd M),\{0\})
\]
be two $(m,\M)$-homotopy sequences of mappings into $(\Z_n(M,\bd M),\{0\})$. We say that $\{\phi_i\}_{i\in \N}$ is homotopic to $\{\theta_i\}_{i\in \N}$ if $\phi_i$ is homotopic to $\theta_i$ with fineness $\eta_i$ and $\eta_i\to 0$ as $i\to \infty$. 
\end{definition} 

As in \cite{pitts2014existence} (also see \cite{li2021min}), it is straightforward to verify that homotopy is an equivalence relation on $(m,\M)$-homotopy sequences of mappings. Therefore it makes sense to consider the set of all homotopy classes of such sequences. 

\begin{definition}
Let $\pi^\sharp_m(\Z_n(M,\bd M),\{0\})$ denote the set of equivalence classes of $(m,\M)$-homotopy sequences of mappings into $(\Z_n(M,\bd M),\{0\})$. 
\end{definition} 

Now consider a homotopy class $\Pi \in \pi^\sharp_m(\Z_n(M,\bd M),\{0\})$. We associate to $\Pi$ the following min-max value called the $\it width$. 

\begin{definition}
Given an $(m,\M)$-homotopy sequence $\phi_i\colon I(m,j_i)_0\to (\Z_n(M,\bd M),\{0\})$, define 
\[
\mathbf{L}(\{\phi_i\}_{i\in \N}) = \limsup_{i\to \infty} \bigg[\sup\left\{\M(\phi_i(x)):\, x\in I(m,j_i)_0\right\}\bigg]. 
\]
Given a homotopy class $\Pi\in  \pi^\sharp_m(\Z_n(M,\bd M),\{0\})$, the quantity 
\[
\mathbf{L}(\Pi) = \inf \{ \mathbf{L}(\{\phi_i\}_{i\in \N}):\, \{\phi_i\}_{i\in \N} \in \Pi\}
\]
is called the width of $\Pi$. 
\end{definition}

A {\em critical sequence} for $\Pi$ is an $(m,\M)$-homotopy sequence $\{\phi_i\}_{i\in \N} \in \Pi$ such that $\mathbf L(\Pi) = \mathbf{L}(\{\phi_i\}_{i\in \N})$.  A straightforward diagonal argument shows that critical sequences always exist.  Given a critical sequence, we associate it with the following critical set. 

\begin{definition}
Let $\Pi \in  \pi^\sharp_m(\Z_n(M,\bd M),\{0\})$ and let $\phi_i \colon I(m,j_i)_0 \to (\Z_n(M,\bd M),\{0\})$ be a critical sequence for $\Pi$. The critical set $K(\{\phi_i\}_{i\in \N})$ is the set of all varifolds $V\in \mathcal V_n(M)$ such that 
\[
V = \lim_{k\to \infty} \vert T_k\vert
\]
where 
\begin{itemize}
\item[(i)] $\{i_k\}_{k\in\N}$ is an increasing sequence of natural numbers,
\item[(ii)] $x_k \in I(m,j_{i_k})_0$,
\item[(iii)] $T_k$ is the canonical representative for $\phi_{i_k}(x_k)$, and 
\item[(iv)] $\M(T_k) \to \mathbf L(\Pi)$ as $k\to \infty$. 
\end{itemize}
\end{definition}

\begin{proposition}
Let $\Pi \in  \pi^\sharp_m(\Z_n(M,\bd M),\{0\})$ and let $\{\phi_i\}_{i\in \N}$ be a critical sequence for $\Pi$. Then the critical set $K(\{\phi_i\}_{i\in \N})$ is a non-empty, compact subset of $\mathcal V_n(M)$. 
\end{proposition} 

\begin{proof}
    The proof is essentially identical to \cite[4.2]{pitts2014existence}.
\end{proof}

\subsection{Discretization and interpolation}

For min-max theory in closed manifolds, Marques and Neves \cite{marques2014min} showed that the discrete mapping sequences and discrete homotopies of Almgren and Pitts \cite{pitts2014existence} can actually be replaced by continuous maps and continuous homotopies. To do this, they proved a discrete to continuous interpolation theorem that takes in a discrete map on $I(m,j)_0$ and produces an $\M$-continuous map on $I^m$ taking the same values on $I(m,j)_0$. They also proved a counterpart discretization theorem which takes in a continuous map and produces an $(m,\M)$-homotopy sequence.  Li and Zhou \cite{li2021min} proved corresponding interpolation and discretization theorems in the classical free boundary setting. In this section, we discuss the analogous interpolation and discretization theorems in locally wedge-shaped manifolds. 

We will be interested in maps $\Phi\colon I^m \to \mathcal Z_n(M,\bd M)$ which are continuous with respect to either the flat topology, the $\mathbf F$-topology, or the mass topology.  For technical reasons, we also assume that the maps in consideration have no concentration of mass. 

\begin{definition}
Let $\Phi\colon I^m \to \mathcal Z_n(M,\bd M)$ be a continuous map in the flat topology.   Then $\Phi$ has {\it no concentration of mass} provided one has
\[
\lim_{r\to 0} \sup \left\{\vert \| T(x)\|\vert (\mB^L_r(p)):\, p\in M, x\in I^m \right\} = 0,
\]
where $T(x)$ is the canonical representative for $\Phi(x)$.
\end{definition}

We can now state and prove the  interpolation and discretization theorems.

\begin{theorem}[Interpolation]\label{Thm: interpolation}
Let $M^{n+1}\subset \R^L$ be a locally wedge-shaped manifold. There exist constants $C_1 > 0$ and $\delta_1 > 0$, depending only on $m$ and $M\subset \R^L$, such that the following is true. Given any map $\phi\colon I(m,j)_0\to \Z_n(M,\bd M)$ with $\f(\phi) < \delta_0$, 
there exists an $\M$-continuous map $\Phi \colon I^m\to \Z_n(M,\bd M)$ such that 
\begin{itemize} 
\item[(i)] $\Phi(x) = \phi(x)$ for all $x\in I(m,j)_0$,
\item[(ii)] the values of $\Phi$ on any cell $\alpha \in I(m,j)$ are completely determined by the values of $\phi$ on the vertices of $\alpha$, and 
\item[(iii)] for any cell $\alpha\in I(m,j)$ and any $x,y\in \alpha$ we have $\M(\Phi(x)-\Phi(y)) \le C_0\f(\phi)$.  
\end{itemize}
\end{theorem}

\begin{proof}
By smoothing out the edge of $M$, we can fix once and for all a bilipschitz homeomorphism $f\colon (M,\bd M)\to (N,\bd N)$ where $(N,\bd N)$ is a smooth manifold with boundary and $\text{Lip}(f),\text{Lip}(f^{-1})\le 2$.  Let $C_0$ and $\delta_0$ be the constants from \cite[Theorem 4.14]{li2021min} applied to $N$.  Define $C_1 := 4^n C_0$ and $\delta_1 := 2^{-n} \delta_0$. 

Now consider a map $\phi\colon I(m,j)_0\to \Z_n(M,\bd M)$ with $\f(\phi)<\delta_1$.
Then the map $\theta = f_\sharp\circ \phi\colon I(m,j)_0 \to \Z_n(N,\bd N)$ satisfies $\f(\theta) \le 2^n \f(\phi) < \delta_0$.  Consequently by \cite{li2021min} Theorem 4.14, there exists an $\M$-continuous map $\Theta\colon I^m \to \Z_n(N,\bd N)$ such that 
\begin{itemize}
\item[(i')] $\Theta(x) = \theta(x)$ for all $x\in I(m,j)_0$,
\item[(ii')] the values of $\Theta$ on any cell $\alpha\in I(m,j)$ are completely determined by the values of $\theta$ on the vertices of $\alpha$, and 
\item[(iii')] for any cell $\alpha\in I(m,j)$ and any $x,y\in \alpha$ we have $\M(\Theta(x)-\Theta(y)) \le C_0 \f(\theta)$. 
\end{itemize}
Now define $\Phi = (f^{-1})_\sharp \circ \Theta\colon I(m,j)_0\to \Z_n(M,\bd M)$. Then $\Phi$ is $\M$-continuous and 
\begin{itemize}
\item[(i)] $\Phi(x) = \phi(x)$ for all $x\in I(m,j)_0$,
\item[(ii)] the values of $\Phi$ on any cell $\alpha \in I(m,j)$ are completely determined by the values of $\phi$ on the vertices of $\alpha$, and 
\item[(iii)] for any cell $\alpha\in I(m,j)$ and any $x,y\in \alpha$ we have $\M(\Phi(x)-\Phi(y)) \le 2^n C_0 \f(\theta) \le C_1\f(\phi)$.
\end{itemize}
Therefore $\Phi$ is as required. 
\end{proof}

\begin{theorem}[Discretization]\label{Thm: discretization} Assume that $\Phi\colon I^m \to \Z_n(M,\bd M)$ is an $\mathcal F$-continuous map with no concentration of mass. Further assume that $\Phi$ is $\mathbf F$-continuous on $\bd I^m$.  Then there exists an increasing sequence of integers $\{j_i\}_{i\in\N}$ and a decreasing sequence of positive numbers $\delta_i \to 0$ and  maps $\phi_i\colon I(m,j_i)\to Z_n(M,\bd M)$ such that  
\begin{itemize}
\item[(i)] $\{\phi_i\}_{i\in \N}$ is an $(m,\M)$-homotopy sequence of mappings into $(\Z_n(M,\bd M), \{0\})$,
\item[(ii)] $\f(\phi_i) < \delta_i$, 
\item[(iii)] there is an increasing sequence of integers $k_i$ such that 
\[
\M(\phi_i(x)) \le \sup\{ \M(\Phi(y)): \alpha \in I(m,k_i)_m \text{ and } x,y\in \alpha\} + \delta_i,
\] 
\item[(iv)] $\mathcal F(\phi_i(x) - \Phi(x)) < \delta_i$ for all $x\in I(m,j_i)_0$, and
\item[(v)] $\M(\phi_i(x)) \le \M(\Phi(x)) + \delta_i$ for all $x\in I(m,j_i)_0$. 
\end{itemize}
Moreover, if $\Phi$ is $\M$-continuous on the bottom face $I^{m-1}\times \{0\}$ then we can ensure $\phi_i(x,[0]) = \Phi(x,0)$ for all $x\in I(m-1,j_i)$. 
\end{theorem}

\begin{proof}
The proof is essentially identical to that of \cite[Theorem 4.12]{li2021min}. Note that Lemma \ref{Lem: F-isoperimetric lemma} shows the $\mathcal F$-isoperimetric Lemma is still valid in $M$. We can then emulate the cut and paste method in \cite[Lemma B.3 Case 1]{li2021min}, except replacing geodesic balls with Euclidean balls in the ambient $\R^L$ and replacing the intrinsic distance functions with Euclidean distance functions $d^{\R_L}_p$. This allows us to adapt the proof of Marques-Neves \cite[Lemma 13.4]{marques2014min}, following the outline put forth in \cite[Theorem 4.12]{li2021min}. The remainder of the proof then follows that of Marques-Neves \cite[Theorem 13.1]{marques2014min} as described in \cite[Theorem 4.12]{li2021min}.
\end{proof}

\subsection{Tightening}

Next, we describe a pull-tight procedure that can be used to ensure that all varifolds in the critical set are stationary with free boundary.  The construction is by now relatively standard and so we shall be somewhat brief. 

\begin{proposition}
\label{proposition:pull-tight}
Fix a homotopy class $\Pi \in \pi^\sharp_m(\Z_n(M,\bd M),\{0\})$. Let $\{\phi_i\}_{i\in \N}$ be a critical sequence for $\Pi$. Then there exists another critical sequence $\{\theta_i\}_{i\in \N}$ for $\Pi$ such that every varifold $V\in K(\{\theta_i\}_{i\in \N})$ is stationary in $M$ with free boundary. 
\end{proposition}

\begin{proof}
Equip $\mathfrak X_{\text{tan}}(M)$ with the $C^1$ topology and equip $\mathcal V(M)$ with the $\mathbf F$-topology. Choose a constant $C > 0$ such that $\M(\phi_i(x)) < C$ for all $i\in \N$ and all $x \in \text{dom}(\phi_i)$. Define 
\[
\mathcal V^C = \{V\in \mathcal V(M): \|V\|(M) \le C\}. 
\]
Let $\mathcal V_\infty = \{V\in \mathcal V^C: V\mbox{ is stationary in $M$ with free boundary}\}$ and for each $k\in \mathbb Z$ let 
\[
\mathcal V_k = \left\{V\in \mathcal V^C: \frac{1}{2^{k+1}} \le \mathbf F(V,\mathcal V_\infty) \le \frac{1}{2^{k-1}}\right\}. 
\] 
The sets $\mathcal V_\infty$ and $\mathcal V_k$ are all compact. 
Using this annular decomposition and a partition of unity argument, one obtains a continuous map $V\in \mathcal V^C \mapsto X_V \in \mathfrak X_{\text{tan}}(M)$ (cf. \cite[Theorem 4.3]{pitts2014existence})
such that 
\begin{itemize}
\item[(i)] $\|X_V\|_{\infty} \le 1$ for all $V\in \mathcal V^C$,
\item[(i)] $\delta V(X_V) < 0$ for all $V\in \mathcal V^C\setminus \mathcal V_\infty$. 
\end{itemize}
Let $H_V(x,t)$ denote the flow of $X_V$. Then one further obtains the existence of continuous functions $G\colon (0,\infty)\to (0,\infty)$ and $S\colon(0,\infty)\to(0,\infty)$ with $\lim_{t \to0} G(t) = 0$ and $\lim_{t\to 0} S(t) = 0$ such that 
\[
\|(H_V)_\sharp(V,S(\mathbf F(V,\mathcal V_\infty)))\|(M) \le  \|V\|(M) - G(\mathbf F(V,\mathcal V_\infty))
\]
for all $V\in \mathcal V^C$. See \cite{colding2003min} or \cite{marques2014min} for more details.

Now let $\phi_i \colon I(m,j_i)_0\to (\Z_n(M,\bd M),\{0\})$ be as in the statement of the proposition. Fix an integer $i$ large enough and apply the interpolation Theorem \ref{Thm: interpolation} to $\phi_i$ to obtain an $\mathbf M$-continuous map $\Phi_i\colon I^m \to \Z_n(M,\bd M)$ such that 
\begin{itemize}
\item[(i)] $\Phi_i(x) = \phi_i(x)$ for all $x\in I(m,j_i)_0$, and 
\item [(ii)] $
\sup_{x\in I^m} \M(\Phi_i(x)) \le \sup_{x\in I(m,j_i)_0} \M(\phi_i(x)) + C\f(\phi_i)$.
\end{itemize}
Define a map $\Psi_i\colon I^m\times I\to \Z_n(M,\bd M)$ by setting
\[
\Psi_i(x,t) = (H_{\vert T_i(x)\vert})_\sharp(\Phi_i(x),t S(\mathbf F(\vert T_i(x)\vert,\mathcal V_\infty)))
\]
where $T_i(x)$ is the canonical representative for $\Phi_i(x)$. We then take a suitably fine discretization $\psi_i\colon I(m,\ell_i)_0\times I(1,\ell_i)_0\to \Z_n(M,\bd M)$ of $\Psi_i$ (Theorem \ref{Thm: discretization}) and define $\theta_i \colon I(m,\ell_i)_0\to \Z_n(M,\bd M)$ by 
\[
\theta_i(x) = \begin{cases}
    0, &\text{if } x\in I_0(m,\ell_i)_0\\
    \psi_i(x,[1]), &\text{if } x\in I(m,\ell_i)_0 \setminus I_0(m,\ell_i)_0.
\end{cases}
\]
It can be shown that $\{\theta_i\}_{i\in \N}$ is the necessary critical sequence. See \cite{marques2014min} for more details. 
\end{proof}

\subsection{Existence of almost minimizing varifolds}
It remains to show that there is a critical sequence for which there is a varifold $V$ in the critical set satisfying a suitable almost-minimizing property.  

\begin{definition}\label{Def: amv in annuli}
    A varifold $V\in \mathcal V(M)$ is said to be {\it almost minimizing in annuli with free boundary} if for each $p\in M$ there is a number $r_{am}(p) > 0$ such that $V$ is almost minimizing in $\mathbb{A}^L_{s,t}(p) \cap M$ with free boundary for all $0 < s < t < r_{am}(p)$. 
\end{definition}

\begin{remark}
    Assume $V \in \mathcal V(M)$ is {\it almost minimizing in annuli with free boundary}. By shrinking $r_{am}(p)$ if necessary, we can assume that 
    \[
    \begin{cases}
        \mB^L_{r_{am}(p)}(p)\cap \bd M = \emptyset, &\text{whenever } p\in \text{int}(M),\\
        \mB^L_{r_{am}(p)}(p) \cap \bd^E M = \emptyset, &\text{whenever } p\in \bd^F M.
    \end{cases}
    \] 
    We will always assume that $r_{am}$ is chosen so that this is the case. 
\end{remark}

\begin{remark}
    Assume $V$ is almost minimizing in annuli with free boundary. Then 
    \begin{itemize}
        \item $V$ is almost minimizing with free boundary in small geodesic annuli centered at points $p\in \text{int}(M)$;
        \item $V$ is almost minimizing with free boundary in small Fermi annuli centered at points $p\in \bd^F M$;
        \item $V$ is almost minimizing with free boundary in small Fermi-type wedge-cut annuli at points $p\in \bd^E M$. 
    \end{itemize}
\end{remark}

The next theorem shows that it is always possible to find a critical sequence whose critical set contains a stationary varifold with free boundary that is also almost minimizing in annuli with free boundary. The proof is essentially identical to \cite[Theorem 4.21]{li2021min}

\begin{theorem}
\label{theorem:existence of almost minimizers}
Fix a homotopy class $\Pi \in \pi^\sharp_m(\Z_n(M,\bd M),\{0\})$. Then there exists a varifold $V\in \mathcal V_n(M)$ such that 
\begin{itemize}
    \item[(i)] $\|V\|(M) = \mathbf L(\Pi)$,
    \item[(ii)] $V$ is stationary in $M$ with free boundary,
    \item[(iii)] $V$ is almost minimizing in annuli with free boundary. 
\end{itemize}
\end{theorem}

\begin{proof}
    The theorem is trivial if $\mathbf L(\Pi) = 0$. So assume that $\mathbf L(\Pi) > 0$. Choose a critical sequence $\{\phi_i\}_{i\in \N}$. By the pull-tight argument Proposition \ref{proposition:pull-tight}, we can assume without loss of generality that every element in the critical set $K(\{\phi_i\}_{i\in \N})$ is stationary with free boundary. 
    
    It remains to show that some element of $K(\{\phi_i\}_{i\in \N})$ is almost minimizing in annuli with free boundary. Here we use the Almgren-Pitts combinatorial argument with minor modifications, as in Li-Zhou \cite[Theorem 4.21]{li2021min}. Suppose to the contrary that no element of $K(\{\phi_i\}_{i\in \N})$ is almost minimizing in annuli with free boundary. Then for each $V$ we can find a point $p_V$ in $M$ and a collection of $c$ disjoint annuli centered at $p_V$ such that $V$ is not almost minimizing with free boundary in any of these annuli. Here $c$ is the constant appearing in \cite[Theorem 4.10]{pitts2014existence}. We then proceed to follow the Almgren-Pitts combinatorial argument \cite[Theorem 4.10]{pitts2014existence}.  For annuli centered at points $p_V \in \text{int}(M)$ no modifications are necessary. For annuli centered at points $p\in \bd^F M \cup \bd^E M$, we note the following modifications. First, Theorem \ref{Thm: equivalent amv} can be used in place of \cite[Theorem 3.9]{pitts2014existence}. Second, we can use the $\M$-isoperimetric choice Lemma 5.3 in \cite[Part 5(c), Page 165]{pitts2014existence}. Finally, in \cite[Part 9, Page 166]{pitts2014existence} we can use the cut-and-paste method from \cite[Lemma B.3 Case 1]{li2021min}, except with ambient Euclidean balls in place of geodesic balls and the Euclidean distance function in place of the intrinsic distance function. This allows the sequence $\{\phi_i\}_{i\in \N}$ to be deformed into a new sequence $\{\theta_i\}_{i\in \N}$ representing the same homotopy class and satisfying $\mathbf L(\{\theta_i\}_{i\in \N}) < \mathbf L(\Pi)$. This is a contradiction. 
\end{proof}

To conclude this section, we note that non-trivial homotopy classes always exist. Indeed, the Almgren isomorphism theorem implies that $\pi_1^\sharp(\Z_n(M,\bd M),\{0\})$ is isomorphic to $H_{n+1}(M,\bd M)$. Choosing $\Pi$ to be such a non-trivial homotopy class, we have $\mathbf L(\Pi) > 0$. Applying the previous theorem to this homotopy class $\Pi$ gives the following corollary. 

\begin{corollary}
    Let $M\subset \R^L$ be a locally wedge-shaped manifold. Then there exists a non-zero varifold $V\in \mathcal V(M)$ which is stationary in $M$ with free boundary and almost minimizing in annuli with free boundary. 
\end{corollary}

\section{Regularity of almost minimizing varifolds}\label{Sec: min-max regualrity}

In this section, we prove the regularity theorem for the stationary almost minimizing varifold $V$ with free boundary constructed by the above min-max theory. 
The interior regularity ($\spt(\|V\|)\cap\interior(M)$) and the face regularity ($\spt(\|V\|)\cap\bd^FM$) of $V$ have been achieved in \cite{pitts2014existence}\cite{schoen1981regularity} and \cite{li2021min}. 
Specifically, we have the following regularity result. 

\begin{theorem}\label{Thm: min-max regularity away from edge}
    {(\rm \cite{pitts2014existence}\cite{schoen1981regularity}\cite{li2021min})}
    Let $M^{n+1}\subset \R^L$ be a locally wedge-shaped Riemannian manifold with dimension $3\leq n+1\leq 7$. 
    Suppose $V\in\V_n(M)$ is a stationary varifold in $M$ with free boundary and is almost minimizing in annuli with free boundary. 
    Then $\spt(\|V\|)\cap (M\setminus\bd^EM)$ is a smooth almost properly embedded FBMH $(\Sigma, \bd\Sigma) \subset (M\setminus\bd^E M, \bd M\setminus\bd^EM)$. 
    Moreover, there exist $\{n_i\}_{i=1}^N\subset \N$ so that 
    \[V\llcorner(M\setminus\bd^EM) = \sum_{i=1}^N n_i|\Sigma_i|,\]
    where $\{\Sigma_i\}_{i=1}^N$ are the connected components of $\Sigma$.  
\end{theorem}

By Theorem \ref{Thm: min-max regularity away from edge}, it is sufficient to consider the regularity of $V$ near the edge points $p\in\spt(\|V\|)\cap\bd^E M$. 
Indeed, we will show that the FBMH $\Sigma$ given in Theorem \ref{Thm: min-max regularity away from edge} can be $C^{2,\alpha}$ extended to the edge $\bd^EM$ as a $C^{2,\alpha}$-to-edge almost properly embedded locally wedge-shaped FBMH.

\subsection{Good replacement property}\label{Subsec: good replacement}

First of all, let us show the {\em good replacement property} for the min-max varifolds. 
Since we mainly focus on the regularity near $\bd^EM$ by Theorem \ref{Thm: min-max regularity away from edge}, we always assume the relative open subset $U\subset M$ satisfies $U\cap\bd^EM\neq\emptyset$ in the following constructions. 

\begin{proposition}\label{Prop: good replacement property}
    Given a relative open subset $U\subset M$ with $U\cap\bd^EM\neq\emptyset$, suppose $V\in\V_n(M)$ is almost minimizing in $U$ with free boundary. 
    Then for any compact subset $K\subset U$, there exists $V^*\in\V_n(M)$, called the {\em replacement of $V$ in $K$}, so that
    \begin{itemize}
        \item[(i)] $V\llcorner (M\setminus K) = V^*\llcorner(M\setminus K)$;
        \item[(ii)] $\|V\|(M) = \|V^*\|(M)$;
        \item[(iii)] $V^*$ is almost minimizing in $U$ with free boundary;
        \item[(iv)] $V^*=\lim_{i\to\infty}|T_i^*|$ as varifolds, where $T_i^*\in Z_n(M,\bd M)$ is locally mass minimizing in $\interior_M (K)$ (relative to $\bd M$). 
    \end{itemize}
    Moreover, if $2\leq n\leq 5$, $\interior(K)$ is simply connected, and $M^{n+1}$ satisfies (\ref{dag}) or (\ref{ddag}), then $V^*\llcorner\interior_M(K)$ is an integer rectifiable varifold whose support $\Sigma:=\spt(\|V^*\|)\cap \interior_M(K)$ is a $C^{2,\alpha}$-to-edge almost properly embedded locally wedge-shaped stable FBMH $(\Sigma,\{\bd_m\Sigma\})\subset (\interior_M(K),\{\interior_M(K)\cap \bd_{m+1}M\})$. 
\end{proposition}
\begin{proof}
    The proof is parallel to \cite[Proposition 5.3, 5.5]{li2021min}, and we shall include some details here for the sake of completeness. 

    {\bf Step 1.} {\em A constrained minimization problem.}
    
    Given $\epsilon,\delta >0$ and $\tau\in \mfa_n(U;\epsilon,\delta;\F)$, let $\C_\tau$ be the set of all $\sigma\in\Z_n(M,\bd M)$ so that there is a finite sequence $\{\tau_i\}_{i=0}^q\subset\Z_n(M,\bd M)$ satisfying $\tau_0=\tau$, $\tau_q=\sigma$, and for all $i=1,\dots,q$,
    \begin{equation}\label{Eq: replacement: finite sequence property}
        \spt(\tau - \tau_i)\subset K,\quad \F(\tau_i - \tau_{i-1})\leq \delta,\quad \M(\tau_i)\leq\M(\tau)+\delta.
    \end{equation}
    Then for any sequence $\{\sigma_i\}_{i\in\N}\subset\C_\tau$ with $\M(\sigma_i)\to \inf\{\M(\sigma):\sigma\in\C_\tau\}$, it follows from the compactness theorem (Lemma \ref{Lem: compactness of currents}) and the lower semi-continuous of mass (\ref{Eq: lower semi-continuous mass}) that $\sigma_i$ weakly converges (up to a subsequence) to some $\tau^*\in\Z_n(M,\bd M)$ with $\spt(\tau^*-\tau)\subset K$ and $\M(\tau^*)\leq \inf\{\M(\sigma):\sigma\in\C_\tau\}$. 
    Noting $\sigma_i\rightharpoonup \tau^*$ implies $\F(\sigma_i - \tau^*)<\delta$ for $i$ large, we can insert $\tau^*$ at the end of the finite sequence in the definition of $\sigma_i\in\C_\tau$ ($i$ large enough), which suggests $\tau^*\in\C_\tau$ is a mass minimizer: $\M(\tau^*) = \inf\{\M(\sigma):\sigma\in\C_\tau\}$. 

    In addition, suppose $\{\tau^*_i\}_{i=1}^q\subset\Z_n(M,\bd M)$ is a finite sequence satisfying $\tau^*=\tau_0^*$ and (\ref{Eq: replacement: finite sequence property}) with $\tau^*,\tau^*_i$ in place of $\tau,\tau_i$. 
    Then, after adding $\{\tau^*_i\}_{i=1}^q$ at the end of the finite sequence in the definition of $\tau^*\in \C_\tau$, we obtain $\M(\tau^*_q)\geq \M(\tau) - \epsilon \geq \M(\tau^*) -\epsilon$ since $\tau\in \mfa_n(U;\epsilon,\delta;\F)$ and $\tau^*$ is a mass minimizer in $\C_\tau$. 
    Hence, $\tau^*\in\mfa_n(U;\epsilon,\delta;\F)$. 

    Moreover, let $T^*$ be the canonical representative of $\tau^*$. 
    Then, after replacing the Fermi half-balls by $\mB^L_r(p)\cap M$, the arguments in \cite[Proposition 5.3, Claim 2]{li2021min} can be taken verbatim to show $T^*$ is locally mass minimizing in $\interior_M(K)$ (relative to $\bd M$). 

    {\bf Step 2.} {\em The existence of replacements.}
    
    By the almost minimizing assumption on $V$, we can take $\{\tau_i\}_{i\in\N}\subset \mfa_n(U;\epsilon_i,\delta_i;\F)$ with $\epsilon_i,\delta_i\to 0$ so that the canonical representatives $\{T_i\in\tau_i\}_{i\in\N}$ satisfy $|T_i|\to \tau$. 
    Then for each $i\in\N$, let $\tau_i^*\in\C_{\tau_i}$ be the mass minimizer given as in {\bf Step 1} so that $\tau_i^*\in \mfa_n(U;\epsilon_i,\delta_i;\F)$ and the canonical representative $T_i^*\in\tau_i^*$ is locally mass minimizing in $\interior_M(K)$. 
    Since $\M(T_i^*)=\M(\tau_i^*)$ is uniformly bounded, a subsequence $|T_i^*|$ converges as varifolds to some $V^*\in\V_n(M)$ satisfying (iii)(iv). 
    Finally, since $\tau_i^*\in\C_{\tau_i}$ and $\tau_i\in \mfa_n(U;\epsilon_i,\delta_i;\F)$, we have $\spt(T_i-T_i^*) = \spt(\tau_i - \tau_i^*)\subset K$ and $\M(\tau_i)-\epsilon_i\leq \M(\tau_i^*)\leq \M(\tau_i)$, which gives (i)(ii). 

    {\bf Step 3.} {\em The regularity of replacements.}
    
    By the additional assumptions, the manifold $M$ is locally wedge-shaped with wedge angle $\theta(p) \leq \pi/2$ for all $p\in\bd^EM$. 
    Therefore, using \cite[Theorem 1.2]{edelen2022regularity} and the statements after it, the local mass minimizers $\{T_i^*\}_{i\in\N}$ constructed in {\bf Step 2} are $C^{1,\alpha}$-to-edge, properly embedded, locally wedge-shaped FBMHs in $\interior_M(K)$ (with integer multiplicities) whose regularity can be further upgraded to $C^{2,\alpha}$-to-edge by Theorem \ref{Thm: regularity of FBMH}.  
    
    Next, we claim each $T_i^*$ is stable in $\interior_M(K)$ (Definition \ref{Def: stabe FBMH}). 
    Otherwise, combining the proper embeddedness of $T_i^*\llcorner \interior_M(K)$ with Definition \ref{Def: stabe FBMH}, there exists $X_i\in \mfX_{tan}(M)$ compact supported in $\interior_M(K)$ so that $\M(T_i^t)< \M(T_i^*)$ for $0\neq t\in (-t_0, t_0)$, where $T_i^t := (f_t^{X_i})_\#T_i^*$ and $\{f_t^{X_i}\}_{t\in(-t_0,t_0)}$ are the diffeomorphisms generated by $X_i$. 
    Then, we have $T_i^0=T_i^*$, and $T_i^t \equiv T_i^*$ in $M\setminus\interior_M(K)$ for all $t\in (t_0,t_0)$. 
    Note $T_i^t$ is also properly embedded in $\interior_M(K)$ for $t>0$ small enough. 
    Hence, for $t>0$ sufficiently small, $\tau_i^t:=[T_i^t]\in\Z_n(M,\bd M)$ has its canonical representative given by $T_i^t$, and satisfies 
    \[ \spt(\tau_i^t - \tau_i^*)\subset\spt(T_i^t - T_i^*)\subset K,\quad \F(\tau_i^t - \tau_{i}^*)\leq\F(T_i^t - T_{i}^*)\leq \delta,\quad \M(\tau_i^t) < \M(\tau^*).\]
    This suggests $\tau_i^t\in \C_{\tau_i^*}\subset\C_{\tau_i}$ for $t>0$ small, which contradicts the mass minimizing property of $\tau_i^*\in\C_{\tau_i}$ in {\bf Step 1}. 

    Finally, by the compactness theorem \ref{Thm: compactness for FBMHs}, we obtain the regularity of $V^*\llcorner\interior_M(K)$.
\end{proof}

\subsection{Tangent cones and rectifiability}\label{Subsec: classify tangent cone}

In this subsection, we use the good replacement property of min-max varifolds to classify their tangent cones and show the rectifiability. 
The proof shares the same spirit as in \cite[Section 5.2]{li2021min}, while we use the fact $\mH^n(\bd^E M) = 0$ (Lemma \ref{Lem: classify face and edge}) to simplify the arguments. 

\begin{lemma}[Rectifiability of tangent cones]\label{Lem: rectifiability of tangent cones}
    Let $2\leq n\leq 5$ and assume $M$ satisfies (\ref{dag}) or (\ref{ddag}). 
    Suppose $V\in\V_n(M)$ is almost minimizing in annuli with free boundary and is stationary in $M$ with free boundary. 
    Then for any $p\in\spt(\|V\|)\cap \bd^EM$ and $C\in\VarTan(V,p)$, $C\in\V_n(T_pM)$ is a rectifiable cone which is stationary in $T_pM$ with free boundary. 
\end{lemma}

\begin{proof}
    Let $r_i\to 0$ and $V_i:=(\bleta_{p,r_i})_\# V$ so that $V_i\to C\in \VarTan(V,p)$. 
    Note $\bleta_{p,r_i}(M)\to T_pM$ smoothly. 
    We have that $C\in\V_n(T_pM)$ is stationary in $T_pM$ with free boundary because the notion of stationary with free boundary is invariant under the map $\bleta_{p,r_i}:\R^L\to\R^L$. 
    Additionally, it follows from Theorem \ref{Thm: min-max regularity away from edge} that $C$ is integer rectifiable in $T_pM\setminus \bd^E T_pM$. 

    \begin{claim}\label{Claim: no measure on edge}
        $\|C\|(\bd^ET_pM) = 0$, and thus $C$ is rectifiable. 
    \end{claim}
    \begin{proof}[Proof of Claim \ref{Claim: no measure on edge}]
        By the monotonicity formula (Theorem \ref{Thm: monotonicity fomula}), we have
        \[\frac{\|C\|(\mB^L_r(0))}{\omega_n r^n} = \lim_{i\to\infty} \frac{\|V_i\|(\mB^L_r(0))}{\omega_n r^n} = \lim_{i\to\infty}\frac{\|V\|(\mB^L_{rr_i}(p))}{\omega_n (rr_i)^n} = \Theta^n(\|V\|, p) <\infty \quad \forall r>0.\]
        For any $y\in \spt(\|C\|)\cap \bd^ET_pM$ and $0<r<R$, the monotonicity formula also gives
        \begin{eqnarray*}
            \frac{\|C\|(\mB^L_r(y))}{\omega_n r^n} \leq \frac{\|C\|(\mB^L_R(y))}{\omega_n R^n} \leq \frac{\|C\|(\mB^L_{R+|y|}(0))}{\omega_n (R+|y|)^n} \left( 1+\frac{|y|}{R} \right)^n = \Theta^n(\|V\|, p)\left( 1+\frac{|y|}{R} \right)^n
        \end{eqnarray*}
        After taking $r\to 0$ and $R\to\infty$ (Remark \ref{Rem: monotonicity radius}), we see $\Theta^n(\|C\|, y) \leq \Theta^n(\|C\|, 0) = \Theta^n(\|V\|, p)$, which suggests $\|C\|(\bd^ET_pM) = 0$ by $\mH^n(\bd^ET_pM)=0$ (Lemma \ref{Lem: classify face and edge}). 
    \end{proof}
    
    Finally, note the direction vector field $\vec{x}\in \mfX_{tan}(T_pM)$. 
    The arguments in \cite[Theorem 19.3]{simon1983lectures} therefore carry over in the Euclidean wedge domain $T_pM$ to show $C$ is a cone. 
\end{proof}

\begin{remark}\label{Rem: rectifiability of blowup}
    Suppose $C=\lim_{p,r_i}\bleta_{p,r_i}V_i$ for some $r_i\to 0$ and $\{V_i\}_{i\in\N}\subset\V_n(M)$ which are almost minimizing in annuli and stationary in $M$ with free boundary with $\sup\|V_i\|(M)<\infty$. 
    Then, after modifying the first formula in the proof of Claim \ref{Claim: no measure on edge} by 
    \[ \lim_{i\to\infty}\frac{\|V_i\|(\mB^L_{rr_i}(p))}{\omega_n (rr_i)^n} \leq \lim_{i\to\infty} C_{mono}\frac{\|V_i\|(\mB^L_{r_{mono}}(p))}{\omega_n r_{mono}^n} < C_0,\]
    one can take the arguments almost verbatim from Lemma \ref{Lem: rectifiability of tangent cones} to show such a varifold $C$ is also stationary rectifiable in $T_pM$ with free boundary. 
\end{remark}

Next, we show the classification of tangent cones for the min-max varifold $V$, which further suggests the rectifiability of $V$. 

\begin{proposition}[Classification of tangent cones]\label{Prop: classify tangent cones}
    Let $2\leq n\leq 5$ and assume $M$ satisfies (\ref{dag}) or (\ref{ddag}). 
    Suppose $V\in\V_n(M)$ is almost minimizing in annuli with free boundary and is stationary in $M$ with free boundary. 
    Then for any $p\in\spt(\|V\|)\cap \bd^EM$ and $C\in\VarTan(V,p)$, either 
    \begin{itemize}
        \item[(i)] $C= 2\Theta^n(\|V\|, p) |S^+|$ and $\theta(p)=\frac{\pi}{2}$ where $2\Theta^n(\|V\|, p)\in\N$ and $S^+\in \{ P_{\pm\frac{\pi}{4}}(p) \}$ is a half $n$-plane in $T_pM$ perpendicular to $\bd^{\mp}T_pM$; or
        \item[(ii)] $C = \frac{2\pi}{\theta(p)}\Theta^n(\|V\|, p)|S\cap T_pM|$ where $\frac{2\pi}{\theta(p)}\Theta^n(\|V\|, p)\in \N$ and $S\in\mathcal{P}_{T_pM}$ is a horizontal hyperplane of $T_pM$ (see Definition \ref{Def: vertical/horizontal hyperplane in manifold}).
    \end{itemize}
    Moreover, $\|V\|(\bd^E M) = 0$, and thus $V$ is rectifiable. 
\end{proposition}

\begin{proof}
    Let $r_i\to 0$ and let $C=\lim_{i\to\infty}(\bleta_{p,r_i})_\# V$ be a stationary rectifiable cone in $T_pM$ with free boundary (Lemma \ref{Lem: rectifiability of tangent cones}). 
    Fix any $\sigma\in (0,1/4)$. 
    Then for each $i$ large enough, we take $V_i^*$ as the replacement of $V$ in $K_i=\Clos(\mAn^L_{\sigma r_i, r_i}(p))\cap M$ given by Proposition \ref{Prop: good replacement property}. 
    Note 
    \begin{equation}\label{Eq: non-empty replacement}
        V_i^*\llcorner (\mAn^L_{\sigma r_i, r_i}(p)\cap M) \neq 0.
    \end{equation}
    Otherwise, there is a smallest $r>0$ so that $V_i^*\llcorner (\mB^L_{r_i}\cap M) = V_i^*\llcorner\Clos(\wB_r(p))$ and $\spt(\|V_i^*\|)\cap \bd_{rel}\wB_r(p)\neq \emptyset$, which contradicts the maximum principle (Theorem \ref{Thm: maximum principle for varifolds}). 
    Next, combining Proposition \ref{Prop: good replacement property}(ii)(iii) with Lemma \ref{Lem: amv imply stationary}, we conclude that $V_i^*$ is also almost minimizing in annuli and stationary in $M$ with free boundary so that $\sup_{i\in\N}\|V_i^*\| <\infty$, which further implies $(\bleta_{p,r_i})_\#V_i^*$ converges to some rectifiable stationary varifold $D\in\V_n(T_pM)$ with free boundary by Remark \ref{Rem: rectifiability of blowup}. 
    In addition, 
    \[ \|D\|(\mB^L_2(0)) = \|C\|(\mB^L_2(0)) \quad{\rm and}\quad \spt(\|D-C\|)\subset \Clos(\mAn^L_{\sigma,1}(0)) \] 
    by Proposition \ref{Prop: good replacement property}(i)(ii). 
    Together with the monotonicity formula (Theorem \ref{Thm: monotonicity fomula}) and $\|C\|(\mB^L_r(0))/(\omega_n r^n) \equiv \Theta^n(\|V\|,p)$ for all $r>0$ (Lemma \ref{Lem: rectifiability of tangent cones}), one can see that $D$ is also a rectifiable cone in $T_pM$, and thus $C=D$. 

    Next, for each $i$ large enough, denote by $\Sigma_\infty :=\spt(\|D\|)\cap\mAn^L_{\sigma,1}(0)$, $\Sigma_i := \spt((\bleta_{p,r_i})_\#V_i^*)\cap \mAn^L_{\sigma,1}(0)$, and $M_i:=\bleta_{p,r_i}(M)$. 
    Then by the last statement in Proposition \ref{Prop: good replacement property}, $(\Sigma_i, \{\bd_m\Sigma_i\}) \subset (\mAn_{\sigma,1}(0)\cap M_i, \{\mAn_{\sigma,1}(0) \cap \bd_{m+1}M_i\})$ is a $C^{2,\alpha}$-to-edge almost properly embedded locally wedge-shaped FBMH. 
    Additionally, 
    \[\mH^n(\Sigma_i) \leq \frac{\|V_i^*\|(\mB^L_{r_i}(p))}{r_i^n} = \frac{\|V\|(\mB^L_{r_i}(p))}{r_i^n} \leq C_{mono}\frac{\|V\|(M)}{r_{mono}^n} < \infty,\]
    by Proposition \ref{Prop: good replacement property}(i)(ii) and the monotonicity formula Theorem \ref{Thm: monotonicity fomula}. 
    It now follows from the compactness theorem \ref{Thm: compactness for FBMHs} that $(\Sigma_\infty, \{\bd_m\Sigma_\infty\}) \subset (\mAn_{\sigma,1}(0)\cap T_pM, \{\mAn_{\sigma,1}(0) \cap \bd_{m+1}T_pM\})$ is a $C^{2,\alpha}$-to-edge almost properly embedded locally wedge-shaped stable FBMH. 
    
    Therefore, noting $\sigma\in (0,1/4)$ is arbitrary, $C=D$ is an almost properly embedded locally wedge-shaped stable minimal cone (with integer multiplicity) in a non-obtuse Euclidean wedge domain $T_pM$ which is $C^{2,\alpha}$ except possibly at the vertex. 
    We claim $C$ is flat by a Bernstein-type theorem, and thus $C$ must be type (i) or (ii). 
    
    Indeed, if $\Sigma:=\spt(\|C\|)$ is not properly embedded, then the maximum principle for FBMHs indicates $\Sigma$ is flat (cf. \cite[Proposition 5.1, 5.2]{mazurowski2023curvature}). 
    Suppose next $\Sigma$ is properly embedded. 
    Then the unit normal vector field of $\Sigma\setminus\{0\}$ is in $\mfX_{tan}(T_pM\setminus\{0\})$. 
    Additionally, the boundary integral in the second variation formula (\ref{Eq: 2nd variation formula}) (used in \cite[(3.5)]{schoen1975curvature}) vanishes since both faces $\bd^{\pm}T_pM$ are flat. 
    Moreover, $\langle \grad \vert A_\Sigma\vert,\eta\rangle(p) =0$ for all $p\in\bd^F\Sigma\setminus\{0\}$ by the local reflection technique in the Bernstein-type theorem \cite[Theorem 5.3]{mazurowski2023curvature}, where $\eta$ is the co-normal of $\Sigma$ along $\bd^F\Sigma\setminus\{0\}$ and $A_\Sigma$ is the second fundamental form of $\Sigma$. 
    This cancels out the boundary term in the integral of Simon's inequality by parts (\cite[(3.4)]{schoen1975curvature}). 
    Hence, for any $C^1$ function $f$ compactly supported in $\Sigma\setminus (\bd^E\Sigma\cup\{0\})$, we can follow the proof in \cite[\S 3]{schoen1975curvature} with the above observations, and obtain the estimate \cite[(3.6))]{schoen1975curvature} for such $\Sigma$ and $f$. 
    Noting $\dim(\bd^E\Sigma) = n-2$, a standard logarithmic cut-off trick implies \cite[(3.6))]{schoen1975curvature} is also valid for all $C^1$ function $f$ compactly supported in $\Sigma\setminus\{0\}$. 
    Thus, the proof in \cite[\S 3]{schoen1975curvature} would carry over to show $\Sigma$ is flat. 

    Finally, we assume by contradiction that $\|V\|(\bd^EM)> 0$. 
    Then there exists $p\in\spt(\|V\|)\cap\bd^E M$ with $\Theta^{*n}(\|V\|\llcorner\bd^EM, p)>0$. 
    However, we can also take $r_i\to 0$ so that $C=\lim_{i\to\infty}(\bleta_{p,r_i})_\# V \in \VarTan(V,p)$ and
    \[ \Theta^{*n}(\|V\|\llcorner\bd^E M, p) = \lim_{r_i\to 0}\frac{\|V\|(\mB^L_{r_i}(p)\cap\bd^E M)}{\omega_n r_i^n} =\|C\|(\mB^L_1(0)\cap T_p(\bd^E M)) =0\]
    by (i)(ii), which is a contradiction. 
    And thus, $V$ is rectifiable by Theorem \ref{Thm: min-max regularity away from edge}. 
\end{proof}

\subsection{Regularity of almost minimizing varifolds}\label{Subsec: main regularity}

Now, we can state and prove our main regularity theorem for min-max varifolds near edge points of $M$. 
Note the dimension of $M$ is at most $6$ in Theorem \ref{Thm: main regularity} instead of at most $7$ as in Theorem \ref{Thm: min-max regularity away from edge}. 
This is because the Bernstein-type Theorem \cite[Theorem 1.2]{mazurowski2023curvature} and the curvature estimates \cite[Theorem 1.3]{mazurowski2023curvature} the authors found in wedge domains only hold in the $3\leq n+1\leq 6$ scenario. 

\begin{theorem}\label{Thm: main regularity}
    Let $3\leq n+1\leq 6$, and $M^{n+1}\subset \R^L$ be a locally wedge-shaped Riemannian manifold satisfying (\ref{dag}) or (\ref{ddag}). 
    Suppose $V\in\V_n(M)$ is a varifold stationary in $M$ with free boundary and is almost minimizing in annuli with free boundary. 
    Then $\spt(\|V\|)$ is a $C^{2,\alpha}$-to-edge almost properly embedded locally wedge-shaped FBMH $(\Sigma, \{\bd_m\Sigma\}) \subset (M, \{\bd_{m+1} M\})$. 
    Moreover, there exists $N\in\N$ and $\{n_i\}_{i=1}^N\subset \N$ so that 
    \[V = \sum_{i=1}^N n_i|\Sigma_i|,\]
    where $\{\Sigma_i\}_{i=1}^N$ are the connected components of $\Sigma$.  
\end{theorem}
\begin{proof}
    By Theorem \ref{Thm: min-max regularity away from edge}, we only consider the regularity of $V$ near $p\in\spt(\|V\|)\cap\bd^EM$. 
    For any such $p$, take $r>0$ sufficiently small so that $r<\frac{1}{4}\min\{r_{am}(p), r_{mono}, r_0\}$, where $r_{am}(p), r_{mono},r_0$ are given by Definition \ref{Def: amv in annuli}, Theorem \ref{Thm: monotonicity fomula} and Lemma \ref{Lem: Fermi-type distance function} respectively. 
    Using the maximum principle (Theorem \ref{Thm: maximum principle for varifolds}) as in (\ref{Eq: non-empty replacement}), we have
    \begin{equation}
        \emptyset\neq \spt(\|W\|)\cap \wS_t(p) = \Clos\big(\spt(\|W\|)\setminus\Clos(\wB_t(p))\big) \cap \wS_t(p)
    \end{equation}
    for any $t\in (0,r)$ and $W\in\V_n(M)$ that is stationary in $\wB_t(p)$ with free boundary and $W\llcorner\wB_t(p)\neq 0$. 

    {\bf Step1.} {\em Successive replacements on overlapping annuli.}

    Fix any $0<s<t<r$. 
    Then by Proposition \ref{Prop: good replacement property}, we can take $V^*$ as the replacement of $V$ in $\Clos(\wA_{s,t}(p))$ so that 
    \[\Sigma_1:=\spt(\|V^*\|)\cap\wA_{s,t}(p)\] 
    is a $C^{2,\alpha}$-to-edge almost properly embedded locally wedge-shaped FBMH $(\Sigma_1,\{\bd_m\Sigma_1\})\subset (\wA_{s,t}(p), \{\wA_{s,t}(p)\cap\bd_{m+1}M\})$ which is stable in $\wA_{s,t}(p)$. 
    By Sard's Theorem, there exists $s_2\in(s,t)$ so that $\wS_{s_2}(p)$ is transversal to $\Sigma_1$ (even at $\bd^F\Sigma_1$ and $\bd^E\Sigma_1$). 
    Then for any $s_1\in (0,s)$, we apply Proposition \ref{Prop: good replacement property} again to get a replacement $V^{**}$ of $V^*$ in $\Clos(\wA_{s_1,s_2}(p))$, whose support 
    \[\Sigma_2:=\spt(\|V^{**}\|)\cap\wA_{s_1,s_2}(p)\]
    is a $C^{2,\alpha}$-to-edge almost properly embedded locally wedge-shaped stable FBMH in $\wA_{s_1,s_2}(p)$. 

    {\bf Step 2.} {\em Gluing $\Sigma_1$ and $\Sigma_2$ across $\wS_{s_2}(p)$.}

    Denote by $\Gamma:=\Sigma_1\cap \wS_{s_{2}}(p)$. 
    Then by the transversal assumption, we see $(\Gamma, \{\bd_m\Gamma\}\subset (\wS_{s_{2}}(p), \{\bd_{m+1}\wS_{s_{2}}(p)\})$ is a $C^{2,\alpha}$-to-edge almost properly embedded locally wedge-shaped $(n-1)$-submanifold. 
    Note in addition that the transversality also implies $\Gamma$ is not contained entirely $\bd^E M $. 
    Take any $q\in \Gamma\setminus\bd^E M$. 
    Then by the smooth gluing arguments in \cite[Theorem 4]{schoen1981regularity} (for $p\in \interior(M)$) and \cite[Theorem 5.2, Step 2]{li2021min} (for $p\in\bd^F M$), $\Sigma_1$ coincides with $\Sigma_2$ in a small neighborhood of $q$. 
    Combining the unique continuation of FBMHs and the fact that $\Clos(\Sigma_i) = \Clos(\Sigma_i\setminus\bd^E M)$ for $i\in \{1,2\}$, we can now conclude 
    \begin{equation}\label{Eq: gluing support}
        \Sigma_1\llcorner\wA_{s,s_2}(p) = \Sigma_2\llcorner\wA_{s,s_2}(p).
    \end{equation}
    Hence, $\Sigma_1$ and $\Sigma_2$ can be glued together as a $C^{2,\alpha}$-to-edge almost properly embedded locally wedge-shaped FBMH in $\wA_{s_1,t}(p)$. 

    Moreover, for any $q\in \Gamma$, it follows from Proposition \ref{Prop: good replacement property}(i) and the classification of tangent cones (see Proposition \ref{Prop: classify tangent cones}, \cite[Proposition 5.10]{li2021min} and \cite[Theorem 7.8]{pitts2014existence} for $q\in\bd^EM,\bd^FM,\interior(M)$ respectively) that 
    \begin{equation}\label{Eq: gluing multiplicity}
        \Theta^n(\|V^{**}\|,q) = \Theta^n(\|V^*\|,q) = 
        \left\{ 
	    \begin{array}{ll}
	    	\frac{\theta(q)}{2\pi} k  & {\rm if~} q\in\bd^E\Sigma_1,
	    	\\
	    	\frac{1}{2}k & {\rm if~}q\in\bd^F\Sigma_1,
            \\
            k & {\rm if~}q\in\interior(\Sigma_1),
	    \end{array}
	    \right.
    \end{equation}
    where $k\in\N$ is the multiplicity of $V^*$ near $q$. 
    Combining (\ref{Eq: gluing support})(\ref{Eq: gluing multiplicity}) with the regularity of replacements (Proposition \ref{Prop: good replacement property}), we conclude 
    $V^*\llcorner \wA_{s,t}(p) = V^{**}\llcorner\wA_{s,t}(p)$.

    {\bf Step 3.} {\em Extending the replacements to a FBMH in a punctured ball center at $p$.}

    Denote by $V^{**}_{s_1}$ the second replacement in {\bf Step 1} with respect to the inner radius $s_1\in (0,s)$, and by $\Sigma_{s_1}:=\spt(\|V^{**}_{s_1}\|)\cap \wA_{s_1,s_2}(p)$. 
    Then for any $0<s_1'<s_1<s$, noting $\Sigma_{s_1} = \Sigma_1$ in $\wA_{s,s_2}(p)$, we can use the unique continuation arguments as in (\ref{Eq: gluing support}) to see $\Sigma_{s_1'}=\Sigma_{s_1}$ in $\wA_{s_1,s_2}(p)$. 
    Therefore, 
    \[\widetilde{\Sigma} := \cup_{0<s_1<s}\Sigma_{s_1}\]
    is a $C^{2,\alpha}$-to-edge almost properly embedded locally wedge-shaped FBMH $(\widetilde{\Sigma},\{\bd_m\widetilde{\Sigma}\})\subset (\wB_{s_2}(p)\setminus\{p\},\{\bd_{m+1}M\cap \wB_{s_2}(p) \setminus\{p\}\})$ which is also stable in $\wB_{s_2}(p)\setminus\{p\}$. 

    By Proposition \ref{Prop: good replacement property}(ii), $V^{**}_{s_1}$ has uniformly bounded mass for $s_1\in (0,s)$, which implies 
    \[\widetilde{V} = \lim_{s_1\to 0} V^{**}_{s_1}\]
    for some stationary varifold $\widetilde{V}\in\V_n(M)$ with free boundary. 
    In addition, by the construction of $\widetilde{\Sigma}$, we see $\spt(\|V^{**}_{s_1}\|) = \widetilde{\Sigma}$ in $\wA_{s_1,s_2}(p)$ for any $s_1\in (0,s)$. 
    Hence, 
    using the compactness theorem \ref{Thm: compactness for FBMHs} and $\|V_{s_1}^{**}\|(\wB_{s_1}(p)) = \|V\|(\wB_{s_1}(p))\to 0$ (as $s_1\to 0$), 
    we have $\widetilde{V}\llcorner(\wB_{s_2}(p)\setminus\{p\})$ is induced by $\widetilde{\Sigma}$ with integer multiplicity. 

    {\bf Step 4.} {\em Remove the singularity of $\widetilde{V}$ at $p$.}

    Note $\widetilde{V}$ is stationary in $M$ with free boundary, and $\widetilde{V}\llcorner(\wB_{s_2}(p)\setminus\{p\})$ is a $C^{2,\alpha}$-to-edge almost properly embedded locally wedge-shaped stable FBMH $\widetilde{\Sigma}$ with integer multiplicity. 
    Hence, one can use the blowup arguments as in Lemma \ref{Lem: rectifiability of tangent cones} and Proposition \ref{Prop: classify tangent cones} (without taking replacements) to show $C\in \VarTan(\widetilde{V},p)$ is in the form of (i) or (ii) in Proposition \ref{Prop: classify tangent cones}. 
    Moreover, we have the following strengthened classification for $\VarTan(\widetilde{V},p)$.

    \begin{claim}\label{Claim: strengthened classify tangent cones}
        Either 
        \begin{itemize}
            \item[(1)] $\VarTan(\widetilde{V},p) = \{ 2\Theta^n(\|\widetilde{V}\|, p) |P_{\pi/4}(p)| \} $ with $\theta(p) = \pi/2$; or
            \item[(2)] $\VarTan(\widetilde{V},p) = \{ 2\Theta^n(\|\widetilde{V}\|, p) |P_{-\pi/4}(p)| \} $ with $\theta(p) = \pi/2$; or
            \item[(3)] $\VarTan(\widetilde{V},p) = \{ \frac{2\pi}{\theta(p)} \Theta^n(\|\widetilde{V}\|, p) |S\cap T_pM| : S\in \mathcal{P}_{T_pM} \} $.
        \end{itemize}
    \end{claim}
    \begin{proof}[Proof of Claim \ref{Claim: strengthened classify tangent cones}]
        We prove by contradictions and separate the arguments into three cases. 
        
        {\bf Case 1:} {\em There exist $C_1,C_2\in \VarTan(\widetilde{V},p)$ in the form of (1)(3) respectively.}
        
        Let $r_i\to 0$ and $r_i'\to 0$ so that 
        \begin{itemize}
            \item $C_1=\lim_{i\to\infty}(\bleta_{p,r_i})_\# \widetilde{V} = 2\Theta^n(\|\widetilde{V}\|, p) |P_{\pi/4}(p)| = 2\Theta^n(\|\widetilde{V}\|, p) |\bd^+T_pM|$;
            \item $C_2=\lim_{i\to\infty}(\bleta_{p,r_i'})_\# \widetilde{V} = \frac{2\pi}{\theta(p)} \Theta^n(\|\widetilde{V}\|, p) |S\cap T_pM|$,
        \end{itemize}
        for some $S\in \mathcal{P}_{T_pM} $. 
        Up to a subsequence, we can assume 
        \[2r_{i+1}'<r_i< \frac{1}{2}r_i',\qquad\forall i\in\N.\]
        For $i$ large, by the stability of $\widetilde{\Sigma}$ in $\wB_{s_2}(p)\setminus\{p\}$ and the compactness theorem \ref{Thm: compactness for FBMHs}, 
        \begin{itemize}
            \item in $\mAn^L_{r_i,2r_i}(p)\cap M$, $\bd^F\widetilde{\Sigma} \subset \bd^-M_{p,r}$ and $\widetilde{\Sigma}$ is orthogonal to $\bd^-M_{p,r}$ along $\bd^F\widetilde{\Sigma}$;
            \item in $\mAn^L_{r_i',2r_i'}(p)\cap M$, $\bd^F\widetilde{\Sigma} \cap (\bd^{\pm}M_{p,r}\setminus\bd^E M) \neq\emptyset$ and $\widetilde{\Sigma}$ is orthogonal to $\bd^{\pm}M_{p,r}$ along $\bd^F\widetilde{\Sigma} \cap (\bd^{\pm}M_{p,r}\setminus\bd^E M)$.
        \end{itemize} 
        Next, we take
        \[R_i:= \sup\{R>0: \bd^F\widetilde{\Sigma} \subset \bd^-M_{p,r}\mbox{ in } \mAn^L_{r_i,R}(p)\cap M\},\]
        which clearly satisfies $2r_i\leq R_i\leq 2r_i'$. 
        Additionally, we also have $R_i/r_i \to \infty$. Otherwise, $\lim_{i\to\infty}\bleta_{p,r_i}(\widetilde{\Sigma})$ will be a stable FBMH $\spt(\|C_1\|)$ meeting $\bd^+T_pM$ orthogonally at some $q\in \bd^F\spt(\|C_1\|)$, which contradicts the choices of $\{r_i\}_{i\in\N}$ and $C_1$. 

        Now, consider $C:=\lim_{i\to\infty}(\bleta_{p,R_i})_\#\widetilde{V}\in\VarTan(\widetilde{V},p)$ (up to a subsequence). 
        Then by Theorem \ref{Thm: compactness for FBMHs} and the choice of $R_i$, 
        \begin{itemize}
            \item in $\mAn^L_{1/2,1}(p)\cap T_pM$, $\spt(\|C\|)$ is orthogonal to $\bd^-T_pM$ along $\bd^F\spt(\|C\|)\subset \bd^-T_pM$;
            \item in $\mAn^L_{1,2}(p)\cap T_pM$, there exists $q\in \bd^F\spt(\|C\|)\cap \bd^+T_pM$ so that $\spt(\|C\|)$ is orthogonal to $\bd^+T_pM$ at $q$,
        \end{itemize}
        which contradicts all the possible choices of $C\in \VarTan(\widetilde{V},p)$ in Proposition \ref{Prop: classify tangent cones}(i)(ii).
        
        {\bf Case 2:}{\em There exist $C_1,C_2\in \VarTan(\widetilde{V},p)$ in the form of (2)(3) respectively.}

        Similar to {\bf Case 1} with $+,-$ replaced by $-,+$.

        {\bf Case 3:}{\em There exist $C_1,C_2\in \VarTan(\widetilde{V},p)$ in the form of (1)(2) respectively.}

        The proof is also similar to {\bf Case 1}. First, let $r_i\to 0$ and $r_i'\to 0$ so that 
        \begin{itemize}
            \item $C_1=\lim_{i\to\infty}(\bleta_{p,r_i})_\# \widetilde{V} = 2\Theta^n(\|\widetilde{V}\|, p) |P_{\pi/4}(p)| = 2\Theta^n(\|\widetilde{V}\|, p) |\bd^+T_pM|$;
            \item $C_2=\lim_{i\to\infty}(\bleta_{p,r_i'})_\# \widetilde{V} = 2\Theta^n(\|\widetilde{V}\|, p) |P_{-\pi/4}(p)| = 2\Theta^n(\|\widetilde{V}\|, p) |\bd^-T_pM|$.
        \end{itemize}
        After passing to a subsequence with $2r_{i+1}'<r_i< \frac{1}{2}r_i'$, we similarly have
        \begin{itemize}
            \item in $\mAn^L_{r_i,2r_i}(p)\cap M$, $\bd^F\widetilde{\Sigma} \subset \bd^-M_{p,r}$ and $\widetilde{\Sigma}$ is orthogonal to $\bd^-M_{p,r}$ along $\bd^F\widetilde{\Sigma}$;
            \item in $\mAn^L_{r_i',2r_i'}(p)\cap M$, $\bd^F\widetilde{\Sigma} \subset \bd^+M_{p,r}$ and $\widetilde{\Sigma}$ is orthogonal to $\bd^+M_{p,r}$ along $\bd^F\widetilde{\Sigma}$.
        \end{itemize} 
        Next, we take
        \[R_i:= \sup\{R>0: \mbox{in  $\mAn^L_{r_i,R}(p)\cap M$, $\widetilde{\Sigma}$ is orthogonal to $\bd^-M_{p,r}$ along $\bd^F\widetilde{\Sigma}\subset \bd^-M_{p,r}$}\}.\]
        As before, we have $2r_i\leq R_i\leq 2r_i'$ and $R_i/r_i \to \infty$.
        Finally, consider $C:=\lim_{i\to\infty}(\bleta_{p,R_i})_\#\widetilde{V}\in\VarTan(\widetilde{V},p)$. 
        A contradiction then follows from the same argument as in {\bf Case 1}. 
    \end{proof}
    
    If we are in the case (1) or (2) of Claim \ref{Claim: strengthened classify tangent cones}, then the $C^\infty$-regularity of $\widetilde{V}$ follows from the arguments in \cite[Theorem 5.2, Step4]{li2021min}. 

    Next, suppose the case (3) of Claim \ref{Claim: strengthened classify tangent cones} occurs. 
    Then by Proposition \ref{Prop: classify tangent cones}, we see 
    \[\frac{2\pi}{\theta(p)}\Theta^n(\|\widetilde{V}\|,p) = k \in \N.\] 
    For any $r_i\to 0$, it follows from the compactness theorem \ref{Thm: compactness for FBMHs} that 
    \begin{equation}\label{Eq: min-max regularity: blowup convergence}
        \bleta_{p,r_i}(\widetilde{\Sigma}) \to k |S\cap T_pM|
    \end{equation}
    locally uniformly and $C^{2,\alpha'}$-to-edge in $\R^L\setminus\{0\}$, where $S\in \mathcal{P}_{T_pM}$ is a horizontal $n$-plane. 
    Hence, $\widetilde{\Sigma}$ is properly embedded in a punctured neighborhood of $p$. 
    Also, by the locally uniformly $C^{2,\alpha'}$-to-edge convergence (\ref{Eq: min-max regularity: blowup convergence}), we can take any $\sigma>0$ small enough so that 
    \begin{equation}\label{Eq: min-max regularity: graphs representation}
        \widetilde{V}\llcorner\mAn^L_{\sigma/2,\sigma}(p) = \sum_{i=1}^{l(\sigma)} k_i(\sigma)|\widetilde{\Sigma}_i(\sigma)|,
    \end{equation}
    where $\widetilde{\Sigma}_i(\sigma)$ is a connected graph over $S\cap \mAn^L_{\sigma/2,\sigma}(p)$, and $l(\sigma), \{k_i(\sigma)\}_{i=1}^{l(\sigma)}$ are positive integers satisfying 
    \begin{equation}\label{Eq: min-max regularity: multiplicity sum}
        \sum_{i=1}^{l(\sigma)} k_i(\sigma) = k. 
    \end{equation}
    By the continuity of $\widetilde{\Sigma}$, we see from (\ref{Eq: min-max regularity: graphs representation}) that $l(\sigma)$ and $\{k_i(\sigma)\}_{i=1}^{l(\sigma)}$ are constants independent on $\sigma$ (for $\sigma>0$ sufficiently small). 
    Therefore, after fixing $\sigma>0$ small enough, each $\widetilde{\Sigma}_i(\sigma)$ can be continued to 
    $\widetilde{\Sigma}_i$ in $(\mB^L_{\sigma}(p)\setminus\{p\})\cap M$ as a $C^{2,\alpha}$-to-edge properly embedded locally wedge-shaped stable FBMH. 
    It then follows from a standard extension argument (cf. the proof in \cite[Theorem 4.1]{harvey1975extending}) that each $\widetilde{\Sigma}_i$ can be extended as a stationary varifold in $\mB^L_{\sigma}(p)\cap M$ with free boundary. 
    By the compactness theorem \ref{Thm: compactness for FBMHs} and Claim \ref{Claim: strengthened classify tangent cones}(3) (as we assumed), $C_i\in\VarTan(\|\widetilde{\Sigma}_i\|,p)$ is a horizontal wedge domain in $T_pM$ with integer multiplicity $m_i\in\mZ_+$, and thus $\frac{2\pi}{\theta(p)}\Theta^n(\|C_i\|,0) = m_i \geq 1$. 
    Combining with (\ref{Eq: min-max regularity: graphs representation})(\ref{Eq: min-max regularity: multiplicity sum}), we conclude that $\Theta^n(\|\Sigma_i\|,p) = \Theta^n(\|C_i\|,0) = \frac{\theta(p)}{2\\pi}$ and $m_i=1$. 
    Now, we can apply the Allard-type regularity theorem (\cite[Theorem 1.1]{edelen2022regularity}) and Theorem \ref{Thm: regularity of FBMH} to show $\widetilde{\Sigma}_i$ extends as a $C^{2,\alpha}$-to-edge properly embedded locally wedge-shaped FBMH near $p$, and $p\in\bd^E\widetilde{\Sigma}_i$. 
    Finally, the maximum principle for FBMHs in wedge domains (\cite[Theorem B.1]{mazurowski2023curvature}) indicates $l=1$, and thus $\widetilde{V}$ is a regular FBMH in $\wB_{s_2}(p)$ with integer multiplicity. 

    {\bf Step 5.} {\em $V=\widetilde{V}$ in a neighborhood of $p$.}

    By the regularity of $\widetilde{V}$ near $p$ given by {\bf Step 4}, we can write $\widetilde{V}\llcorner \wB_{s_0}(p) = k|\widetilde{\Sigma}|$ for some $s_0\in (0,s)$ sufficiently small, where $k\in\mZ_+$ and $\widetilde{\Sigma}$ is a {\em connected} $C^{2,\alpha}$-to-edge almost properly embedded locally wedge-shaped FBMH containing $p$. 
    Then, given $s_1\in (0,s_0)$, the constructions in {\bf Step 3} suggests $V^{**}_{s_1} \llcorner \wA_{s_1,s_0}(p) = k_{s_1}|\widetilde{\Sigma}\cap \wA_{s_1,s_0}(p)| $ for some $k_{s_1}\in \N$. 
    In addition, by Theorem \ref{Thm: compactness for FBMHs} and $\widetilde{V}=\lim_{s_1\to 0}V^{**}_{s_1}$, there exists $s_0'\in (0,s_0)$ so that $k_{s_1}\equiv k$ for all $s_1\in (0,s_0')$. 
    Thus, for any $s_1\in (0,s_0')$, 
    \[V_{s_1}^{**} \llcorner \wA_{s_1,s_0}(p) = \widetilde{V} \llcorner \wA_{s_1,s_0}(p) = k |\widetilde{\Sigma}\cap \wA_{s_1,s_0}(p)|.\]
    On the other hand, it follows from Lemma \ref{Lem: classify face and edge}(ii) and the rectifiability of $V$ that $\spt(\|V\llcorner\wB_{s_0'}(p)\|)$ can not be contained entirely in $\bd^E M$. 
    Hence, by the regularity of $V$ in $M\setminus\bd^EM$ (Theorem \ref{Thm: min-max regularity away from edge}), there exists $s_1\in (0,s_0')$ so that 
    $\wS_{s_1}(p)$ meets $\spt(\|V\|) $ transversally in $\mB^L_\epsilon(q)$ for some $q\in \wS_{s_1}(p)\cap \spt(\|V\|) \setminus\bd^E M$ and $\epsilon>0$ with $\mB^L_\epsilon(q)\cap M \subset \wB_{s_0'}(p)\setminus\bd^E M$. 
    By the gluing arguments in {\bf Step 1} (\ref{Eq: gluing support})(\ref{Eq: gluing multiplicity}), we have $V=V^{**}_{s_1}$ in $\mB^L_\epsilon(q)$, and thus $V=V^{**}_{s_1}=\widetilde{V}$ in a small half ball $\mB^L_\epsilon(q)\cap \wA_{s_1,s_0}(p)$. 
    Finally, as in the proof of (\ref{Eq: gluing support}), one can use the unique continuation of FBMHs and the fact $\Clos(\widetilde{\Sigma}) = \Clos(\widetilde{\Sigma}\setminus\bd^E M)$ to see $V = \widetilde{V}$ in $\wB_{s_0}(p)$, which finishes the proof. 
\end{proof}

\appendix
	\addcontentsline{toc}{section}{Appendices}
	\renewcommand{\thesection}{\Alph{section}}
\section{Fermi-type distance function}\label{Sec: Fermi distance}

Consider a compact connected $(n+1)$-dimensional locally wedge-shaped Riemannian manifold $M^{n+1}\subset \R^L$ satisfying (\ref{dag}) or (\ref{ddag}). 
In this Appendix, we are going to construct a function $\wD_p$ for any $p\in\bd^EM$ whose level sets are smooth strictly mean convex locally wedge-shaped hypersurfaces meeting $\bd M$ orthogonally (even at $\bd^E M$). 

To begin with, we first construct a nice local coordinates system near any $p\in \bd^E M$. 
Roughly speaking, we will construct a foliation $\{S_t\}$ starting at $S_0=\bd^+M_{p,r}$ so that $S_t\cap \bd^-M_{p,r}=(\dist_{\bd^-M_{p,r}}(\cdot , \bd^E M))^{-1}(t)$. Then for any $q$ near $p$, $x_1$ represents the leaf in the foliation containing $q$, $x_2$ represents the geodesic distance in $S_{x_1}$ from $q$ to $S_{x_1}\cap \bd^-M_{p,r}$, and $(x_3,\dots,x_{n+1})$ represent the coordinates of $n_1\circ n_2 (q)\in \bd^EM$ on the edge, where $n_2$ is the geodesic nearest projection in $S_{x_1}$ to $S_{x_1}\cap \bd^-M_{p,r}$, and $n_1$ is the geodesic nearest projection in $\bd^-M_{p,r}$ to $\bd^E M$. 
\begin{lemma}\label{Lem: local chart}
    For any $p\in\bd^E M$, there exists $r>0$ so that $M_{p,r}:=M\cap\mB^L_r(p)$ admits local coordinates $x=(x_1,x_2,\dots,x_{n+1}) = (x_1,x_2,x')$ satisfying $p=0$, $M_{p,r}\subset\{x:x_1\geq 0, x_2\geq 0\}$ in the coordinates, and 
    \begin{itemize}
        \item[(i)] $\bd \tM^+_{p,r} = \{x:x_1=0\}$, $\bd \tM^-_{p,r} = \{x:x_2=0\}$ in the coordinates (see (\ref{Eq: local extension across face}));
        \item[(ii)] the Riemannian metric $g$ in the coordinates satisfies 
            \begin{eqnarray*}
               && g(x_1,0,x') = 
	           \begin{pmatrix}
	           	1	&	\cos(\theta(x'))	& \mathbf{0} \\
	           	\cos(\theta(x'))	& 1 &	\mathbf{0} \\
	           	\mathbf{0}	&	\mathbf{0}	& (g_{ij})
	           \end{pmatrix}, \qquad
	           g(0,x_2,x') = 
	           \begin{pmatrix}
	           	g_{11}	&	g_{12}	& g_{1i} \\
	           	g_{21}	& 1 &	\mathbf{0} \\
	           	g_{i1}	&	\mathbf{0}	& (g_{ij})
	           \end{pmatrix},
                \\
               && g^{-1}(x_1,0,x') = 
	           \begin{pmatrix}
		          \frac{1}{\sin^2(\theta(x'))}	&	-\frac{\cos(\theta(x'))	}{\sin^2(\theta(x'))} & \mathbf{0} \\
		          -\frac{\cos(\theta(x'))	}{\sin^2(\theta(x'))}	& \frac{1}{\sin^2(\theta(x'))} &	\mathbf{0} \\
		          \mathbf{0}	&	\mathbf{0}	& (g^{ij})
	           \end{pmatrix}, ~
                g(0,0,0) = 
	           \begin{pmatrix}
	           	1	&	\cos(\theta(0))	& \mathbf{0} \\
	           	\cos(\theta(0))	& 1 &	\mathbf{0} \\
	           	\mathbf{0}	&	\mathbf{0}	& (\delta_{ij})
	           \end{pmatrix}; 
            \end{eqnarray*}
        \item[(iii)] $\nabla_{\bd_1}\bd_1\perp \{x:x_2=0\}$ and $\nabla_{\bd_2}\bd_2 \perp \{x:x_1=t\}$ for any $t\geq 0$;
        \item[(iv)] if (\ref{ddag}) is satisfied in $M_{p,r}$, then $A_{\{x:x_2=0\}}(\bd_1,\bd_l) = A_{\{x:x_1=0\}}(\bd_2,\bd_l) = -\frac{1}{2} \frac{\bd \theta}{\bd  x_l}$,
	       \[ \frac{\bd g_{1l}}{\bd x_2}(0,0,x') = \frac{\bd g_{2l}}{\bd x_1}(0,0,x') = \frac{\bd g^{1k}}{\bd x_2}(0,0,x') =\frac{\bd g^{2k}}{\bd x_1}(0,0,x') =0, \]
        for all $3\leq k,l\leq n+1$.
    \end{itemize}
\end{lemma}
\begin{remark}
    Note one can also rotate the coordinates $(x_3,\dots,x_{n+1})$ of the edge. 
\end{remark}
\begin{proof}
    Firstly, take the Fermi coordinates of $\tM^-_{p,r}$ in $\tM_{p,r}$ with respect to $\bd^E M$ so that $(0,0,x_3,\dots,x_{n+1})$ is the normal coordinates of $p$ in $\bd^E M$, $x_1$ stands for the (signed) geodesic distance to $\bd^E M$ in $\tM^-_{p,r}$ (positive in $M$), and $x_2$ stands for the (signed) geodesic distance to $\tM^-_{p,r}$ in $\tM_{p,r}$ (positive in $M$). 
    For simplicity, denote by $x'=(x_3,\dots,x_{n+1})$. 
    
    Then $\tM^-_{p,r}=\{x:x_2=0\}$, and $\tM^+_{p,r}=\{x:x_1=f(x_2,x')\}$ for some smooth function $f:\R^n=\{x:x_1=0\}\to \R$ so that $f(0,x')=0$. 
    By a direct computation, the unit normal of ${\rm Graph}(f) = \tM^+_{p,r}$ is
    \[
    	\nu_f:=\frac{-D_if g^{ia} + g^{1a}}{\sqrt{g^{11}-2D_if g^{1i}+ D_ifD_jfg^{ij}}} (f(x_2,x'), x_2,x') \cdot \frac{\bd }{\bd x_a},
    \]
    where $i\in\{2,\dots,n+1\}$, $a\in\{1,\dots,n+1\}$. 
    Hence, we have 
    \begin{equation}\label{Eq: local chart: meeting angle}
        \frac{\bd f}{\bd x_2}(0,x') = \cot(\theta(x')),
    \end{equation}
    where $\theta(x') \in (0,\pi)$ is the wedge angle of $M=\{x: x_1\geq f, x_2\geq 0\}$ at $(0,0,x')$. 

    For $s\in (-r,r)$ small enough, $f_s(x_2,x') := f(x_2,x')+s$ is also a smooth function satisfying $f_s(0,x')=0$ and (\ref{Eq: local chart: meeting angle}), which implies ${\rm Graph}(f_s):=\{x: x_1= f(x_2,x') + s\}$ is also a smooth hypersurface meeting $\tM^-_{p,r}$ with the same angle as $\tM^+_{p,r}$. 
    Therefore, we have a new coordinates $(y_1,y_2,y'):= (s,x_2,x')$ near $p$, where $s$ stands for $x_1 = f_s(x_2,x')$. 
    In this coordinates, $\tM^-_{p,r}=\{y:y_2=0\}$, $\tM^+_{p,r}=\{y:y_1=0\}$, and $y_1=x_1$, $y'=x'$ on $\tM^-_{p,r}$. 

    Next, for any $t\geq 0$, denote by $S_{t}:= \{y : y_1=t, y_2\geq 0\}$ the half $t$-leaf in the foliation $\{{\rm Graph}(f_s)\}_{|s|<\sigma}$, and $\bd S_t := \{y : y_1=t, y_2= 0\}$. 
    Given $y=(y_1,y_2,y')\in M$, we define 
    \[ z_1:= y_1,\qquad z_2:= \dist_{S_{y_1}}(y, \bd S_{y_1}), \qquad z':= y'(n_{S_{y_1}}(y)) ,\]
    where $\dist_{S_{y_1}}(y, \bd S_{y_1})$ is the geodesic distance from $y$ to $\bd S_{y_1}$ in $S_{y_1}$, and $n_{S_{y_1}}(y)$ is the nearest geodesic projection from $y$ to $\bd S_{y_1}$ in $S_{y_1}$. 
    One verifies that $(z_1,z_2,z')$ is also a local coordinate system centered at $p$. 
    In addition, (i)-(iii) are obtained by noting that $(y_2,y')$ are the Fermi-type coordinates in each half leaf $S_t$, and $z_1=y_1$, $z'=y'$ on $\tM^-_{p,r}$. 

    Rename the coordinates $(z_1,z_2,z')$ by $(x_1,x_2,x')$. 
    Then for any $3\leq l\leq n+1$, by (i)-(iii), we have $\frac{\bd g_{2l}}{\bd x_2}(0,0,x')= \langle \nabla_{\bd_2}\bd_2,\bd_l\rangle +\frac{1}{2}\frac{\bd g_{22}}{\bd x_l}=0$ and $\frac{\bd g_{1l}}{\bd x_1}(0,0,x')=\langle \nabla_{\bd_1}\bd_1,\bd_l\rangle +\frac{1}{2}\frac{\bd g_{11}}{\bd x_l}=0$. 
    Moreover, it also follows from (i)-(iii) and a direct computation that 
    \begin{eqnarray}
        \frac{\bd g_{1l}}{\bd x_2}(0,0,x') &=& \frac{\bd g_{2l}}{\bd x_1} -\sin(\theta(x'))\frac{\bd \theta}{\bd x_l}(x') - 2\sin(\theta(x')) A_{\{x:x_2=0\}}(\bd_1,\bd_l) ,
        \\
        0=\frac{\bd g_{2l}}{\bd x_1}(0,0,x') &=& \frac{\bd g_{1l}}{\bd x_2} -\sin(\theta(x'))\frac{\bd \theta}{\bd x_l}(x') - 2\sin(\theta(x')) A_{\{x:x_1=0\}}(\bd_2,\bd_l) , 
    \end{eqnarray}
    which implies $2A_{\{x:x_1=0\}}(\bd_2,\bd_l)  +\frac{\bd \theta}{\bd  x_l} = 2A_{\{x:x_2=0\}}(\bd_1,\bd_l) +\frac{\bd \theta}{\bd  x_l}=-\frac{1}{\sin(\theta(x'))}\frac{\bd g_{1l}}{\bd x_2}$ at $(0,0,x')$. 
    Therefore, (\ref{ddag}) suggests $0=\frac{\bd g_{1l}}{\bd x_2}(0,0,x')=\frac{\bd g_{2l}}{\bd x_1}(0,0,x')$. 
    Finally, at $(0,0,x')$, we have
    \[
    \frac{\bd g^{1k}}{\bd x_2} = - g^{11}g^{kl} \frac{\bd g_{1l}}{\bd x_2} - g^{12}g^{kl}\frac{\bd g_{2l}}{\bd x_2} = 0 \quad {\rm and }\quad \frac{\bd g^{2k}}{\bd x_1} = - g^{22}g^{kl} \frac{\bd g_{2l}}{\bd x_1} - g^{21}g^{kl}\frac{\bd g_{1l}}{\bd x_1}=0,
    \]
    which gives (iv).
\end{proof}

For any $p\in\bd^EM$, the (inward) unit normal of $\bd^{\pm} M_{p,r}$ is given respectively by
\[ \nu_{\{x_1=0\}}(0,x_2,x') = \sum_{a=1}^{n+1}\frac{g^{1a}}{\sqrt{g^{11}}} \frac{\bd }{\bd x_a}(0,x_2,x'),\quad  \nu_{\{x_2=0\}}(x_1,0,x') = \sum_{a=1}^{n+1}\frac{g^{2a}}{\sqrt{g^{22}}} \frac{\bd }{\bd x_a}(x_1,0,x'),\]
in the coordinates defined in Lemma \ref{Lem: local chart}. 
Suppose $f$ is a smooth function in $M_{p,r}$ whose level sets are orthogonal to $\bd M$. 
Then
\begin{itemize}
	\item[(1)]	$\left(g^{11}\frac{\bd f}{\bd x_1} +g^{12} \frac{\bd f}{\bd x_2} + g^{1l}\frac{\bd f}{\bd x_l}\right) (0,x_2,x') = 0$;
	\item[(2)] $\left(g^{21}\frac{\bd f}{\bd x_1} +g^{22} \frac{\bd f}{\bd x_2} + g^{2l}\frac{\bd f}{\bd x_l}\right) (x_1,0,x') = 0$, 
\end{itemize} 
where summation convention is used for $l\in\{3,\dots,n+1\}$. 
Since $g^{1l}(0,0,x') = g^{2l}(0,0,x')=0$ and the matrix $[g^{11}, g^{12}; ~g^{21},g^{22}]$ is invertible, we further have 
\begin{itemize}
	\item[(3)] $\frac{\bd f}{\bd x_1}(0,0,x') = \frac{\bd f}{\bd x_2}(0,0,x') = 0$. 
\end{itemize}
Additionally, at $(0,0,x')$, considering the derivative of (1) with respect to $x_2$ and the derivative of (2) with respect to $x_1$, we have the following results by (3) and Lemma \ref{Lem: local chart}(ii):
\begin{itemize}
	\item[(4)] $\left(g^{11}\frac{\bd^2 f}{\bd x_1\bd x_2} +g^{12} \frac{\bd^2 f}{\bd x_2\bd x_2} + \frac{\bd g^{1l}}{\bd x_2}\frac{\bd f}{\bd x_l}\right) (0,0,x') = 0$;
	\item[(5)] $\left(g^{22}\frac{\bd^2 f}{\bd x_2\bd x_1} +g^{21} \frac{\bd^2 f}{\bd x_1\bd x_1} + \frac{\bd g^{2l}}{\bd x_1}\frac{\bd f}{\bd x_l}\right) (0,0,x') = 0$,
\end{itemize}
where $l\in\{3,\dots,n+1\}$. 
To construct the desired function $\wD_p$, we will use the equations (1)-(5) to design a second order polynomial-like function $F$ as $(\wD_p)^2$, which is close to $|x|^2$. 

\begin{lemma}\label{Lem: Fermi-type distance function}
    Let $M^{n+1}\subset \R^L$ be a compact connected $(n+1)$-dimensional locally wedge-shaped Riemannian manifold satisfying (\ref{dag}) or (\ref{ddag}). 
    Then for any $p\in\bd^EM$, there exist $r_0>0$ and a function $\wD_p$ defined in a neighborhood of $p$ so that 
    \begin{itemize}
        \item[(i)] $(\wD_p)^2$ is smooth, and $\wD_p(q)\geq 0$ with equality only at $q=p$;
        \item[(ii)] for any $t\in (0,r)$, $\wS_t(p):= \{q\in M : \wD_p(q) = t\}$ is a smooth strictly mean convex locally wedge-shaped hypersurface meeting $\bd M$ orthogonally. 
    \end{itemize}
\end{lemma}
\begin{proof}
    For $p\in\bd^EM$, consider the coordinates $(x_1,\dots,x_{n+1})$ given by Lemma \ref{Lem: local chart}. 
    Firstly, we define $F$ (which will be $(\wD_p)^2$) on the edge $\bd^E M_{p,r}$ by 
    \begin{equation}\label{Eq: foliation function on edge}
        F(0,0,x') = h_0(x') := |x'|^2.
    \end{equation}
    Next, we take a function $h_{12}$ on the edge given by
    \begin{eqnarray}\label{Eq: foliation function x1x2}
	   h_{12}(x') := -2\frac{g^{12}}{g^{11}}(0,0,x') = -2\frac{g^{12}}{g^{22}}(0,0,x') =2\cos(\theta(x')),
    \end{eqnarray}
    which will stand for $\frac{\bd^2 F}{\bd x_1\bd x_2}(0,0,x')$. 
    Then, in the light of (4)(5) and (\ref{dag})(\ref{ddag}), we can define
    \begin{equation}\label{Eq: foliation function x1x1}
	   h_{11}(x')  := 
        \left\{ 
	    \begin{array}{ll}
	    	2- \frac{1}{g^{21}}\frac{\bd g^{2l}}{\bd x_1} \frac{\bd h_0}{\bd x_l} (0,0,x'), & \theta(0)\neq \frac{\pi}{2},
	    	\\
	    	2, & \theta(0)=\frac{\pi}{2},
	    \end{array}
	    \right.
    \end{equation}
    \begin{equation}\label{Eq: foliation function x2x2}
	 h_{22}(x')  :=
	 \left\{ 
	 \begin{array}{ll}
	 	2- \frac{1}{g^{12}}\frac{\bd g^{1l}}{\bd x_2} \frac{\bd h_0}{\bd x_l} (0,0,x'), & \theta(0)\neq \frac{\pi}{2},
	 	\\
	 	2, & \theta(0)=\frac{\pi}{2},
	 \end{array}
	 \right.
    \end{equation}
    which will stand for $\frac{\bd^2 F}{\bd x_1\bd x_1}(0,0,x')$ and $\frac{\bd^2 F}{\bd x_2\bd x_2}(0,0,x')$ respectively. 
    Using (3), we define $F$ on the faces $\bd^{\pm} M_{p,r}$ by second order polynomial-like functions $F_1$ and $F_2$ respectively:
    \begin{eqnarray}\label{Eq: foliation function on faces}
	   F(0,x_2,x') = F_2(x_2,x') &:=& |x'|^2 + \frac{x_2^2}{2} h_{22}(x'), \\
	   F(x_1,0,x') = F_1(x_1,x') &:=& |x'|^2 + \frac{x_1^2}{2} h_{11}(x').
    \end{eqnarray}
    Finally, in the view of (1)(2), we define 
    \begin{eqnarray}\label{Eq: foliation function}
	   F(x_1,x_2,x') &=& F_1(x_2,x') + F_2(x_1,x') - h_0(x') \nonumber
        \\&&  + x_1\left[\frac{1}{g^{11}}\left(-g^{12}\frac{\bd F_2}{\bd x_2} - g^{1l}\frac{\bd F_2}{\bd x_l} \right) \right](0,x_2,x')
        \\&&  +x_2\left[\frac{1}{g^{22}}\left(-g^{21}\frac{\bd F_1}{\bd x_1} - g^{2l}\frac{\bd F_1}{\bd x_l} \right) \right](x_1,0,x') - x_1x_2h_{12}(x') .\nonumber
    \end{eqnarray}
    One verifies that $F(0,0,x') =h_0$, $F(0,x_2,x') = F_2$, $F(x_1,0,x')=F_1$, and $\frac{\bd^2 F}{\bd x_ix_j} (0,0,x') = h_{ij}$ for $i,j\in\{1,2\}$. 
    Therefore, (1)-(5) are satisfied with $F$ in place of $f$, and $F$ is a smooth second order polynomial-like function in $M_{p,r}$.  

    We now show that $\wD_p$ can be taken as $\sqrt{F}$. To begin with, a direct computation shows 
	\begin{itemize}
		\item $F(0)=0$ and $DF(0) =0 $;
		\item $D_{11}F(0)= D_{22}F(0) = D_{ll}F(0) = 2$, for $l\geq 3$;
		\item $D_{12}F(0) = 2\cos(\theta(0))\geq 0$, and $D_{1l}F(0)=D_{2l}F(0)=0$ for $l\geq 3$, 
	\end{itemize}
	which implies $F\geq 0$ in a neighborhood of $0$, and $0$ is the unique minimum point. 
    Hence, $\wD_p=\sqrt{F}$ is well defined and satisfies (i). 
	
	Additionally, by a standard blowup argument, we can assume the Riemannian metric $g=g_{_M}$ in the coordinates is sufficiently close to the flat metric $g_0 = g(0,0,0)$ (in the form of Lemma \ref{Lem: local chart}(ii)). 
	Thus, $F$ is sufficiently close to the square of the Euclidean distance function $d_0$ given by $(d_0(x))^2:=x_1^2+x_2^2 +2\cos(\theta(0))x_1x_2+|x'|^2$ in a $\theta(0)$-wedge domain with the above coordinates, which implies the level sets of $F$, as well as $\wD_p$, are small perturbations of the level sets of the Euclidean distance function $d_0$. 
	Therefore, in a small neighborhood of $0$ (i.e. $p$), the level sets of $\wD_p$ are smooth strictly mean convex hypersurfaces, which also meet $\bd^{\pm}M_{p,r} $ orthogonally since $F$ satisfies (1)(2). 
\end{proof}

\begin{definition}\label{Def: Fermi-type ball}
    For any $p\in\bd^EM$, let $r_0>0$ and $\wD_p$ be given as in Lemma \ref{Lem: Fermi-type distance function}. 
    Then for any $t\in (0,r_0]$, we define the {\em Fermi-type spherical wedge} and the {\em Fermi-type wedge-cut sphere} of radius $t$ centered at $p$ respectively by 
    \[\wB_t(p):=\{q\in M: \wD_p(q) < t \}, \qquad \wS_t(p):=\{q\in M: \wD_p(q) = t \}.\]
\end{definition}

Moreover, for $r>0$ small enough, we have $\wD_p$ is sufficiently close to the Euclidean distance $d_0$ in the above coordinates, and thus
\[ \frac{1}{2}\dist_{\R^L}(p,q) \leq  \wD_p(q) \leq 2 \dist_{\R^L}(p,q)\]
for $q\in M$ in a sufficient small neighborhood of $p$.

\section{Maximum principle for stationary varifolds with free boundary}\label{Sec: maximum principle for varifolds}

In this appendix, we follow the process in \cite{li2021maximum} to show the maximum principal for stationary varifolds with free boundary in a strictly mean convex wedge-shaped domain. 
The main difficulty here is the lack of the Fermi coordinates. 
Nevertheless, we can use the idea in Lemma \ref{Lem: Fermi-type distance function} to construct the foliation in the coordinates given by Lemma \ref{Lem: local chart}

\begin{proof}[Proof of Theorem \ref{Thm: maximum principle for varifolds}]
    Let $V\in\V_n(M)$ be stationary in a relative open subset $U\subset M$ with free boundary, and $K\subset\subset U$ be a smooth relatively open connected subset in $M$ so that $\spt(\|V\|)\subset \Clos(K)$ and $\bd_{rel}K$ is a smooth strongly mean-convex hypersurface meeting $\bd M$ orthogonally. 
    Suppose on the contrary that there exists $p\in \spt(\|V\|)\cap \bd_{rel}K$. 
    Then by the interior (\cite{white2010maximum}) and the boundary (\cite{li2021maximum}) maximum principle, we must have $p$ lying on the edge of $\bd_{rel}K$. Let $x=(x_1,\dots,x_{n+1})$ be the coordinates given by Lemma \ref{Lem: local chart} near $p$.

    {\bf Step 1:} {\em Construct locally a smooth hypersurface $\Gamma$ outside $K$ that is orthogonal to $\bd M$ and touches $\bd_{rel}K$ up to second order at $p$.} 
    
    The construction is parallel to Lemma \ref{Lem: Fermi-type distance function}. 
    Firstly, by the construction in Lemma \ref{Lem: local chart}, we can rotate $(x_3,\dots,x_{n+1}) = (x'',x_{n+1})$ the coordinates of the edge so that
    \[\bd_{rel}K \cap \mB^L_r(p) = {\rm Graph}(f) = \{(x_1,x_2,x'', x_{n+1}) : x_{n+1} = f(x_1,x_2,x'')\} ,\]
    and $K\cap \mB^L_r(p) \subset \{x_{n+1} > f(x_1,x_2,x'')\}$ for some real-valued smooth function $f$ defined in the $n$-wedge domain $\Omega:= \{x :x_1\geq 0, x_2\geq 0,  x_{n+1}=0 \} $ satisfying
    \begin{itemize}
	   \item[(I)] $f(0)=0$ and $Df(0)=0$. 
    \end{itemize} 
    Note the unit normal of ${\rm Graph}(f)$ (pointing inward $K$) is given by
    \[ \nu_{{\rm Graph}(f)} = \sum_{a}^{n+1} \frac{-D_if g^{ia} + g^{n+1,a}}{\sqrt{g^{n+1,n+1}-2D_if g^{i,n+1} + D_if D_jf g^{ij}}} (x_1,x_2,x'',f(x_1,x_2,x'')) \frac{\partial}{\partial x_a}  .\]
    Since $\bd_{rel} K$ is orthogonal to $\bd^{\pm} M_{p,r}$, we compute as before Lemma \ref{Lem: Fermi-type distance function} and see 
    \begin{itemize}
        \item[(II)]	$\left(g^{11}\frac{\bd f}{\bd x_1} +g^{12} \frac{\bd f}{\bd x_2} + g^{1l}\frac{\bd f}{\bd x_l}\right) (0,x_2,x'',f(0,x_2,x'')) = g^{1,n+1}(0,x_2,x'',f(0,x_2,x'')) $;
        \item[(III)] $\left(g^{21}\frac{\bd f}{\bd x_1} +g^{22} \frac{\bd f}{\bd x_2} + g^{2l}\frac{\bd f}{\bd x_l}\right) (x_1,0,x'',f(x_1,0,x'')) = g^{2,n+1}(x_1,0,x'',f(x_1,0,x''))$; 
        \item[(IV)] $\frac{\bd f}{\bd x_1}(0,0,x'') = \frac{\bd f}{\bd x_2}(0,0,x'') = 0$, and thus $\frac{\bd^2 f}{\bd x_1 x_l}(0,0,x'') = \frac{\bd^2 f}{\bd x_2x_l}(0,0,x'') = 0$; 
        \item[(V)] $\left(g^{11}\frac{\bd^2 f}{\bd x_1\bd x_2} +g^{12} \frac{\bd^2 f}{\bd x_2\bd x_2} + \frac{\bd g^{1l}}{\bd x_2}\frac{\bd f}{\bd x_l}\right) (0,0,x'',f(0,0,x'')) =\frac{\bd g^{1,n+1}}{\bd x_2}(0,0,x'',f(0,0,x''))$;
        \item[(VI)] $\left(g^{22}\frac{\bd^2 f}{\bd x_2\bd x_1} +g^{21} \frac{\bd^2 f}{\bd x_1\bd x_1} + \frac{\bd g^{2l}}{\bd x_1}\frac{\bd f}{\bd x_l}\right) (0,0,x'',f(0,0,x'')) = \frac{\bd g^{2,n+1}}{\bd x_1}(0,0,x'',f(0,0,x''))$,
    \end{itemize} 
    where the summation convention is used for $3\leq l\leq n$. 

    Next, we construct the desired $\Gamma$ as a graph of some third order polynomial-like function $\tilde{f} \leq f$. 
    Firstly, fix a small constant $\epsilon>0$ and take $h_0:\bd^E\Omega\to\R$ so that
    \begin{equation}\label{Eq: outside touching surface on edge}
    	{\rm Graph}(h_0) = \{x\in \bd^EM\setminus K: \dist_{\bd^EM}(x,\bd^E(\bd_{rel}K)) =     \epsilon\cdot \dist_{\bd^EM}^4(x,0)\}.
    \end{equation}
    Note $h_0$ will stand for $\tilde{f}(0,0,x'')$ making $\bd^E\Gamma$ a hypersurface in $\bd^E M$ touching $\bd^E(\bd_{rel}K)$ from outside $K\cap\bd^E M$ up to second order at $p$ (\cite[Claim 1]{li2021maximum}). 
    Additionally, since $\bd^E\Gamma:={\rm Graph}(h_0)$ is a smooth hypersurface in $\bd^EM$ near $0$, we can assume $(x_3,\dots,x_n)$ is a normal coordinate of $\bd^E\Gamma$ near $p$, and $x_{n+1}$ stands for the signed geodesic distance in $\bd^E M$ to $\bd^E\Gamma$ so that $x_{n+1}>0$ in $K$. 
    Then one can follow the constructions in the previous appendix to give a coordinates system satisfying Lemma \ref{Lem: local chart}(i)-(iv). 
    Moreover, in this renewed coordinate system, we rewrite the graph function of $\bd^E\Gamma$ by $h_0=0$, and the graph function of $\bd_{rel}K$ is still denoted by $f$, which is defined in $\Omega:= \{x :x_1\geq 0, x_2\geq 0,  x_{n+1}=0 \} $ satisfying (I)-(VI) as before. 
    The following constructions are in the renewed coordinates. 
    
    Secondly, in the light of (V)(VI), let
    \begin{align}
    	 h_{12}(x'') &:= \frac{\bd^2 f}{\bd x_1\bd x_2}(0,0,0),\quad\mbox{($=0$ if $\theta(0)=\pi/2$ by (V)(VI)(\ref{ddag}))}, \label{Eq: outside touching surface 12}\\
    	 h_{22}(x'') &:=
    	\left\{ 
    	 \begin{array}{ll}
    	 	-\frac{1}{g^{12}}\Big[g^{11}\big(\frac{\bd^2 f}{\bd x_1\bd x_2}(0,0,0)\big) + \frac{\bd g^{1l}}{\bd x_2}   \frac{\bd h_0}{\bd x_l} \Big] (0,0,x'',h_0(x'')), & \theta(0)\neq \frac{\pi}{2},
    	 	\\
    	 	\frac{\bd^2 f}{\bd x_2\bd x_2}(0,0,0), & \theta(0)=\frac{\pi}{2},
    	 \end{array}
    	 \right. \label{Eq: outside touching surface 11}\\
         h_{11}(x'') &:=
    	\left\{ 
    	 \begin{array}{ll}
    	 	-\frac{1}{g^{21}}\Big[g^{22}\big(\frac{\bd^2 f}{\bd x_1\bd x_2}(0,0,0)\big) + \frac{\bd g^{2l}}{\bd x_1} \frac{\bd h_0}{\bd x_l} \Big] (0,0,x'',h_0(x'')), & \theta(0)\neq \frac{\pi}{2},
    	 	\\
    	 	\frac{\bd^2 f}{\bd x_1\bd x_1}(0,0,0), & \theta(0)=\frac{\pi}{2},
    	 \end{array}
    	 \right.\label{Eq: outside touching surface 22}
    \end{align}
    Using (IV), we define $\Gamma\cap\bd^{\pm}M_{p,r}$ on the faces by the graph of some third order polynomial-like functions $F_1$ and $F_2$ respectively:
    \begin{align}
    	F_2(x_2,x'') &:= h_0(x'') + \frac{x_2^2}{2}h_{22}(x'') + \frac{x_2^3}{6} \Big(\frac{\bd^3 f}{\bd x_2^3}(0,0,0) - \epsilon \Big), \label{Eq: outside touching surface on faces 2}\\
    	F_1(x_1,x'') &:= h_0(x'') + \frac{x_1^2}{2}h_{11}(x'') + \frac{x_1^3}{6} \Big(\frac{\bd^3 f}{\bd x_1^3}(0,0,0) - \epsilon \Big). \label{Eq: outside touching surface on faces 1}
    \end{align}
    Finally, in view of (II)(III), we define
    \begin{eqnarray}\label{Eq: outside touching surface}
    	\tilde{f}(x_1,x_2,x'') &=& F_1(x_1,x'') + F_2(x_2,x'') - h_0(x'') \nonumber
    	\\&& + x_1\left[\frac{1}{g^{11}}\left(g^{1,n+1}-g^{12}\frac{\bd F_2}{\bd x_2} - g^{1l}\frac{\bd F_2}{\bd x_l} \right) \right](0,x_2,x'',F_2(x_2,x''))
    	\\&&  +x_2\left[\frac{1}{g^{22}}\left(g^{2,n+1}-g^{21}\frac{\bd F_1}{\bd x_1} - g^{2l}\frac{\bd F_1}{\bd x_l} \right) \right](x_1,0,x'',F_1(x_1,x''))\nonumber
    	\\&&  - x_1x_2h_{12}(x''), \nonumber
    \end{eqnarray}
    which satisfies $\tilde{f}(0,0,x'')=h_0$, $\tilde{f}(x_1,0,x'')=F_1$, $\tilde{f}(0,x_2,x'')=F_2$, and $\frac{\bd^2 \tilde{f}}{\bd x_i\bd x_j} = h_{ij}$ for $i,j\in\{1,2\}$. 
    Additionally, (I)-(VI) are satisfied with $\tilde{f}$ in place of $f$, and thus $\Gamma={\rm Graph}(\tilde{f})$ is a smooth hypersurface meeting $\bd M$ orthogonally. 
    Moreover, one verifies that $\tilde{f}\leq f$ in a small neighborhood of $p$, and $\tilde{f}$ agrees with $f$ up to second order at $0$.

    {\bf Step 2.} {\em Construct a local foliation $\{\Gamma_s\}_{s\in (-\delta,\delta)}$ orthogonal to $\bd M$ starting at $\Gamma_0 = \Gamma$.}
    
    In the coordinates system after (\ref{Eq: outside touching surface on edge}), $\Gamma ={\rm Graph}(\tilde{f})=\{x: x_{n+1}=\tilde{f}(x_1,x_2,x'')\}$ for some function $\tilde{f}$ given by (\ref{Eq: outside touching surface}) 
    so that 
    \[\tilde{f}(0,0,x'')\equiv 0,\]
    and the above (I)-(VI) are satisfied with $\tilde{f}$ in place of $f$. 

    Next, we apply the technique in Lemma \ref{Lem: Fermi-type distance function} again to construct the foliation $\Gamma_s$. 
    For any $s\in\R$ sufficiently close to $0$, let 
    \begin{equation}
    	\tilde{h}_0^s (x'') = \tilde{f}(0,0,x'')  + s = s,
    \end{equation}
    whose graphs give a foliation in $\bd^EM$ starting at $\bd^E\Gamma$. 
    Then we use $\tilde{h}_0^s$ 
    in place of $h_0$ 
    in (\ref{Eq: outside touching surface 12})-(\ref{Eq: outside touching surface 22}) 
    to get $\tilde{h}_{12}^s(x''), \tilde{h}_{22}^s(x''), \tilde{h}_{11}^s(x'')$.
    Next, define $\tilde{F}_2^s(x_2,x'')$ and $\tilde{F}_1^s(x_1,x'')$ on each face as in (\ref{Eq: outside touching surface on faces 2})-(\ref{Eq: outside touching surface on faces 1}) with $\tilde{h}_0^s, \tilde{h}_{12}^s, \tilde{h}_{22}^s, \tilde{h}_{11}^s$ in place of $h_0, h_{12}, h_{22},h_{11}$, whose graphs give a foliation in $\bd^{\pm}M_{p,r}$ starting at $\bd^{\pm}\Gamma$.  
    Finally, we define
    \begin{eqnarray}
    	\tilde{f}_s(x_1,x_2,x'') &:=& \tilde{F}_1^s(x_1,x'') + \tilde{F}_2^s(x_2,x'') - \tilde{h}_0^s(x'') \nonumber
    	\\&& + x_1\left[\frac{1}{g^{11}}\left(g^{1,n+1}-g^{12}\frac{\bd \tilde{F}_2^s}{\bd x_2} - g^{1l}\frac{\bd \tilde{F}_2^s}{\bd x_l} \right) \right](0,x_2,x'', \tilde{F}_2^s(x_2,x''))
    	\\&& +x_2\left[\frac{1}{g^{22}}\left(g^{2,n+1}-g^{21}\frac{\bd \tilde{F}_1^s}{\bd x_1} - g^{2l}\frac{\bd \tilde{F}_1^s}{\bd x_l} \right) \right](x_1,0,x'',\tilde{F}_1^s(x_1,x''))\nonumber
    	\\&& - x_1x_2\tilde{h}_{12}^s(x''), \nonumber
    \end{eqnarray}
    which satisfies $\tilde{f}_s(0,0,x'')=\tilde{h}_0^s$, $\tilde{f}_s(x_1,0,x'')=\tilde{F}_1^s$, $\tilde{f}_s(0,x_2,x'')=\tilde{F}_2^s$, and $\frac{\bd^2 \tilde{f}_s}{\bd x_i\bd x_j} = \tilde{h}_{ij}^s$ for $i,j\in\{1,2\}$. 
    In addition, (I)-(VI) are satisfied with $\tilde{f}_s$ in place of $f$ except $\tilde{f}_s(0)=s$.  Hence, $\Gamma_s:={\rm Graph}(\tilde{f}_s)$ is a smooth hypersurface orthogonal to $\bd M$ for each $s$ close to $0$. 
    One easily verifies that $\{\Gamma_s\}_{s\in (-\delta,\delta)}$ forms a foliation near $0$ with $\Gamma_0=\Gamma$ for $\delta >0$ small. 


    {\bf Step 3.} {\em Construct $X\in\mathfrak{X}_{tan}(M)$ using $\{\Gamma_s\}_{s\in (-\delta,\delta)}$, and get a contradiction.}

    Firstly, we make the following observations from the above constructions. 
    Since $\bd_{rel}K$ is strictly mean convex and $\Gamma_0=\Gamma$ touches $\bd_{rel}K$ up to second order, there exist $c_0>0$ and $\delta_0>0$ small enough so that 
    \[ H_{\Gamma_s} >c_0, \qquad\forall s\in [0,\delta_0),\]
    where $H_{\Gamma_s}$ is the mean curvature of $\Gamma_s$ with respect to the unit normal pointing outside $K$. 
    Using the foliation $\{\Gamma_s\}_{s\in (-\delta_0,\delta_0)}$, we also have a small neighborhood $\mB^L_r(p)$ of $p$ and a smooth function $s: \mB^L_r(p)\cap M\to (-\delta_0,\delta_0)$ so that $s(q)$ is the unique $s$ with $q\in\Gamma_{s(q)}$. 
    In addition, note $\bd_{rel}\Gamma_0$ lies outside of $K$.
    We can further assume $\bd_{rel}\Gamma_{s(q)}$ lies outside of $K$ for all $q\in \mB^L_r(p)$ by shrinking to a smaller neighborhood $\mB^L_r(p)$ of $p$.

    Next, let $\nu$ be a unit vector field normal to each $\Gamma_s$ and pointing inward $K$. 
    Then, since $\Gamma_s$ is orthogonal to $\bd M$, $\nu$ is tangent to $\bd^FM$ at face points of $M$ and tangent to $\bd^EM$ at edge points of $M$. 
    In addition, the smooth function $s: \mB^L_r(p)\cap M\to (-\delta_0,\delta_0)$ satisfies
    \[\nabla s = \psi \nu \]
    for some function $\psi$ with $\psi(p)=1$ by the constructions. 
    After shrinking the neighborhood $\mB^L_r(p)$ even smaller, we can assume 
    \[\psi\geq \frac{1}{2} \qquad \mbox{in $\mB^L_r(p)\cap M$}.\]
    Now, given $\delta\in (0, \delta_0)$, define 
    \[ X(q):= \phi(s(q))\nu,\qquad\phi(s):=
	   \left\{ 
	   \begin{array}{ll}
		  e^{\frac{1}{s-\delta}} & \forall s\in [0,\delta),\\
		  0 & s\geq \delta,
	   \end{array} 
	   \right.
	   \quad{\rm with}\quad \phi'\leq -\frac{1}{\delta^2}\phi\leq 0.
    \]
    Since $\bd_{rel}\Gamma_{s(q)}$ lies outside of $K$ for all $q\in \mB^L_r(p)$, we can take $\delta\in (0, \delta_0)$ sufficiently small so that $X$ (as well as $\phi$) admits a smooth $0$ extension in $K$ which can be further extended to an element in $\mfX_{tan}(M)$ with compact support (still denote by $X$ and $\phi$). 

    Consider the bilinear form $Q$ on $T^*M$ given by 
    \[ Q(u,v)(q) = \langle\nabla_uX,v\rangle.\]
    Then for all $q\in \Clos(K) $ and $u,v\perp\nu$ at $q$,
    \begin{equation}\label{Eq: maximum principle: bilinear form}
    	Q(\nu,\nu) = \phi'\psi,\quad Q(u,\nu) = 0,\quad Q(\nu, u) = \phi\langle\nabla_\nu \nu, u\rangle,\quad Q(u, v)= -\phi A_{\Gamma_s}(u,v),
    \end{equation}
    where $A_{\Gamma_s}$ is the second fundamental form of $\Gamma_s$ with respect to $-\nu$. 
    We also have 
    \begin{equation}\label{Eq: maximum principle: bilinear form estimates}
        |Q(\nu, u)|\leq \phi|u|C \quad{\rm and}\quad |Q(u, v)|\leq \phi|u||v|C, 
    \end{equation}
    for some $C>0$ depending on $\{\Gamma_s\}$ (independent on $\delta$), where $u,v\perp\nu$ at $q\in \Clos(K) $.

    Given $q\in \Clos(K)$ with $X(q)\neq 0$, take any $n$-subspace $P\subset T_qM$.
    If $P=T_q\Gamma_s$, then ${\rm tr}_PQ \leq -\phi c_0 <0$. 
    If $P\neq T_q\Gamma_s$, let $\{v_1,\dots,v_{n-1}\}$ be an orthonormal basis of $P\cap T_q\Gamma_s$, $v_0\in T_q\Gamma_s$ and $v_{n}\in P$ be the unit normal of $P\cap T_q\Gamma_s$ in $T_q\Gamma_s$ and in $P$ respectively. 
    Hence, there exists $\theta\in [0, 2\pi]$ so that $v_n = \cos\theta v_0 + \sin\theta\nu$. 
    It follows from (\ref{Eq: maximum principle: bilinear form})(\ref{Eq: maximum principle: bilinear form estimates}) that 
    \begin{eqnarray*}
    	{\rm tr}_PQ(q) &=& \sum_{i=1}^{n-1}Q(v_i,v_i) + Q(v_n,v_n)
    	\\&=& \sum_{i=0}^{n-1}Q(v_i,v_i) + \sin^2\theta (Q(\nu,\nu) - Q(v_0,v_0)) + \sin\theta\cos\theta (Q(\nu,v_0)+Q(v_0,\nu))
    	\\&=& -\phi H_{\Gamma_{q(s)}} + \sin^2\theta(\phi'\psi + \phi A_{\Gamma_s}(v_0,v_0)) + \sin\theta\cos\theta Q(\nu,v_0)
    	\\&\leq& -c_0\phi + \phi \sin^2\theta \Big( -\frac{1}{2\delta^2}+C \Big) + \sin\theta\cos\theta \cdot \phi C
    	\\&\leq& -c_0\phi + \phi \Big( -\frac{\sin^2\theta}{2\delta^2}+\sin^2\theta C + |\sin\theta\cos\theta| C \Big)
    	\\&\leq& -c_0\phi + \phi \Big( -\frac{\sin^2\theta}{2\delta^2}+\sin^2\theta C + \frac{\sin^2\theta C^2}{c_0}+\frac{c_0\cos^2\theta }{4} \Big)
    	\\&\leq& -\frac{3}{4}c_0\phi  + \phi \sin^2\theta \Big( -\frac{1}{2\delta^2}+ C + \frac{ C^2}{c_0}\Big).
    \end{eqnarray*}
    By taking $\delta \in (0,\delta_0)$ small enough, we have ${\rm tr}_PQ(q)\leq -\frac{3}{4}c_0\phi < 0$. 

    Combining with $\spt(\|V\|)\subset\Clos(K)$, we have $\delta V(X)<0$ as a contradiction. 
\end{proof}

\section{Elliptic regularity for FBMHs in non-obtuse wedge domain}\label{Sec: elliptic regularity}

Consider a locally wedge-shaped Riemannian manifold $M^{n+1}\subset\R^L$ and a $C^{1,\alpha}$-to-edge almost properly embedded locally wedge-shaped FBMH $(\Sigma,\{\bd_m \Sigma\})\subset  (M,\{\bd_{m+1}M\})$ for some $\alpha\in(0,1]$. 
By the classical regularity results, $\Sigma$ is smooth away from the edge $\bd^E\Sigma$. 
In this appendix, we will show the regularity of $\Sigma$ can be upgraded to  $C^{2,\alpha_0}$-to-edge for some $\alpha_0\in (0,1)$ provided (\ref{dag}) or (\ref{ddag}). 

Take any $p\in \bd^E\Sigma$. 
We first write down the elliptic equation and the boundary conditions for $\Sigma$ in a local model near $p$. 
Then in Case (\ref{dag}), we combine the regularity results in \cite{lieberman1988oblique}\cite{grisvard2011elliptic} to show the $C^{2,\alpha_0}$-regularity of $\Sigma$ near $p$ provided $n=2$. 
For higher dimensional cases, the proof can be restricted on each $2$-dimensional slice near $p$ together with a continuity argument. 
In Case (\ref{ddag}), we apply the reflection technique to show the regularity of $\Sigma$. 

\subsection{Equations of FBMHs}

Let $(\phi_p, \mB^L_R(p), \Omega(p))$ be a local model of $M$ centered at $p\in\bd^E\Sigma$, where $\Omega(p) = \Omega^{2}_{\theta_0} \times \R^{n-1}$ is an $(n+1)$-wedge domain in the standard form (\ref{Eq: standard wedge}) and $\theta_0 = \theta(p)$. 
Note 
\begin{equation}\label{Eq: metric is Euclidean at 0}
	g_{ij}(0) = \delta_{ij}, \quad\forall i,j\in\{1,\dots,n+1\}. 
\end{equation}
Additionally, the outward unit normal $\nu^{\pm}$ on $\bd^{\pm}\Omega(p) \subset \{x: x_1= \pm\cot(\theta_0/2)x_2\}$ with respect to the metric $g:=g_{_M}$ is given by 
\[ \nu_{\pm}  ~=~ - \frac{\mp\cot(\frac{\theta_0}{2})g^{2a} + {g}^{1a}}{\sqrt{g^{11} \mp 2\cot(\frac{\theta_0}{2})g^{12} + \cot^2(\frac{\theta_0}{2})g^{22}}} \frac{\partial}{\partial x_a} ,\]
where the summation convention is used, $a\in\{1,\dots,n+1\}$.

Let $\Omega:= \{x\in \Omega(p): x_{n+1}=0\}$, $\tx:=(x_1,\dots,x_n)$, and $\Omega_r:= \Omega\cap \Clos(\mB^n_r(0))$ for $r>0$. 
Without loss of generality, we may assume that $T_p\Sigma = \Omega$ and there exists a function $u\in C^{1,\alpha}(\Omega_{R})\cap C^{\infty}(\Omega_{R}\setminus\bd^E\Omega)$ so that 
$ \Sigma = {\rm Graph}(u) := \{(\tx, u(\tx)) :\tx\in\Omega_R\} $, and thus
\begin{equation}\label{Eq: u on tangent space}
	u(0)=0, \qquad Du(0)=0.
\end{equation}
Then, the upward unit normal of ${\rm Graph}(u)\subset\Omega(p)$ is given by 
$$\nu_{\Sigma}(\tx, u(\tx)) = \frac{-u_i g^{ia} + g^{n+1,a}}{\sqrt{g^{n+1,n+1}-2u_i g^{i,n+1} + u_i u_j g^{ij}}} (\tx,u(\tx)) \frac{\partial}{\partial x_a}\Big|_{(\tx,u(\tx))} ,$$
where $i,j\in\{1,\dots,n\}$, $a,b\in\{1,\dots,n+1\}$, and $u_i:=\frac{\partial u}{\partial x_i}$. 
By the definition of FBMHs, $u$ is a solution of the following elliptic PDE:
\begin{equation}\label{Eq: FBMH equation general}
	\left\{\begin{array}{ll}
		Lu := \Div_{\Sigma} \nu_\Sigma = 0 &\qquad  \mbox{in $\interior(\Omega_{R})$}, \\
		Bu := \beta_{\pm} \cdot Du = f_{\pm} &\qquad  \mbox{on $\bd^{\pm}\Omega_{R}$},
	\end{array}\right.
\end{equation}
where $\beta_{\pm}:=(\beta_{\pm, u}^1,\dots,\beta_{\pm, u}^n )$ and $f_{\pm}:=\beta_{\pm, u}^{n+1}$ are defined by
\begin{equation}\label{Eq: boundary coefficients}
	\beta_{\pm, u}^a(\tx):=(\pm\cos(\frac{\theta_0}{2})g^{2a} - \sin(\frac{\theta_0}{2})g^{1a})(\tx,u(\tx)) ~ \in C^{1,\alpha}(\bd^{\pm}\Omega)\cap C^{\infty}(\bd^{\pm}\Omega\setminus\bd^E\Omega).
\end{equation}
In addition, if we write $Lu$ in the form of 
\[Lu=a_{ij}u_{ij} + b_iu_i + cu -f ,\]
then 
\begin{itemize}
	\item[(i)] $a_{ij}(\tx) := a_{ij}(\tx, u(\tx), Du(\tx)) \in C^{0,\alpha}(\Omega_{R})\cap C^{\infty}(\Omega_{R}\setminus\bd^E\Omega) $ is a uniformly positive-definite matrix-valued function so that $a_{ij}(0) = \delta_{ij}$;
	\item[ (ii)] $b_i(\tx) := b_i(\tx, u(\tx), Du(\tx))\in C^{0,\alpha}(\Omega_{R})\cap C^{\infty}(\Omega_{R}\setminus\bd^E\Omega)$ and $c(\tx) =0$; 
	\item[(iii)] $f(\tx) := f(\tx, u(\tx)) \in C^{1,\alpha}(\Omega_{R})\cap C^{\infty}(\Omega_{R}\setminus\bd^E\Omega)$;
	\item[(iv)] $f_{\pm}(0)=0$, and $\beta_{\pm}(0) = \nu_{\pm, \delta}$ is the outward unit normal of $\bd^{\pm}\Omega$ with respect to the Euclidean metric $\delta$. 
\end{itemize}
Noting $a_{ij}(0)=\delta_{ij}$ and $\beta_{\pm}(0) = \nu_{\pm, \delta}$, we have the following lemma when $n=2$ by \cite[Theorem 1.4]{lieberman1988oblique}. 
\begin{lemma}\label{Lem: C1 alpha with large alpha}
	Suppose $R>0$, $\alpha\in (0,1]$, $\Omega$ is an $n$-wedge domain with wedge angle $\theta_0$, and $u\in C^{1,\alpha}(\Omega_{R})\cap C^{\infty}(\Omega_{R}\setminus\bd^E\Omega)$ is a function so that (\ref{Eq: metric is Euclidean at 0})-(\ref{Eq: boundary coefficients}) are satisfied. 
	Then $u$ is $C^{1,\alpha}$ in $\Omega_{R}$ for {\em all} $\alpha \in (0,1)$ provided $n=2$ and $\theta_0\in (0, \frac{\pi}{2}]$. 
\end{lemma}
\begin{proof}
	In a $2$-dimensional wedge domain $\Omega$ with wedge angle $\theta_0\in (0, \frac{\pi}{2}]$, the optimal weighted H\"older coefficient $\alpha_1$ in \cite[Lemma 1.3, Theorem 1.4]{lieberman1988oblique} can be taken as $1$ for the equations (\ref{Eq: FBMH equation general}), which immediately gives the lemma by the estimates \cite[(1.21)]{lieberman1988oblique} and \cite[Lemma 2.1]{gilbarg1980intermediate}. 
	Indeed, after rotating $\Omega_{\theta_0}$, we can take $A^{ij} := a_{ij}(0) = \delta_{ij}$, $\beta_1:=\nu_{-, \delta}$, $\beta_2:=\nu_{+, \delta}$, and $\tau=\cot(\theta_0/2)$ in \cite[Lemma 1.1]{lieberman1988oblique}. 
    Note the inward pointing condition \cite[(1.3a)]{lieberman1988oblique} can be changed equivalently to our outward pointing case $\tau \beta_1^1+\beta_1^2\leq -1$. 
	Hence, we can follow the proof of \cite[Lemma 1.1]{lieberman1988oblique}  to get 
	$\beta_1=(0,-1)$, $\beta_2=(-\sin(\theta_0), \cos(\theta_0))$, and 
	$
		T = \frac{\beta_{2}^{1} \beta_{1}^{1}+\frac{A^{11}}{A^{22}} \beta_{2}^{2} -\frac{2 A^{12}}{A^{22}\beta_1^2} \beta_{2}^{2} \beta_{1}^{2}}{\beta_{2}^{1} \beta_{1}^{2} - \beta_{1}^{1} \beta_{2}^{2}} = \frac{\cos(\theta_0)}{\sin(\theta_0)} \geq 0,
	$
	which implies $T=\cot(\delta)$ for some $\delta\in (0, \frac{\pi}{2}]$. 
    Next, by the statements on \cite[Page 5]{lieberman1988oblique} before Lemma 1.2, the constant $\alpha_1$ in \cite[Lemma 1.1]{lieberman1988oblique} can be taken as $\alpha_1=(\pi - \delta)/\theta_0 \geq 1$. 
	Therefore, for any $\alpha\in (0,1)$ (i.e. $\alpha < \alpha_1$), after applying \cite[Lemma 1.2]{lieberman1988oblique} with our coefficients (i)-(iv), the arguments in \cite[Lemma 1.3, Theorem 1.4]{lieberman1988oblique} suggest $u$ is $C^{1,\alpha}$.
\end{proof}

\subsection{In $2$-dimension acute wedge domains}

We now consider the case that $n=2$ and $\theta_0<\pi/2$. 
By Lemma \ref{Lem: C1 alpha with large alpha}, we may assume $u\in C^{1,\alpha}(\Omega_{R})\cap C^{\infty}(\Omega_{R}\setminus\bd^E\Omega)$ with
\begin{equation}\label{Eq: C 1 alpha with large alpha}
	\frac{1}{2}< \alpha = \alpha_1 + \epsilon  <1, \quad \alpha_1>\frac{1}{2}>  \epsilon>0. 
\end{equation}
Then, rewrite the equations (\ref{Eq: FBMH equation general}) by 
\begin{equation}\label{Eq: FBMH equation 2-dim (1)}
	\left\{\begin{array}{ll}
		\hat{L}u := a_{ij}u_{ij} = \hat{f} &\qquad  \mbox{in $\interior(\Omega_{R})$}, \\
		\hat{B}u := \beta_{\pm}Du = f_{\pm} &\qquad  \mbox{on $\bd^{\pm}\Omega_{R}$},
	\end{array}\right.
\end{equation}
where $\hat{f} := -b_iu_i + f \in C^{0,\alpha}(\Omega_{R})\cap C^{\infty}(\Omega_{R}\setminus\bd^E\Omega)$. 

\begin{lemma}\label{Lem: growth rate of DDDu}
	Using the above notations, for all $i,j,k\in \{1,2\}$ and $x\in \Omega_{R}\setminus\{0\}$, we have 
	\[|x|^{1-\alpha}|u_{ij} (x)| \leq C \quad {\rm and}\quad |x|^{2-\alpha}|u_{ijk}(x)|\leq C,\]
	where $|x|:=\sqrt{x_1^2+ x_2^2}$ and $C= C(g,\|u\|_{C^{1,\alpha}(\Omega_{R})},\Omega_{R})>0$. 
\end{lemma}

\begin{proof}
	Denote by $C>0$ a varying constant depending only on $g$, $\|u\|_{C^{1,\alpha}(\Omega_{R})}$ and $\Omega_{R}$. 
 
    Fix any $x\in \Omega_{R}$ with $0<|x|<1$. 
    Define $\tilde{u}(\xi):=u(|x|\xi)$, $\tilde{f}(\xi) := \hat{f}(|x|\xi)$, $\tilde{a}_{ij}(\xi):=a_{ij}(|x|\xi)$, $\tilde{\beta}_{\pm}(\xi):=\beta_{\pm}(|x|\xi)$, and $\tilde{f}_{\pm}(\xi):=f_{\pm}(|x|\xi)$. 
	Then we have 
	\[\left\{\begin{array}{ll}
		\tilde{a}_{ij}\tilde{u}_{ij} = |x|^2\tilde{f} &\qquad  \mbox{in $\interior(\Omega_{R/|x|})$}, \\
		\tilde{\beta}_{\pm}D\tilde{u} = |x|\tilde{f}_{\pm} &\qquad  \mbox{on $\bd^{\pm}\Omega_{R/|x|}$}.
	\end{array}\right.\]
	Note $\tilde{L}:=\tilde{a}_{ij} D_{ij}$ is uniformly elliptic, and $\|\tilde{a}_{ij}\|_{C^{0,\alpha_1}}$, $\|\tilde{\beta}_{\pm}\|_{C^{1,\alpha_1}}$ are uniformly bounded (independent on $|x|<1$). 
	Hence, for $y\in \Omega_{R}$ with $|x-y|<|x|/4$, we can take $\xi := x/|x| $ and $\eta:=y/|x|$ in $\Omega$, and then apply the standard elliptic estimates (c.f. \cite[Theorem 4.10]{lieberman2013oblique}) to $\tilde{u}$: 
	\begin{eqnarray*}
		|x|^2 |u_{ij}(x)| + |x|^{2+\alpha_1}\frac{|u_{ij}(x)-u_{ij}(y)|}{|x-y|^{\alpha_1}} &=& |\tilde{u}_{ij}(\xi)| + \frac{|\tilde{u}_{ij}(\xi) - \tilde{u}_{ij}(\eta)|}{|\xi - \eta|^{\alpha_1}}
		\\&\leq & C \left( |x|^2\| \tilde{f}\|_{C^{0,\alpha_1}(U_{\xi})}  + |x| \|\tilde{f}_{\pm}\|_{C^{1,\alpha_1}(U_\xi\cap\bd^{\pm}\Omega)} + \|\tilde{u}\|_{C^0(U_\xi)} \right)
		\\&\leq & C \Big ( |x|^{2}\| \hat{f}\|_{C^{0}(V_x)} + |x|^{2+\alpha_1}  [\hat{f}]_{\alpha_1,V_x} 
		\\&& \quad + \sum_{|\sigma|\leq 1} |x|^{1+|\sigma|} \sup_{z\in V_x\cap\bd^{\pm}\Omega} |D^\sigma f_{\pm}(z) | 
		\\&& \quad + \sum_{|\sigma|= 1} |x|^{2+\alpha_1} \sup_{z,w\in V_x\cap\bd^{\pm}\Omega}\frac{ |D^\sigma f_{\pm}(z) - D^\sigma f_{\pm}(w) |}{|z-w|^{\alpha_1}} 
		\\&& \quad + \|u\|_{C^0(V_x)} \Big),
	\end{eqnarray*}
	where $i,j\in\{1,2\}$, $U_{\xi}:= \{\zeta\in\Omega_{R/|x|}: |\zeta - \xi |<1/2\}$, $V_x:=\{y\in \Omega_{R}: |y-x|<|x|/2\}$, and $[\cdot]_{\alpha_1}$ is the H\"{o}lder semi-norm. 
	
	It follows immediately from (\ref{Eq: u on tangent space}) and $u\in C^{1,\alpha_1+\epsilon}(\Omega_{R})$ that 
	\[|x|^{-1-\alpha}|u(x)|=|x|^{-1-\alpha_1-\epsilon}|u(x)| \leq C.\]
	Similarly, since $f_{\pm}(0)=0$ and $f_{\pm}\in C^{1,\alpha}(\bd^{\pm}\Omega^{n}_{R})$, we also have 
	\[|x|^{-1}|f_{\pm}(x)| \leq C .\]
	Hence, using (\ref{Eq: boundary coefficients}), (i)-(iv) before Lemma \ref{Lem: C1 alpha with large alpha}, and the fact that $|x|/2<|y|<3|x|/2$ for $y\in V_x$, we can times the above inequality by $|x|^{-1-\alpha_1-\epsilon}$ and conclude:
	\begin{eqnarray*}
		|x|^{1-\alpha}|u_{ij}|(x) + |x|^{1-\epsilon}\frac{|u_{ij}(x)-u_{ij}(y)|}{|x-y|^{\alpha_1}} &\leq & C \Big(\|\hat{f}\|_{\Lambda_{1-\epsilon}^{0,\alpha_1}(\Omega_{R})} + \|f_{\pm}\|_{\Lambda_{1-\epsilon}^{1,\alpha_1}(\bd^{\pm}\Omega_{R})} + \sup_{x\in \Omega_{R}} |x|^{-1-\alpha}|u(x)| \Big) 
		\\&\leq & C \Big(\|\hat{f}\|_{C^{0,\alpha_1}(\Omega_{R})} + \|f_{\pm}\|_{C^{1,\alpha_1}(\bd^{\pm}\Omega_{R})} + \sup_{x\in \Omega_{R}} |x|^{-1-\alpha}|u(x)| \Big) 
		\\&\leq & C,
	\end{eqnarray*}
	where $\|\cdot\|_{\Lambda_\delta^{l,\alpha}}$ is defined by 
	\[\|v\|_{\Lambda_{\delta}^{l, \alpha}(K)} := \sum_{|\sigma| \leq l} \sup _{x \in K}|x|^{\delta-l-\alpha+|\sigma|}\left|D_{\sigma} v(x)\right|+\sum_{|\sigma|=l} \sup _{\substack{x, y \in K \\|x-y| \leq|x| / 2}}|x|^{\delta} \frac{\left|D_{\sigma} v(x)-D_{\sigma} v(y)\right|}{|x-y|^{\alpha}}.\]
	(See also \cite[Section 2.1, Lemma 3.4]{mazya2004schauder}.)
	This shows the first part of this lemma. 
	
	Next, noting $u\in\Lambda_{1-\epsilon}^{2,\alpha_1}(\Omega_{R})$ and $a_{ij}(x) = a_{ij}(x, u(x), Du(x))$, the above result further suggests $|D\tilde{a}_{ij}(\xi)| = |x||Da_{ij}(|x|\xi)| \leq C$ and 
	\[ \frac{|D\tilde{a}_{ij}(\xi) - D\tilde{a}_{ij}(\eta)|}{|\xi-\eta|^{\alpha_1}} = |x|^{1+\alpha_1}\frac{|Da_{ij}(|x|\xi) - Da_{ij}(|x|\eta)|}{\big||x|\xi-|x|\eta\big|^{\alpha_1}} \leq C .\]
	Thus, $\tilde{L}:=\tilde{a}_{ij} D_{ij}$ is uniformly elliptic and $\|\tilde{a}_{ij}\|_{C^{1,\alpha_1}}$ is uniformly bounded (independent on the choice of $x$). 
	Similarly, we have $\|\tilde{\beta}_{\pm}\|_{C^{2,\alpha_1}}$, $\|\hat{f}\|_{\Lambda_{1-\epsilon}^{1,\alpha_1}} $ and $\|f_{\pm}\|_{\Lambda_{1-\epsilon}^{2,\alpha_1}}$ are uniformly bounded. 
	Therefore, using the standard elliptic estimates as the above procedure,
	\begin{eqnarray*}
		|x|^3 |u_{ijk}(x)| = |\tilde{u}_{ijk}(\xi)| &\leq & C \Big ( \sum_{|\sigma|\leq 1} |x|^{2+|\sigma|} \sup_{z\in V_x} |D^\sigma \hat{f}(z) | 
		\\&& \quad + \sum_{|\sigma|= 1} |x|^{3+\alpha_1} \sup_{z,w\in V_x}\frac{ |D^\sigma \hat{f}(z) - D^\sigma \hat{f}(w) |}{|z-w|^{\alpha_1}} 
		\\&& \quad + \sum_{|\sigma|\leq 2} |x|^{1+|\sigma|} \sup_{z\in V_x\cap\bd^{\pm}\Omega} |D^\sigma f_{\pm}(z) | 
		\\&& \quad + \sum_{|\sigma|= 2} |x|^{3+\alpha_1} \sup_{z,w\in V_x\cap\bd^{\pm}\Omega}\frac{ |D^\sigma f_{\pm}(z) - D^\sigma f_{\pm}(w) |}{|z-w|^{\alpha_1}} 
		\\&& \quad + \|u\|_{C^0(V_x)} \Big).
	\end{eqnarray*}
	Multiplying this inequality by $|x|^{-1-\alpha_1-\epsilon}$, we conclude:
	\[ |x|^{2-\alpha}|u_{ijk}|(x)  \leq C \Big(\|\hat{f}\|_{\Lambda_{1-\epsilon}^{1,\alpha_1}(\Omega_{R})} + \|f_{\pm}\|_{\Lambda_{1-\epsilon}^{2,\alpha_1}(\bd^{\pm}\Omega_{R})} + \sup_{x\in \Omega_{R}} |x|^{-1-\alpha}|u(x)| \Big) 
		\leq C,\]
	which gives the lemma. 
\end{proof}

Using the above estimates, we show the $C^{2,\alpha_0}$-regularity of $u$ near $p$. 

\begin{theorem}\label{Thm: C2 alpha regularity in 2-dim acute angle domain}
    Given $R>0$ and $\alpha\in (0,1]$, suppose $\Omega$ is a $2$-dimensional wedge domain with wedge angle $\theta_0<\pi/2$, and $u\in C^{1,\alpha}(\Omega_{R})\cap C^{\infty}(\Omega_{R}\setminus\bd^E\Omega)$ is a function so that (\ref{Eq: metric is Euclidean at 0})-(\ref{Eq: boundary coefficients}) are satisfied. 
	Then $u\in C^{2,\alpha_0}(\Omega_{R})$ for some $\alpha_0\in (0,1)$. 
\end{theorem}
\begin{proof}
	Firstly, rewrite the equations (\ref{Eq: FBMH equation 2-dim (1)}) by 
		\begin{equation}\label{Eq: FBMH equation 2-dim (2)}
			\left\{\begin{array}{ll}
				\triangle_{\delta} u = F &\qquad  \mbox{in $\interior(\Omega_{R})$}, \\
				\nu_{\pm,\delta} \cdot Du = F_{\pm} &\qquad  \mbox{on $\bd^{\pm}\Omega_{R}$},
			\end{array}\right.
		\end{equation}
	where $F := (\delta_{ij}-a_{ij})u_{ij} + \hat{f} $, $F_{\pm} = (\nu_{\pm,\delta}-\beta_{\pm}) \cdot Du + f_{\pm}$, $\triangle_\delta$ is the Laplacian operator in the Euclidean metric $\delta$, and $\nu_{\pm,\delta}$ is the outward unit normal of $\bd^{\pm}\Omega$ under the metric $\delta$. 
	We then claim that there exists $\gamma\in (0,1)$ so that 
    \begin{equation}\label{Eq: Holder continuity of right hand side}
        F\in C^{0,\gamma}(\Omega_{R})\quad \mbox{and}\quad F_{\pm} \in C^{1,\gamma}(\bd^{\pm}\Omega_{R}).
    \end{equation}
	Indeed, by Lemma \ref{Lem: growth rate of DDDu}, we  can take $\alpha>\frac{1}{2}$ as in (\ref{Eq: C 1 alpha with large alpha}), and have $|x|^{1-\alpha} |u_{ij}(x)|\leq C$ and $|x|^{2-\alpha}|u_{ijk}(x)|\leq C$ for $x\in \Omega_{R}\setminus\{0\}$. 
	Since $a_{ij}(0)=\delta_{ij}$, we further obtain $|x|^{\alpha}(\delta_{ij}-a_{ij})\leq C$ and $|x|^{1-\alpha}|Da_{ij}(x)|\leq C$. 
	Together, we conclude 
	\[|((\delta_{ij}-a_{ij})u_{ij}) (x)| \leq C |x|^{2\alpha -1} \quad{\rm and}\quad |D((\delta_{ij}-a_{ij})u_{ij}) (x)|\leq C|x|^{2\alpha-2},\]
	which implies $(\delta_{ij}-a_{ij})u_{ij}\in W^{1,p}(\Omega^2_{\theta_0})$ for some $p>2$, and thus $(\delta_{ij}-a_{ij})u_{ij}$ as well as $F$ is in $C^{0,\gamma}(\Omega_{R})$ for some $\gamma\in (0,1)$. 
	A similar argument shows $F_{\pm} \in C^{1,\gamma}(\bd^{\pm}\Omega_{R})$. 
	
	Next, by (\ref{Eq: Holder continuity of right hand side}), we can apply \cite[Theorem 6.4.2.6]{grisvard2011elliptic} directly and see 
	\[u-\sum_{-(2+\sigma)<\lambda_{m}<0} C_{m} \widetilde{\Im}_{m} \in C^{2, \sigma}(\Omega_{R}),\]
	where $\sigma\in (0,1)$, $m\in\mZ$, $C_m\in\R$, $\lambda_{m}$ and $\widetilde{\Im}_{m}$ are given in \cite[Page 297, Definition 5.1.3.4]{grisvard2011elliptic}. 
	Using $\theta_0 \in (0,\frac{\pi}{2})$ and $\nu_{\pm,\delta} \cdot Du = F_{\pm} $, one easily check that $\sum_{-(2+\sigma)<\lambda_{m}<0} C_{m} \widetilde{\Im}_{m} \in C^{2,\alpha_0}(\Omega_{R})$ for some $\alpha_0\in (0, \min\{\gamma, \sigma\})$, which implies $u\in C^{2,\alpha_0}(\Omega_{R})$. 
\end{proof}

\subsection{In higher dimensional acute wedge domains}

\begin{theorem}\label{Thm: C2 alpha regularity in n-dim acute angle domain}
    Given $R>0$, $n\geq 2$ and $\alpha\in (0,1]$, suppose $\Omega$ is an $n$-dimensional wedge domain with wedge angle $\theta_0<\pi/2$, and $u\in C^{1,\alpha}(\Omega_{R})\cap C^{\infty}(\Omega_{R}\setminus\bd^E\Omega)$ is a function so that (\ref{Eq: metric is Euclidean at 0})-(\ref{Eq: boundary coefficients}) are satisfied. 
	Then $u\in C^{2,\alpha_0}(\Omega_{r})$ for some $\alpha_0\in (0,1)$ and $r\in (0,R)$. 
\end{theorem}
\begin{proof}
    Firstly, combining \cite{lieberman1988oblique} with a standard difference quotient method, we have $u_k\in C^{1,\alpha}(\Omega_{(2-\delta)R})$ for all $k\geq 3$. 
    It's now sufficient to consider $u_{ij}$ with $i,j\in\{1,2\}$. 
    Hence, rewrite the equation (\ref{Eq: FBMH equation general}) by 
    \begin{equation}\label{Eq: FBMH equation n-dim (1)}
    	\left\{\begin{array}{ll}
    		\tilde{L} u := \sum_{i,j=1}^2 a_{ij}u_{ij} = \tilde{f} &\qquad  \mbox{in $\interior(\Omega_{R})$}, \\
    		\tilde{B} u := \sum_{i=1}^2\beta_{\pm}^i u_i = \tilde{f}_{\pm} &\qquad  \mbox{on $\bd^{\pm}\Omega_{R}$},
    	\end{array}\right.
    \end{equation}
    where $\tilde{f} := -\sum_{i~{\rm or}~j \geq 3} a_{ij}u_{ij} - b_i u_i + f \in C^{0,\alpha}(\Omega_{R})\cap C^{\infty}(\Omega_{R}\setminus\bd^E\Omega)$, and $\tilde{f}_{\pm}:=-\sum_{i=3}^n\beta_{\pm}^iu_i + f_{\pm}\in C^{1,\alpha}(\bd^{\pm}\Omega_R)\cap C^{\infty}(\bd^{\pm}\Omega_R\setminus\bd^E\Omega)$.

    Then, in the $2$-dimensional subspace $\Omega^2(0):=\{(x_1,x_2,0,\dots,0)\in \Omega\}$, (\ref{Eq: metric is Euclidean at 0})(\ref{Eq: u on tangent space}) are satisfied. 
    Thus, by the proof of Theorem \ref{Thm: C2 alpha regularity in 2-dim acute angle domain}, we have $u$ is $C^{2,\alpha_0}$ for some $\alpha_0\in (0,1)$ in this $2$-dimensional subspace $\Omega^2(0)$. 

    Similarly, for any $x'\in \bd^E\Omega=\R^{n-2}$ sufficiently close to $0$, we consider the function 
    \[v_{x'}(x_1,x_2):=u(x_1,x_2,x')-u(0,0,x')-x_1\cdot u_1(0,0,x') -x_2\cdot u_2(0,0,x')\]
    restricted in the $2$-dimensional wedge domain $\Omega^2(x'):=\{(x_1,x_2,x')\in\Omega\}$. 
    By taking $x'\in\bd^E\Omega$ sufficiently close to $0$, $g\llcorner \Omega^2(x')$ is sufficiently close to $g\llcorner \Omega^2(0)$, and the angle of $\Omega^2(x')$ is still acute with respect to the metrics $g$ and $\delta$. 
    Therefore, by restricting (\ref{Eq: FBMH equation n-dim (1)}) in $\Omega^2(x')$ and applying the arguments in Lemma \ref{Lem: C1 alpha with large alpha}, we have $v_{x'}$ satisfies (\ref{Eq: C 1 alpha with large alpha})(\ref{Eq: FBMH equation 2-dim (1)}) with coefficients $a_{ij}(0,0,x')\approx \delta_{ij}$, $\beta_{\pm}(0,0,x') \approx \nu_{\pm, \delta}$, and $f_{\pm}(0,0,x')=0$ (since $Dv_{x'}(0)=0$). 
    Thus, the proof of Theorem \ref{Thm: C2 alpha regularity in 2-dim acute angle domain} would carry over for $v_{x'}$ in $\Omega^2_R(x')$, which indicates $u$ is $C^{2,\alpha_0}$ in the $2$-dimensional slice $\Omega^2_R(x')$ for some $\alpha_0\in(0,1)$. 

    Moreover, by varying $x'$, the coefficients in (\ref{Eq: FBMH equation 2-dim (2)}) with respect to $v_{x'}\llcorner\Omega^2(x')$ vary continuously in $C^{0,\gamma}$ or $C^{1,\gamma}$: 
    \[
	   \left\{\begin{array}{ll}
	   	\triangle_{\delta} v_0 = F_0 &\mbox{in $\interior(\Omega_{R}^2(0))$}, \\
	   	\nu_{\pm,\delta} \cdot Dv_0 = F_{\pm,0} &  \mbox{on $\bd^{\pm}\Omega_{R}^2(0)$},
	   \end{array}\right.
	   \rightarrow~
	   \left\{\begin{array}{ll}
		a_{ij}(0,0,x') (v_{x'})_{ij} = F_{x'} &  \mbox{in $\interior(\Omega_{R}^2(x'))$}, \\
		\beta_{\pm}(0,0,x') \cdot Dv_{x'} = F_{\pm,x'} &  \mbox{on $\bd^{\pm}\Omega_{R}^2(x')$},
	   \end{array}\right.
    \]
    Therefore, by the elliptic estimates in the proof of \cite[Theorem 6.4.2.5]{mazya2004schauder}, we have $v_{x'}\llcorner\Omega^2_R(x') - C_{m,x'}\widetilde{\Im}_{m,x'}$ is varying continuously in $C^{2,\sigma}$ with respect to $x'$ for some $\sigma\in (0,1)$, which implies $u$ is $C^{2,\alpha_0}$ in a small neighborhood of $0\in \Omega$ for some $\alpha_0\in (0,\min\{\gamma,\sigma\})$. 
\end{proof}

\subsection{In right angle wedge domains}

We now consider the case that $\theta(p)=\pi/2$. 
Then by (\ref{ddag}), $\theta(q)\equiv \pi/2$ for $q\in\bd^EM$ near $p$. 
Instead of using the local model as before, we take the local coordinates $(x_1,\dots,x_n)$ centered at $p$ given by Lemma \ref{Lem: local chart}, and write the FBMH $\Sigma$ as the graph of a function $u\in  C^{1,\alpha}(\Omega_{2R})\cap C^{\infty}(\Omega_{2R}\setminus\bd^E\Omega)$, where 
\begin{equation}\label{Eq: regularity in right wedge: wedge domain}
    \Omega:=\{ x_1\geq 0, x_2\geq 0, x_{n+1}=0 \} \qquad{\rm and}\qquad \Omega_r:=\Omega\cap\mB^n_r(0),~\forall r>0.
\end{equation}
Similar to the proof of Theorem \ref{Thm: C2 alpha regularity in n-dim acute angle domain}, we have $u_k\in C^{1,\alpha}(\Omega_{(2-\delta)R})$ for all $k\geq 3$, and the equations of $u$ can be written in the form of (\ref{Eq: FBMH equation n-dim (1)}) by: 
\begin{equation}\label{Eq: FBMH equation n dim right domain (1)}
	\left\{\begin{array}{ll}
		\tilde{L} u := \sum_{i,j=1}^2 a_{ij}u_{ij} = \tilde{f} &\qquad  \mbox{in $\interior(\Omega_{R})$}, \\
		\tilde{B} u := \sum_{i=1}^2\beta_{\pm}^i u_i = \tilde{f}_{\pm} &\qquad  \mbox{on $\bd^{\pm}\Omega_{R}$},
	\end{array}\right.
\end{equation}
where $\bd^+\Omega=\{x_1=0\}$, $\bd^-\Omega=\{x_2=0\}$, and 
\begin{itemize}
	\item $\tilde{f} := -\sum_{i~{\rm or}~j \geq 3} a_{ij}u_{ij} - b_i u_i + f \in C^{0,\alpha}(\Omega_{R})\cap C^{\infty}(\Omega_{R}\setminus\bd^E\Omega)$;
	\item  $\tilde{f}_{\pm}:=-\sum_{i=3}^n\beta_{\pm}^iu_i + f_{\pm}\in C^{1,\alpha}(\bd^{\pm}\Omega_R)\cap C^{\infty}(\bd^{\pm}\Omega_R\setminus\bd^E\Omega)$;
	\item $\beta_{+}^i = g^{1i}(0,x_2,\dots, u(0,x_2,x'))$ and $\beta_{-}^i = g^{2i}(x_1,0,\dots, u(x_1,0,x'))$;
	\item $f_{+} = g^{1,n+1}(0,x_2,\dots, u(0,x_2,x'))$ and $f_{-} = g^{2,n+1}(x_1,0,\dots, u(x_1,0,x'))$.
\end{itemize}
Note the boundary conditions can be computed as in Appendix \ref{Sec: maximum principle for varifolds} (II)(III). 

On $\bd^-\Omega=\{x_2=0\}$, the boundary condition implies $u_2(x_1, 0, x') = 0$, and thus
\[ u_{21}(0,0,x') =0 .\]
On $\bd^+\Omega=\{x_1=0\}$, the boundary condition together with Lemma \ref{Lem: local chart}(ii) imply 
\[u_1(0,0,x') = \tilde{f}_+(0,0,x') = 0,\]
and thus by (\ref{Eq: FBMH equation n dim right domain (1)}) and $\theta\equiv \pi/2$, 
\begin{eqnarray}\label{Eq: first order agree}
	\frac{\bd \tilde{f}_+}{\bd x_2}(0,0,x') &=& \sum_{i=1}^2 \frac{\bd \beta_{+}^i u_i}{\bd x_2}(0,0,x') 
	\\&=& \frac{\bd g^{11}}{\bd x_2}(0,0,x',u(0,0,x')) \cdot u_1(0,0,x') + \lim_{x_2\to 0_+}u_{12}(0,x_2,x')\nonumber
	\\&=& \lim_{x_2\to 0_+}u_{12}(0,x_2,x').\nonumber
\end{eqnarray}
As a necessary condition for $u\in C^2(\Omega_R)$, we shall have $\lim_{x_2\to 0_+}u_{12}(0,x_2,x')= 0$ and 
\begin{eqnarray}\label{Eq: necessary condition}
	0 = \frac{\bd \tilde{f}_+}{\bd x_2}(0,0,x') = \left( - \sum_{i\geq 3} \frac{\bd g^{1i}}{\bd x_2} u_i + \frac{\bd g^{1,n+1}}{\bd x_2}\right) (0,0,x',u(0,0,x')).
\end{eqnarray}
Noted that this necessary condition (\ref{Eq: necessary condition}) is satisfied by the assumption (\ref{ddag}) and Lemma \ref{Lem: local chart}(iv). 
Moreover, it turns out (\ref{Eq: necessary condition}) is also a sufficient condition to show $u$ is $C^{2,\alpha_0}$.

\begin{theorem}\label{Thm: C2 alpha regularity in n-dim right domain}
    Given $R>0$, $n\geq 2$ and $\alpha\in (0,1]$, let $\Omega$ be a right angle wedge domain given as in (\ref{Eq: regularity in right wedge: wedge domain}) so that the metric $g$ in $\Omega_R\times (-R,R)$ satisfies the properties in Lemma \ref{Lem: local chart}. 
    Suppose $u\in C^{1,\alpha}(\Omega_{R})\cap C^{\infty}(\Omega_{R}\setminus\bd^E\Omega)$ is a function satisfying (\ref{Eq: FBMH equation n dim right domain (1)})(\ref{Eq: necessary condition}). 
	Then $u\in C^{2,\alpha_0}(\Omega_{r})$ for some $\alpha_0\in (0,1)$ and $r\in (0, R)$. 
\end{theorem}
\begin{proof}
    Suppose we have $u$ with (\ref{Eq: FBMH equation n dim right domain (1)})(\ref{Eq: necessary condition}). 
    Then we rewrite the equations of $u$ by 
    \begin{equation}\label{Eq: FBMH equation n-dim right domain (2)}
	   \left\{\begin{array}{ll}
		A_{ij}u_{ij} := a_{ij}(0,0,x') u_{ij}(x_1,x_2,x') = F(x_1,x_2,x') &\qquad  \mbox{in $\interior(\Omega_{R})$}, \\
		B_{\pm}^iu_i := \beta_{\pm}^i(0,0,x') u_i(x_1,x_2,x') = F_{\pm}(x_1,x_2,x') &\qquad  \mbox{on $\bd^{\pm}\Omega_{R}$},
	   \end{array}\right.
    \end{equation}
    where $F := (a_{ij}(0,0,x')-a_{ij})u_{ij} + \tilde{f} $, $F_{\pm} = (\beta_{\pm}^i(0,0,x')-\beta_{\pm}^i) u_i + \tilde{f}_{\pm}$, and $i,j\in\{1,2\}$. 
    Similar to the proof of Theorem \ref{Thm: C2 alpha regularity in 2-dim acute angle domain}, \ref{Thm: C2 alpha regularity in n-dim acute angle domain}, we have 
    \[ F\in C^{0,\gamma}(\Omega_{r})\quad \mbox{and}\quad F_{\pm} \in C^{1,\gamma}(\bd^{\pm}\Omega_{r}),\]
    for some $\gamma\in (0,1)$ and $r\in (0,R)$. 

    By Lemma \ref{Lem: local chart}(ii), we have $(\beta_-^1,\beta_-^2) = (0,-1)$ and $\tilde{f}_- = 0$ on $\bd^-\Omega=\{x_2=0\}$. 
    Thus, $u_2(x_1,0,x')=F_-=0$, $u_{21}(x_1,0,x')=u_{2l}(x_1,0,x')=0$ ($\forall l\geq 3$), and the even extension 
    \[ u(x_1, x_2,x') :=  u(x_1, -x_2,x') \qquad \forall x_2<0, x_1\geq 0,\]
    gives a $C^{1,\alpha}$ function $u$ in $\mB_r^+(0):=\{x\in \mB^n_r(0): x_1\geq 0\}$, which is also $C^2$ in $\interior(\mB_r^+(0))$.
    Next, extend the coefficients to $\{x: x_1\geq 0, x_2<0\}$ by 
    \[ F(x_1,x_2,x') := F(x_1, - x_2,x'), \qquad F_{+}(0,x_2,x') := F_+(0,-x_2,x'); \]
    \[
    	A_{ij}(x_1,x_2,x') :=  \left\{ 
    	\begin{array}{ll}
    		{a_{ij}(0,0,x')}, & {\mbox{if $\{i,j\}$ contains even number of $2$},}\\
    		{- a_{ij}(0,0,x')}, & {\mbox{if $\{i,j\}$ contains odd number of $2$};}
    	\end{array}
    	\right.
    \]
    \[ B_{+}^i(0,x_2,x'):= \left\{ 
    	\begin{array}{ll}
    		{B_{+}^i(0,0,x')}, & {\mbox{if $i=1$},}\\
    		{- B_{+}^i(0,0,x')}, & {\mbox{if $i=2$}.}
    	\end{array}
    	\right.
    \]
    By even extending, $F\in C^{0,\gamma}(\mB_r^+(0))$, $A_{11},A_{22}\in C^{0,\alpha}(\mB_r^+(0))$, and $B_+^1\in C^{1,\alpha}(\bd\mB_r^+(0)\cap \{x_1=0\})$.
    Additionally, by $u_1(0,0,x')=u_2(x_1,0,x')=0$ and (\ref{Eq: necessary condition}),
    \begin{eqnarray*}
    	 \frac{\bd F_+}{\bd x_2}(0,0,x') &=&  \frac{\bd \tilde{f}_+}{\bd x_2}(0,0,x') - \sum_{i=1}^2 \frac{\bd [\beta_{\pm}^i(0,0,x')-\beta_{\pm}^i)u_{i}]}{\bd x_2}(0,0,x')
    	 \\&=&\frac{\bd \tilde{f}_+}{\bd x_2}(0,0,x') - \sum_{i=1}^2 \Big[\frac{\bd \beta_{+}^i }{\bd x_2}u_i \Big](0,0,x') 
    	 = \frac{\bd \tilde{f}_+}{\bd x_2}(0,0,x') =0,
    \end{eqnarray*}
    which implies the even extended function $F_+\in C^{1,\gamma}(\bd\mB_r^+(0)\cap \{x_1=0\})$.
    By $u_2(x_1,0,x')=0$ and Lemma \ref{Lem: local chart}(ii), a direct compute shows $A_{12}(0,0,x') = A_{21}(0,0,x') = B_+^2(0,0,x') = 0$, which implies the odd extended functions $A_{12}=A_{21}\in C^{0,\alpha}(\mB_r^+(0))$, and $B_+^2\in C^{1,\alpha}(\bd\mB_r^+(0)\cap \{x_1=0\})$.
    Together, we see the extended $u$ is a solution of a second order elliptic PDE with oblique boundary condition in a half-space with nice coefficients, which implies $u\in C^{2,\alpha_0}(\mB_r^+(0))$ by \cite[Theorem 4.10]{lieberman2013oblique} for some $\alpha_0\in (0,1)$. 
\end{proof}

\begin{remark}
    For the case that $n=2$ and $\theta(p)=\pi/2$, we only need the curvature condition in (\ref{ddag}) is valid at a single point $p\in\bd^E\Sigma$ to show the $C^{2,\alpha_0}$ regularity of $\Sigma$ near $p$. 
\end{remark}


\bibliographystyle{abbrv}

\providecommand{\bysame}{\leavevmode\hbox to3em{\hrulefill}\thinspace}
\providecommand{\MR}{\relax\ifhmode\unskip\space\fi MR }
\providecommand{\MRhref}[2]{%
  \href{http://www.ams.org/mathscinet-getitem?mr=#1}{#2}}
\providecommand{\href}[2]{#2}

\bibliography{reference.bib}   

%
%






\end{document}